\documentclass[12pt]{amsart}
 \usepackage{hyperref}
\usepackage{amsthm,amssymb}
\usepackage{enumerate}
\usepackage[arrow]{xy}
\usepackage{comment}
\usepackage{epsfig}
\usepackage{diagbox}

\usepackage{amssymb}
\usepackage{latexsym}
\usepackage{cite}
\usepackage{epsf}
\usepackage{graphics}

\usepackage{color}

\usepackage[margin=1.0in]{geometry}
\usepackage{graphicx,color}
\usepackage{chngcntr} 
\usepackage{pstricks}
\usepackage{pst-node}
\usepackage{pst-tree}
\usepackage{pgffor}

\newtheorem{theorem}{Theorem}[section]
\newtheorem{thm}[theorem]{Theorem}
\newtheorem{lemma}[theorem]{Lemma}
\newtheorem{lem}[theorem]{Lemma}
\newtheorem{proposition}[theorem]{Proposition}
\newtheorem{prop}[theorem]{Proposition}
\newtheorem{corollary}[theorem]{Corollary}
\newtheorem{cor}[theorem]{Corollary}
\newtheorem{conjecture}[theorem]{Conjecture}

\theoremstyle{remark}

\theoremstyle{definition}

\newtheorem{remark}[theorem]{Remark}
\newtheorem{example}[theorem]{Example}
\newtheorem{definition}[theorem]{Definition}

\def\H{\mathcal{H}}

\def\Ccal{\mathcal{C}} 
\def\Scal{\mathcal{S}}

\def\x{\mathfrak{x}}
\def\Cc{\mathfrak{C}} 
\def\S{\mathfrak{S}}
\def\hh{\mathfrak{h}}

\def\Zz{\mathcal{Z}}
\def\D{\mathcal{D}}

\def\R{\mathbb{R}}

\def\Z{\mathbb{Z}}
\def\Q{\mathbb{Q}}

\def\<{{\langle}}
\def\>{{\rangle}}

\def\det{{ \operatorname{det}}}

\def\e{{\epsilon}}

\def\VertL{{\operatorname{Vert}(\affL)}}
\def\map{\u}
\def\coeff{\theta}
\def\eig{\v}
\def\mutmat{\omega}
\def\coxmat{\mathbf{C}}
\def\bwmat{A}

\def\speed{ \operatorname{SPEED}}
\def\summ{ \operatorname{SUM}}

\newcommand{\col}[2]{\begin{pmatrix}
                         #1\\#2
                        \end{pmatrix}}

\def\rec{J}

\def\Vert{{ \operatorname{Vert}}}
\def\VertQ{{ \Vert(Q)}}
\def\x{{ \mathbf{x}}}

\def\affinite{{affine $\boxtimes$ finite }}
\def\affaff{{affine $\boxtimes$ affine }}

\def\affA{{\hat A}}
\def\affD{{\hat D}}
\def\affE{{\hat E}}
\def\affL{{\hat \Lambda}}

\def\l{{ \lambda}}
\def\Ttr{ \mathfrak{t}^\l}
\def\Ytr{ y^\l}

\def\i{{\mathbf{i}}}
\def\op{{ \operatorname{op}}}

\newcommand{\plusOne}[1]{%
\number\numexpr#1+1\relax%
}
\newcommand{\bl}[1]{[fillstyle=solid,fillcolor=lightgray,mnode=circle]#1}
\newcommand{\wh}[1]{[fillstyle=solid,fillcolor=white,mnode=circle]#1}

\newcommand{\ra}[2]{\ncline[linecolor=red]{#1}{#2}}
\newcommand{\ba}[2]{\ncline[linecolor=blue]{#1}{#2}}
\newcommand{\arr}[2]{\ncline[linecolor=black]{#1}{#2}}

\counterwithin{section}{part}

\title{Quivers with subadditive labelings: classification and integrability}

\numberwithin{equation}{section}

\begin{document}

\author{Pavel Galashin}
\address{\hspace{-.3in} Department of Mathematics, Massachusetts Institute of Technology,
Cambridge, MA 02139, USA}
\email{galashin@mit.edu}

\author{Pavlo Pylyavskyy}
\address{\hspace{-.3in} Department of Mathematics, University of Minnesota,
Minneapolis, MN 55414, USA}
\email{ppylyavs@umn.edu}

\date{\today}

\thanks{P.~P. was partially supported by NSF grants  DMS-1148634, DMS-1351590, and Sloan Fellowship.}



\begin{abstract}

Strictly subadditive, subadditive and weakly subadditive labelings of quivers were introduced by the second author, generalizing Vinberg's definition for undirected graphs. In our previous work we have shown that quivers with strictly subadditive labelings are exactly the quivers exhibiting 
Zamolodchikov periodicity. In this paper, we classify all quivers with subadditive labelings. We conjecture them to exhibit a certain form of integrability, namely, as the $T$-system dynamics proceeds, the values at each vertex satisfy a linear recurrence. Conversely, 
we show that every quiver integrable in this sense is necessarily one of the $19$ items in our classification. For the quivers of type $\hat A \otimes A$ we express the coefficients of the recurrences 
in terms of the partition functions for domino tilings 
of a cylinder, called \emph{Goncharov-Kenyon Hamiltonians}.  
We also consider tropical $T$-systems 
of type $\hat A \otimes A$ and explain how affine slices exhibit solitonic behavior, i.e. soliton resolution and speed conservation. Throughout, we conjecture how the results in the paper are expected to generalize from $\hat A \otimes  A$ to all other quivers in our classification.
\end{abstract}

\maketitle

\setcounter{tocdepth}{1}
\tableofcontents

\section*{Introduction}

 A \emph{quiver} $Q$ is a directed graph without $1$-cycles (i.e. loops) and directed $2$-cycles. For a vertex $v$ of a quiver, one can define a certain operation called \emph{a mutation}, which produces a new quiver denoted $\mu_v(Q)$ (see Definition~\ref{dfn:mutations}). We say that a quiver is bipartite if its underlying graph is bipartite, in which case we say that a map $ \e:\VertQ\to\{0,1\},\ v\mapsto\e_v$ is a \emph{bipartition} if for every edge $u\to v$ of $Q$ we have $\e_u\neq \e_v$. Here $\VertQ$ is the set of vertices of $Q$.
 
 It is clear from Definition~\ref{dfn:mutations} that $\mu_u$ and $\mu_v$ commute if $u,v$ are not connected by an edge in $Q$. Therefore, we can define 
 \[\mu_\circ=\prod_{u:\e_u=0} \mu_u;\quad \mu_\bullet=\prod_{v:\e_v=1} \mu_v.\]
 We say that $Q$ is \emph{recurrent} if $\mu_\circ(Q)=\mu_\bullet(Q)=Q^{\op}$ where $Q^\op$ is the same quiver as $Q$ but with all the arrows reversed.

Let $Q$ be a bipartite recurrent quiver. Denote $\x:=\{x_v\}_{v\in\VertQ}$ to be the set of indeterminates, one for each vertex of $Q$, and let $\Q(\x)$ be the field of rational functions in these variables. The \emph{$T$-system} associated with $Q$ is a family $T_v(t)$ of elements of $\Q(\x)$ satisfying the following relations for all $v\in\VertQ$ and all $t\in\Z$ :
\begin{equation*}
T_v(t+1)T_v(t-1)=\prod_{u\to v} T_u(t)+\prod_{v\to w} T_w(t). 
\end{equation*}
 Here the products are taken over all \emph{arrows} connecting the two vertices. 

It is clear that the parity of $t+\e_v$ in all of the terms is the same, so the $T$-system associated with $Q$ splits into two completely independent ones. Without loss of generality we may consider only one of them. From now on we assume that the $T$-system is defined only for $t\in\Z$ and $v\in\VertQ$ satisfying 
\[t+\e_v\equiv 0\pmod 2.\] 

The $T$-system is set to the following initial conditions: 
\[T_v(\e_v)=x_v\]
for all $v\in\VertQ$.

Let us say that the $T$-system associated with a recurrent quiver $Q$ is {\it {integrable}} if for every vertex $v\in\VertQ$, there exists an integer $N$ and elements $\rec_0,\rec_1,\dots,\rec_{N}\in \Q(\x)$ satisfying $\rec_0,\rec_N\neq 0$ and 
\[\sum_{j=0}^{N} \rec_j T_v(t+2j)=0\] 
for all $t\in\Z$ with $t+\e_v$ even. We also refer to a recurrent quiver $Q$ as \emph{Zamolodchikov integrable}
if the associated $T$-system is integrable. If the recurrence has the form $T_v(t+2N)=T_v(t)$ for all $t\in\Z$ and $v\in\VertQ$, then we call $Q$ \emph{Zamolodchikov periodic}.

Just as in the periodic case, we call a quiver \emph{Zamolodchikov integrable} when the {\underline{bipartite}} $T$-system is integrable. More general notions of $T$-systems can be found in \cite{Na}.

Zamolodchikov periodicity for the case when $Q$ is a tensor product of two finite $ADE$ Dynkin diagrams has been studied extensively (see \cite{Z,RVT,KN,KNS,FZy,FS,GT,Vo,Sz}) and was proven in full generality in \cite{K} and later in  \cite{IIKKN1, IIKKN2}, where tropical $Y$-systems played a major role. By analyzing tropical $T$-systems  (see Part~\ref{part:tropical}), we have classified Zamolodchikov periodic quivers in~\cite{GP}, where we showed that these are exactly the quivers admitting a strictly subadditive labeling (Definition~\ref{dfn:subadditive}). 

Besides thermodynamic Bethe ansatz \cite{Z}, $T$-systems and $Y$-systems arise naturally in a lot of different contexts in physics and representation theory, e.g. \cite{KNS,KR,R,OW,FR,Kni,N}, see \cite{KNSi} for a survey.

Assem, Reutenauer and Smith~\cite{ReuA} showed that the affine Dynkin diagrams of types $\affA$ and $\affD$ are Zamolodchikov integrable, and later Keller and Scherotzke~\cite{KS} extended this result to all affine Dynkin diagrams. Conversely, it was shown in~\cite{ReuA} that if every vertex of a Zamolodchikov integrable quiver $Q$ is either a source or a sink then $Q$ is necessarily an affine Dynkin diagram. 

In Sections~\ref{sec:implies} and~\ref{subsec:implies_tropical} we prove (Theorem~\ref{thm:implies}) that if a bipartite recurrent quiver is Zamolodchikov integrable then it admits a subadditive labeling (see Definition~\ref{dfn:subadditive}). 

We then classify (Theorem~\ref{thm:class}) quivers that admit subadditive labelings in Part~\ref{part:classification}. We conjecture all of them to be Zamolodchikov integrable (Conjecture~\ref{conj:linearizable}). 

When $Q$ is a tensor product (see Definition~\ref{dfn:tensor_product}) of type $\affA\otimes A$, it was shown in~\cite{P} that $Q$ is Zamolodchikov integrable. In Sections~\ref{sec:variables}-\ref{sec:pleth}, we express the recurrence coefficients $\rec_1,\dots,\rec_N$ for the vertices of such $Q$ in terms of partition functions of domino tilings on the cylinder, called~\emph{Goncharov-Kenyon Hamiltonians}. See Theorem~\ref{thm:rec} and Corollary~\ref{cor:pleth}. In Section~\ref{sec:laur} we show that the above Goncharov-Kenyon Hamiltonians belong to the upper cluster algebra. 


In Part~\ref{part:tropical}, we analyze the tropical $T$-system associated with quivers admitting a subadditive labeling. We show (Corollary~\ref{cor:speeds}) that when $t\gg0$ or $t\ll0$, every affine slice of the tropical $T$-system moves with some constant speed. We explain how this can be seen as \emph{soliton resolution} in Section~\ref{sec:resolution}, and then we proceed to show \emph{speed conservation} in Section~\ref{sec:conservation}: for the quiver of type $\affA_{2n-1}\otimes A_m$, the speeds of the solitons at $t\gg0$ are equal to the speeds of solitons at $t\ll0$ after one applies a diagram automorphism to $A_m$. See Example~\ref{example:evolution} for an illustration of these solitonic phenomena. 

Finally, in Sections~\ref{sec:geometric_conjectures} and~\ref{sec:tropical_conjectures} we conjecture most of our results for all other quivers in our classification.

{\Large\part{Zamolodchikov integrable quivers.}}

\section{Preliminaries}

\subsection{Bigraphs}

In \cite{S} Stembridge studies admissible $W$-graphs for the case when $W=I(p)\times I(q)$ is a direct product of two dihedral groups. These $W$-graphs encode the structure of representations of Iwahori-Hecke algebras, and were first introduced by Kazhdan and Lusztig in \cite{KL}.
The following definitions are adapted from \cite{S} with slight modifications. A \emph{bigraph}  is an ordered pair of simple (undirected) graphs $(\Gamma, \Delta)$ which share a common set of vertices $V:=\Vert(\Gamma)=\Vert(\Delta)$ and do not share edges. A bigraph is called \emph{bipartite} if there is a map $\e:V\to\{0,1\}$ such that for every edge $(u,v)$ of $\Gamma$ or of $\Delta$ we have $\e_u\neq \e_v$. 

There is a simple one-to-one correspondence between bipartite quivers and bipartite bigraphs. Namely, to each bipartite quiver $Q$ with a bipartition $\e:\VertQ\to\{0,1\}$ we  
associate a bigraph $G(Q)=(\Gamma(Q),\Delta(Q))$ on the same set of vertices defined as follows:
\begin{itemize}
 \item $\Gamma(Q)$ contains an (undirected) edge $(u,v)$ if and only if $Q$ contains a directed edge $u\to v$ with $\e_u=0,\e_v=1$;
 \item $\Delta(Q)$ contains an (undirected) edge $(u,v)$ if and only if $Q$ contains a directed edge $u\to v$ with $\e_u=1,\e_v=0$. 
\end{itemize}
Similarly, we can direct the edges of any given bipartite bigaph $G$ to get a bipartite quiver $Q(G)$.

It is convenient to think of $(\Gamma, \Delta)$ as of a single graph with edges of two colors: red for the edges of $\Gamma$ and blue for the edges of $\Delta$. 

\begin{definition}\label{dfn:tensor_product}
 Let $S$ and $T$ be two bipartite undirected graphs. Then their \emph{tensor product} $S\otimes T$ is a bipartite bigraph $G=(\Gamma,\Delta)$ with vertex set $\Vert(S)\times \Vert(T)$ and the following edge sets:
 \begin{itemize}
  \item for each edge $\{u,u'\}\in S$ and each vertex $v\in T$ there is an edge between $(u,v)$ and $(u',v)$ in $\Gamma$;
  \item for each vertex $u\in S$ and each edge $\{v,v'\}\in T$ there is an edge between $(u,v)$ and $(u,v')$ in $\Delta$;
 \end{itemize}
 An example of a tensor product is given in Figure~\ref{fig:tensor_product}.
\end{definition}

\begin{figure}
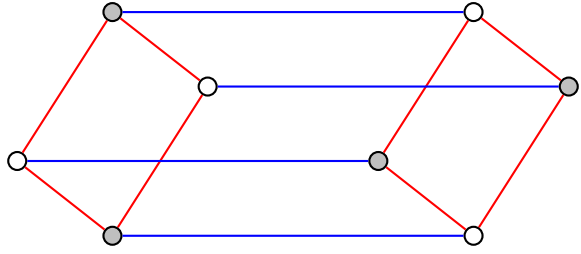

{$
\psmatrix[colsep=1cm,rowsep=0.5cm,mnode=circle]
            & \bl       &           &           &           & \wh\\
            &           & \wh       &           &           &            & \bl\\
\wh         &           &           &           & \bl\\
            & \bl       &           &           &            & \wh
\psset{arrows=-,arrowscale=2}
\ra{1,2}{3,1}
\ra{4,2}{3,1}
\ra{4,2}{2,3}
\ra{1,2}{2,3}
\ra{1,6}{3,5}
\ra{4,6}{3,5}
\ra{4,6}{2,7}
\ra{1,6}{2,7}
\ba{1,2}{1,6}
\ba{4,2}{4,6}
\ba{2,3}{2,7}
\ba{3,1}{3,5}
\endpsmatrix $}
 \caption{\label{fig:tensor_product} A tensor product of a square (type $\affA_3$) and a single edge (type $A_2$).}
\end{figure}

\subsubsection{Reformulation of the dynamics in terms of bigraphs}\label{subsec:bigraph_quiver_translation}

  Let $G=(\Gamma,\Delta)$ be a bipartite bigraph with a vertex set $V$. Then the associated $T$-system for $G$ is defined as follows:
 \begin{eqnarray*}
 T_v(t+1)T_v(t-1)&=&\prod_{(u,v)\in\Gamma} T_u(t)+\prod_{(v,w)\in\Delta} T_w(t);\\
 T_v(\e_v) &=& x_v.
 \end{eqnarray*}
 It is easy to see that this system is equivalent to the corresponding system defined for $Q(G)$ in the Introduction.

\subsection{Finite and affine $ADE$ Dynkin diagrams and their Coxeter numbers}\label{sec:ADEDynkin} 
By a \emph{finite $ADE$ Dynkin diagram} we mean a Dynkin diagram of type $A_n,D_n,E_6,E_7,$ or $E_8$. An \emph{affine $ADE$ Dynkin diagram} is a Dynkin diagram of type $\affA_n,\affD_n,\affE_6,\affE_7,$ or $\affE_8$, see Figure~\ref{figure:affine_ADE}.

The following characterization of finite and affine $ADE$ Dynkin diagrams is due to Vinberg \cite{V}:
\begin{theorem}\label{thm:Vinberg}
 Let $G=(V,E)$ be an undirected graph with possibly multiple edges. Then:
 \begin{itemize}
  \item $G$ is a finite $ADE$ Dynkin diagram if and only if there exists a map $\nu:V\to\R_{>0}$ such that for all $v\in V$, 
 \[2\nu(v)>\sum_{(u,v)\in E} \nu(u).\]
 \item$G$ is an affine $ADE$ Dynkin diagram if and only if there exists a map $\nu:V\to\R_{>0}$ such that for all $v\in V$, 
 \begin{equation}\label{eq:Vinberg_affine} 
  2\nu(v)=\sum_{(u,v)\in E} \nu(u).
 \end{equation}
 \end{itemize}
 The values of $\nu$ satisfying (\ref{eq:Vinberg_affine}) are given in Figure~\ref{figure:affine_ADE}.
\end{theorem}

\begin{figure}
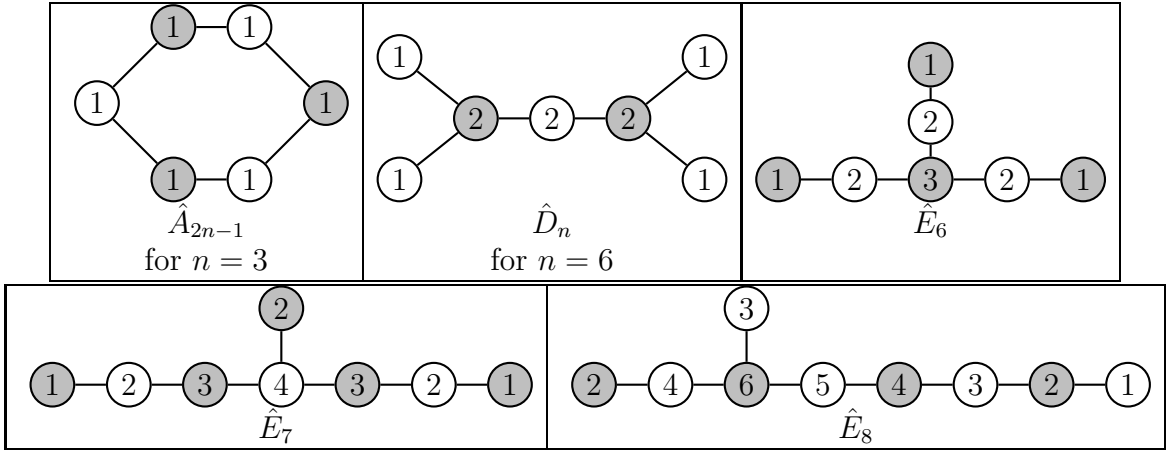

 \begin{tabular}{|c|c|c|} \hline
 {
$
\psmatrix[colsep=0.4cm,rowsep=0.4cm,mnode=circle]
            & \bl1      & \wh1\\
\wh1        &           &           & \bl1\\
            & \bl1      & \wh1
\psset{arrows=-,arrowscale=2}
\arr{2,1}{1,2}
\arr{2,1}{3,2}
\arr{2,4}{1,3}
\arr{2,4}{3,3}
\arr{1,2}{1,3}
\arr{3,2}{3,3}
\endpsmatrix $} 
&
{$
\psmatrix[colsep=0.4cm,rowsep=0.2cm,mnode=circle]
   \wh1     &           &           &           &  \wh1     \\
            &  \bl2     & \wh2      &   \bl2                \\
   \wh1     &           &           &           &  \wh1     
\psset{arrows=-,arrowscale=2}
\arr{2,2}{1,1}
\arr{2,2}{3,1}
\arr{2,2}{2,3}
\arr{2,4}{2,3}
\arr{2,4}{1,5}
\arr{2,4}{3,5}
\endpsmatrix $} 
&
{$
\psmatrix[colsep=0.4cm,rowsep=0.15cm,mnode=circle]
            &           &     \bl1                          \\
            &           &     \wh2                          \\
\bl1        &\wh2       &\bl3       &\wh2       &\bl1       
\arr{3,1}{3,2}
\arr{3,3}{3,2}
\arr{3,3}{3,4}
\arr{3,5}{3,4}
\arr{3,3}{2,3}
\arr{1,3}{2,3}
\endpsmatrix $} 
\\
$\affA_{2n-1}$ &$\affD_{n}$  &$\affE_6$ \\
for $n=3$      & for $n=6$ & \\\hline
\end{tabular}
\begin{tabular}{|c|c|}\hline
{
$
\psmatrix[colsep=0.4cm,rowsep=0.4cm,mnode=circle]
            &           &           & \bl2                                          \\
\bl1        &\wh2       &\bl3       &\wh4       &   \bl3    &\wh2       &\bl1       
\psset{arrows=-,arrowscale=2}
\arr{2,1}{2,2}
\arr{2,3}{2,2}
\arr{2,3}{2,4}
\arr{2,5}{2,4}
\arr{2,5}{2,6}
\arr{2,7}{2,6}
\arr{2,4}{1,4}
\endpsmatrix $} 
&
{
$
\psmatrix[colsep=0.4cm,rowsep=0.4cm,mnode=circle]
            &           & \wh3                                          \\
\bl2        &\wh4       &\bl6       &\wh5       &   \bl4    &\wh3       &\bl2 & \wh1   
\psset{arrows=-,arrowscale=2}
\arr{2,1}{2,2}
\arr{2,3}{2,2}
\arr{2,3}{2,4}
\arr{2,5}{2,4}
\arr{2,5}{2,6}
\arr{2,7}{2,6}
\arr{2,7}{2,8}
\arr{2,3}{1,3}
\endpsmatrix $}
\\
$\affE_{7}$ &$\affE_{8}$ \\\hline
 \end{tabular}
\caption{\label{figure:affine_ADE}Affine $ADE$ Dynkin diagrams and their additive labelings.}
\end{figure}

For each finite $ADE$ Dynkin diagram $\Lambda$ there is an associated integer $h(\Lambda)$ called \emph{Coxeter number}. We list Coxeter numbers of finite $ADE$ Dynkin diagrams in Figure~\ref{table:coxeter}. If $\affL$ is an affine Dynkin diagram, we set $h(\affL)=\infty$. 

\begin{table}
\centering
\begin{tabular}{|c|c|c|c|c|c|}\hline
 $\Lambda$   & $A_n$  & $D_m$  & $E_6$  & $E_7$ & $E_8$\\\hline
$h(\Lambda)$ & $n+1$  & $2m-2$ & $12$   & $18$  & $30$\\\hline
\end{tabular}
\caption{\label{table:coxeter}Coxeter numbers of finite $ADE$ Dynkin diagrams}
\end{table}

It is well-known that the Coxeter number has a nice interpretation in terms of eigenvalues of the adjacency matrix:
\begin{proposition}\label{prop:coxeter_eigenvalues}
 \begin{itemize}
  \item If $\Lambda$ is a finite $ADE$ Dynkin diagram then the dominant eigenvalue of its adjacency matrix equals $2\cos(\pi/h(\Lambda))$;
  \item if $\affL$ is an affine $ADE$ Dynkin diagram then the dominant eigenvalue of its adjacency matrix equals $2$.
 \end{itemize} \qed
\end{proposition}
In particular, the second claim justifies setting $h(\affL):=\infty$. 

\subsection{Subadditive labelings}\label{subsec:subadditive}

Let $G=(\Gamma,\Delta)$ be a bipartite bigraph on vertex set $V$. A {\it {labeling}} of its vertices is a function $\nu: V\rightarrow \mathbb R_{>0}, \;$ which assigns to each vertex $v$ of $G$ a positive real label $\nu(v)$. 
\begin{definition}\label{dfn:subadditive}
A labeling $\nu:V\rightarrow\R_{>0}$ is called
\begin{itemize}
 \item \emph{strictly subadditive} if for any vertex $v\in V$, 
\[2 \nu(v) > \sum_{(u,v)\in\Gamma} \nu(u), \; \text{ and }  \; 2 \nu(v) > \sum_{(v,w)\in\Delta} \nu(w).\]
 \item \emph{subadditive} if for any vertex $v\in V$, 
\[2 \nu(v) \geq \sum_{(u,v)\in\Gamma} \nu(u), \; \text{ and }  \; 2 \nu(v) > \sum_{(v,w)\in\Delta} \nu(w).\]
 \item \emph{weakly subadditive} if for any vertex $v\in V$, 
\[2 \nu(v) \geq \sum_{(u,v)\in\Gamma} \nu(u), \; \text{ and }  \; 2 \nu(v) \geq \sum_{(v,w)\in\Delta} \nu(w).\]
\end{itemize}
Examples of each type can be found in Figure~\ref{figure:labelings}.
\end{definition}

\begin{figure}
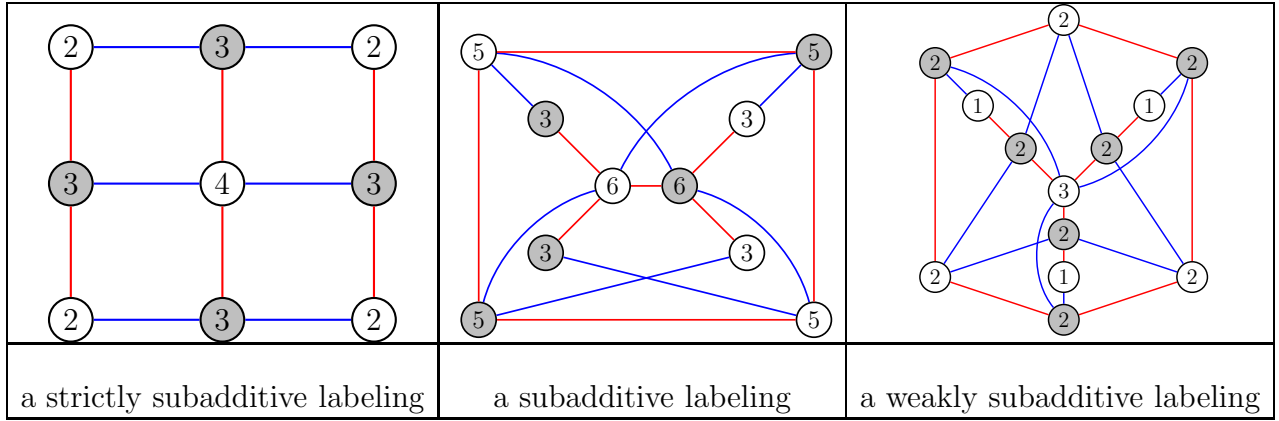

 \begin{tabular}{|c|c|c|}\hline
 {
 $\psmatrix[colsep=1.4cm,rowsep=1.2cm,mnode=circle]
\wh2 & \bl3 & \wh2 \\
\bl3 & \wh4 & \bl3 \\
\wh2 & \bl3 & \wh2 
\psset{arrows=-,arrowscale=2}
\ra{1,1}{2,1}
\ra{3,1}{2,1}
\ra{1,3}{2,3}
\ra{3,3}{2,3}
\ra{2,2}{1,2}
\ra{2,2}{3,2}
\ba{1,2}{1,1}
\ba{1,2}{1,3}
\ba{3,2}{3,1}
\ba{3,2}{3,3}
\ba{2,1}{2,2}
\ba{2,3}{2,2}
\endpsmatrix $
 }
  &
 \scalebox{0.8}{
$
\psmatrix[colsep=0.5cm,rowsep=0.5cm,mnode=circle]
   \wh5      &           &           &           &           &   \bl5     \\
            &   \bl3     &           &           &       \wh3                 \\
            &           &    \wh6    &  \bl6                               \\
            &   \bl3     &           &           &        \wh3                \\
   \bl5      &           &           &           &           &      \wh5     
\psset{arrows=-,arrowscale=2}
\ncarc[arcangle=30,linecolor=blue]{-}{5,1}{3,3}
\ncarc[arcangle=-30,linecolor=blue]{-}{5,6}{3,4}
\ncarc[arcangle=30,linecolor=blue]{-}{1,1}{3,4}
\ncarc[arcangle=-30,linecolor=blue]{-}{1,6}{3,3}
\ra{1,1}{1,6}
\ra{5,6}{1,6}
\ra{5,6}{5,1}
\ra{1,1}{5,1}
\ra{2,2}{3,3}
\ra{4,2}{3,3}
\ra{3,4}{3,3}
\ra{3,4}{2,5}
\ra{3,4}{4,5}
\ba{1,1}{2,2}
\ba{5,6}{4,2}
\ba{5,1}{4,5}
\ba{1,6}{2,5}
\endpsmatrix $}
  &
 \scalebox{0.7}{
$
\psmatrix[colsep=0.2cm,rowsep=0.2cm,mnode=circle]
            &           &           &   \wh2                                        \\
\bl2        &           &           &           &           &           & \bl2      \\
            &   \wh1    &           &           &           &  \wh1                 \\
            &           & \bl2      &           & \bl2                              \\
            &           &           &   \wh3                                        \\
            &           &           &   \bl2                                        \\
\wh2        &           &           &   \wh1    &           &           &  \wh2     \\
            &           &           &    \bl2                                         

\psset{arrows=-,arrowscale=2}
\ncarc[arcangle=30,linecolor=blue]{-}{2,1}{5,4}
\ncarc[arcangle=30,linecolor=blue]{-}{2,7}{5,4}
\ncarc[arcangle=40,linecolor=blue]{-}{8,4}{5,4}
\ra{2,1}{1,4}
\ra{2,7}{1,4}
\ra{2,7}{7,7}
\ra{8,4}{7,7}
\ra{8,4}{7,1}
\ra{2,1}{7,1}
\ra{3,2}{4,3}
\ra{5,4}{4,3}
\ra{5,4}{4,5}
\ra{3,6}{4,5}
\ra{5,4}{6,4}
\ra{7,4}{6,4}
\ba{2,1}{3,2}
\ba{2,7}{3,6}
\ba{8,4}{7,4}
\ba{7,1}{4,3}
\ba{1,4}{4,3}
\ba{1,4}{4,5}
\ba{7,7}{4,5}
\ba{7,7}{6,4}
\ba{7,1}{6,4}
\endpsmatrix $}
\\\hline
 & & \\
  a strictly subadditive labeling & a subadditive labeling & a weakly subadditive labeling \\\hline
 \end{tabular}
\caption{\label{figure:labelings} Different kinds of labelings}
\end{figure}

Strictly subadditive, subadditive and weakly subadditive labelings of quivers have been introduced in \cite{P}. The terminology is motivated by Vinberg's {\it {subadditive labelings}} \cite{V} for non-directed graphs (see Theorem~\ref{thm:Vinberg}).

\subsection{Quivers}\label{subsec:quiver_mutations}

\begin{definition}\label{dfn:mutations}
 For a vertex $v$ of $Q$ one can define the \emph{quiver mutation $\mu_v$ at $v$} as follows:
\begin{enumerate}
 \item\label{step:trans} for each pair of edges $u \rightarrow v$ and $v \rightarrow w$ create an edge $u \rightarrow w$;
 \item\label{step:reverse} reverse the direction of all edges adjacent to $v$;
 \item\label{step:remove} if some directed $2$-cycle is present, remove both of its edges; repeat until there are no more directed $2$-cycles.
\end{enumerate}
Let us denote the resulting quiver $\mu_v(Q)$. See Figure~\ref{figure:mut} for an example of each step.
\end{definition}

\newcommand{\ga}[2]{\ncline[linecolor=orange]{#1}{#2}}
\newcommand{\gaa}[2]{\ncarc[arcangle=30,linecolor=orange]{#1}{#2}}
\newcommand{\arrr}[2]{\ncline[linecolor=orange]{#1}{#2}}

\begin{figure}
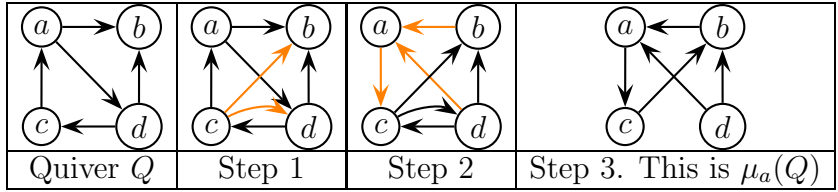

 \begin{tabular}{|c|c|c|c|}\hline
$
\psmatrix[colsep=0.7cm,rowsep=0.7cm,mnode=circle]
\wh{a}         & \wh{b}\\
\wh{c}         & \wh{d}
\psset{arrows=->,arrowscale=2}
\arr{2,1}{1,1}
\arr{1,1}{1,2}
\arr{1,1}{2,2}
\arr{2,2}{1,2}
\arr{2,2}{2,1}
\endpsmatrix $
&
$
\psmatrix[colsep=0.7cm,rowsep=0.7cm,mnode=circle]
\wh{a}         & \wh{b}\\
\wh{c}         & \wh{d}
\psset{arrows=->,arrowscale=2}
\arr{2,1}{1,1}
\arr{1,1}{1,2}
\arr{1,1}{2,2}
\arr{2,2}{1,2}
\arr{2,2}{2,1}
\ga{2,1}{1,2}
\gaa{2,1}{2,2}
\endpsmatrix $
&
$
\psmatrix[colsep=0.7cm,rowsep=0.7cm,mnode=circle]
\wh{a}         & \wh{b}\\
\wh{c}         & \wh{d}
\psset{arrows=->,arrowscale=2}
\arrr{1,1}{2,1}
\arrr{1,2}{1,1}
\arrr{2,2}{1,1}
\arr{2,2}{1,2}
\arr{2,2}{2,1}
\arr{2,1}{1,2}
\ncarc[arcangle=30,linecolor=black]{2,1}{2,2}
\endpsmatrix $
&
$
\psmatrix[colsep=0.7cm,rowsep=0.7cm,mnode=circle]
\wh{a}         & \wh{b}\\
\wh{c}         & \wh{d}
\psset{arrows=->,arrowscale=2}
\arr{1,1}{2,1}
\arr{1,2}{1,1}
\arr{2,2}{1,1}
\arr{2,2}{1,2}
\arr{2,1}{1,2}
\endpsmatrix $
\\\hline
Quiver $Q$ & Step~\ref{step:trans} & Step~\ref{step:reverse}& Step~\ref{step:remove}. This is $\mu_a(Q)$\\\hline
 \end{tabular}
\caption{\label{figure:mut}Mutating a quiver at vertex $a$.}
\end{figure}

Now, let $Q$ be a bipartite quiver. Recall that $\mu_\circ$ (resp., $\mu_\bullet$) is the simultaneous mutation at all white (resp., all black) vertices of $Q$, and that $Q$ is \emph{recurrent} if $\mu_\circ(Q)=\mu_\bullet(Q)=Q^{\operatorname{op}}$.

As we have observed in~\cite{GP}, this property translates nicely into the language of bigraphs:
\begin{corollary}\label{cor:recurrent_commuting}
 A bipartite quiver $Q$ is recurrent if and only if the associated bipartite bigraph $G(Q)$ has commuting adjacency matrices $A_\Gamma,A_\Delta$. 
\end{corollary}

We define three variations of Stembridge's admissible $ADE$ bigraphs (see~\cite{S}):

\begin{definition}\label{dfn:affinite}
Let $G=(\Gamma,\Delta)$ be a bipartite bigraph, and assume that the adjacency $|V|\times|V|$ matrices $A_\Gamma$ and $A_\Delta$ of $\Gamma$ and $\Delta$ commute. In this case we encode the three definitions in Table~\ref{table:def}. For instance, $G$ is an \emph{\affinite $ADE$ bigraph} if each connected component of $\Gamma$ is an affine $ADE$ Dynkin diagram and each connected component of $\Delta$ is a finite $ADE$ Dynkin diagram. We similarly define the notions of \emph{admissible} and \emph{\affaff}$ADE$ bigraphs.
\begin{table}
 \begin{tabular}{|c|c|c|}\hline
  {\bf Notion}                      & {\bf all components of $\Gamma$ are }  & {\bf all components of $\Delta$ are} \\\hline
  admissible $ADE$ bigraph    & finite $ADE$ Dynkin diagrams     & finite $ADE$ Dynkin diagrams  \\\hline
  \affinite  $ADE$ bigraph    & affine $ADE$ Dynkin diagrams     & finite $ADE$ Dynkin diagrams  \\\hline
  \affaff  $ADE$ bigraph    & affine $ADE$ Dynkin diagrams     & affine $ADE$ Dynkin diagrams  \\\hline
 \end{tabular}
 \caption{\label{table:def} Three types of bigraphs}
\end{table}
\end{definition}

The following fact is an easy consequence of \cite[Lemma~4.3]{S}:

\def\v{\mathbf{v}}
\def\u{\mathbf{u}}
\begin{lemma}\label{lemma:eigenvalues}
 Let $G=(\Gamma,\Delta)$ be a bigraph and assume that the adjacency matrices $A_\Gamma,A_\Delta$ commute. Then the dominant eigenvalues of all components of $\Gamma$ are equal to the same value $\lambda_\Gamma$, and the dominant eigenvalues of all components of $\Delta$ are equal to the same value $\lambda_\Delta$. Matrices $A_\Gamma$ and $A_\Delta$ have a common dominant eigenvector $\v$ such that 
 \[A_\Gamma \v=\l_\Gamma\v;\quad A_\Delta \v=\l_\Delta\v.\] \qed
\end{lemma}

\begin{corollary}\label{cor:common_coxeter}
 Let $G=(\Gamma,\Delta)$ be a bigraph and assume that the adjacency matrices $A_\Gamma,A_\Delta$ commute, and assume that all connected components of $\Gamma$ and of $\Delta$ are either affine or finite $ADE$ Dynkin diagrams. Then all connected components of $\Gamma$ have the same Coxeter number denoted $h(\Gamma)$, and all connected components of $\Delta$ have the same Coxeter number denoted $h(\Delta)$.
\end{corollary}

Combining Lemma~\ref{lemma:eigenvalues}, Definition~\ref{dfn:affinite}, Definition~\ref{dfn:subadditive}, Vinberg's characterization (Theorem~\ref{thm:Vinberg}), and Proposition~\ref{prop:coxeter_eigenvalues}, we get the following proposition, whose part (\ref{item:admissible_ADE}) was shown in \cite[Proposition 5.1]{GP}. The proof for parts (\ref{item:affinite}) and (\ref{item:affaff}) is completely analogous and we refer the reader to~\cite{GP} for details.

\begin{proposition}\label{prop:affinite_subadditive}
 Let $Q$ be a bipartite recurrent quiver $Q$ and $G(Q)=(\Gamma,\Delta)$ be the corresponding bipartite bigraph. Then 
 \begin{enumerate}
  \item \label{item:admissible_ADE} $Q$ admits a strictly subadditive labeling if and only if $G(Q)$ is an admissible $ADE$ bigraph;
  \item \label{item:affinite} $Q$ admits a subadditive labeling which is not strictly subadditive if and only if $G(Q)$ is an \affinite $ADE$ bigraph;
  \item \label{item:affaff} $Q$ admits a weakly subadditive labeling which is not subadditive if and only if $G(Q)$ is an \affaff $ADE$ bigraph.
 \qed
 \end{enumerate}
\end{proposition}

\section{Zamolodchikov integrable quivers admit weakly subadditive labelings}\label{sec:implies}

Recall that a bipartite recurrent quiver $Q$ is called \emph{Zamolodchikov integrable} if for every vertex $v\in\VertQ$, there exists an integer $N$ and rational functions $\rec_0,\dots,\rec_N\in \Q(\x)$ such that $\rec_0,\rec_N\neq 0$ and 
\[\sum_{j=0}^{N} \rec_j T_v(t+2j)=0\]
for all $t\in\Z$ with $t+\e_v$ even.  

The following lemma is the first step towards the proof of Theorem~\ref{thm:implies}
\begin{lemma}\label{lemma:implies}
 If a bipartite recurrent quiver $Q$ is Zamolodchikov integrable then $Q$ admits a weakly subadditive labeling.
\end{lemma}
\begin{proof}

For $v\in\VertQ,t\in\Z$, define a positive number $a(v,t):=T_v(2t+\e_v)\mid_{\x:=1}$ to be the value of $T_v(2t+\e_v)$ if one substitutes $x_u:=1$ for all $u\in\VertQ$. By the Laurent Phenomenon (see~\cite{FZ}), the numbers $a(v,t)$ are integers. Note that, unlike $T_v(t)$, the numbers $a(v,t)$ are defined for all $v,t$, regardless of parity. 

Since $a(v,t)$ is always a positive integer, it is easy to see that the sequences $a(v,*):=\left(a(v,t)\right)_{t\in\Z}$ are either simultaneously bounded or simultaneously unbounded (for all $v$). Assume for the sake of contradiction that for some vertex $v$, the sequence $a(v,*)$ is unbounded, but there is another vertex $u$ for which the sequence $a(u,*)$ is bounded, say, $|a(u,t)|<C$ for all $t\in\Z$. Since $Q$ is connected, we may assume that $u$ and $v$ are neighbors in $Q$. Let $t$ be such that $a(v,t)>C^2$. Then by the definition of the $T$-system, we have
\[a(u,t+1)> \frac{a(v,t)}{a(u,t)}>C,\]
where the first inequality uses the fact that all the numbers involved are positive integers, hence each of them is at least $1$. This leads to an immediate contradiction.

If all the sequences are simultaneously bounded then they are periodic with the same period. This implies that the $T$-system associated with $Q$ is periodic for \emph{any} initial data, see~\cite[Remark 7.2]{GP}. In particular, such $Q$ admits a strictly subadditive labeling by~\cite[Theorem 1.10]{GP}. Thus the only case left for us to consider is when the sequence $a(v,*)$ is unbounded for every $v$.

We need to show that if $Q$ is Zamolodchikov integrable then $Q$ admits a weakly subadditive labeling. The way to find such a labeling is going to be very similar to the proof of \cite[Theorem 1]{ReuA}.

The fact that $Q$ is Zamolodchikov integrable implies that for each $v$, the sequences $a(v,*)$ satisfy a linear recurrence. Knowing that each of them is unbounded suggests using \cite[Lemma 1]{ReuA} that describes the asymptotic behavior of sequences $a(v,*)$. Before we state it, let us denote $A(k)\approx B(k)$ for two functions of $k$ if their ratio tends to a positive constant as $k\to \infty$. 
\begin{lemma}[see {\cite[Lemma 1]{ReuA}}]\label{lemma:reu}
Let $a(v,*)$ be an unbounded sequence of positive integers satisfying a linear recurrence for each $v\in\VertQ$. Then there exist:
\begin{itemize}
 \item an integer $p\geq 1$;
 \item real numbers $\l(v,l)\geq 1$ for each $v\in\VertQ,\ l=0,\dots,p$;
 \item integers $d(v,l)\geq 0$ for each $v\in\VertQ,\ l=0,\dots,p$;
 \item a strictly increasing sequence $(n_k)_{k\in\Z_{\geq 0}}$ of nonnegative integers
\end{itemize}
such that the following things hold:
\begin{enumerate}
 \item for every $v\in\VertQ$ and every $l=0,\dots,p$, $a(v,pn_k+l)\approx \l(v,l)^{n_k} n_k^{d(v,l)}$;
 \item \label{item:reu2} for every $v\in \VertQ$ there exists $l=0,\dots,p$ such that $\l(v,l)>1$ or $d(v,l)\geq 1$;
 \item \label{item:reu3} for every $v\in\VertQ$ we have $\l(v,0)=\l(v,p)$ and $d(v,0)=d(v,p)$.
\end{enumerate}
\qed
\end{lemma}
Clearly, the sequences $a(v,*)$ satisfy all the requirements of Lemma~\ref{lemma:reu}. For each $v\in\VertQ$, define 
\[\l(v):=\prod_{l=0}^{p-1} \l(v,l)\in\R_{\geq 1};\quad d(v):=\sum_{l=0}^{p-1} d(v,l)\in\Z_{\geq 0}.\]
For all $v\in\VertQ$ and $t\in\Z$, define
\[b(v,t):=\prod_{l=0}^{p-1} a(v,t+l).\]
Applying Lemma~\ref{lemma:reu} yields
\begin{equation}\label{eq:asymptotic}
b(v,pn_k)\approx \l(v)^{n_k} n_k^{d(v)}.
\end{equation}

By property (\ref{item:reu3}) of Lemma~\ref{lemma:reu}, we have $a(v,pn_k)\approx a(v,pn_k+p)$ for every $v\in\VertQ$ and thus we can write 
\begin{eqnarray*}
b(v,pn_k)^2 &\approx& \prod_{l=0}^{p-1} a(v,pn_k+l)a(v,pn_k+l+1) =\prod_{l=0}^{p-1} T_v(2(pn_k+l)+\e_v)T_v(2(pn_k+l+1)+\e_v)\\
 &=&\prod_{l=0}^{p-1} \left(\prod_{u\to v}T_u(2(pn_k+l)+\e_v+1)+\prod_{v\to w}T_w(2(pn_k+l)+\e_v+1)\right)\\
 &\geq&\left( \prod_{u\to v}\prod_{l=0}^{p-1}T_u(2(pn_k+l)+\e_v+1)\right)+\left(\prod_{v\to w}\prod_{l=0}^{p-1} T_w(2(pn_k+l)+\e_v+1)\right)\\
 &\approx& \left( \prod_{u\to v}b(u,pn_k)\right)+\left(\prod_{v\to w}b(w,pn_k)\right).
\end{eqnarray*}
The last equality is justified as follows: if $\e_v=0$ then $T_u(2(pn_k+l)+\e_v+1)$ is indeed equal to $a(u,pn_k+l)$. If $\e_v=1$ then $T_u(2(pn_k+l)+\e_v+1)=a(u,pn_k+l+1)$ and then we again use $a(u,pn_k)\approx a(u,pn_k+p)$ in order to get to the last line. 

By analyzing the asymptotics (\ref{eq:asymptotic}) of $b(v,pn_k)$, we see that for all $v\in\VertQ$,
\begin{equation}\label{eq:almost_there}
\l(v)^2\geq \max\left(\prod_{u\to v}\l(u),\prod_{v\to w}\l(w)\right);\quad d(v)\geq\max\left(\sum_{u\to v} d(u),\sum_{v\to w} d(w)\right). 
\end{equation}

Note that $\log\l(v)\geq 0$ and $d(v)\geq 0$ for all $v\in\VertQ$. Define $\nu(v):=\log\l(v)+d(v)$. By property~(\ref{item:reu2}) of Lemma~\ref{lemma:reu}, $\nu(v)>0$ for all $v\in\VertQ$. By (\ref{eq:almost_there}), $\nu$ is a weakly subadditive labeling of $\VertQ$.
\end{proof}

Applying Proposition~\ref{prop:affinite_subadditive}, we get
\begin{corollary}\label{cor:finite_or_affine}
 If a bipartite recurrent quiver $Q$ is Zamolodchikov integrable then $G(Q)$ is either 
 \begin{enumerate}
  \item\label{item:finite_finite} an admissible $ADE$ bigraph, or
  \item\label{cor:finite_or_affine_item:affinite} an \affinite $ADE$ bigraph, or
  \item an \affaff $ADE$ bigraph. \qed
 \end{enumerate} 
\end{corollary}

\begin{remark}\label{rmk:zperiod}
 We have shown in~\cite{GP} that case (\ref{item:finite_finite}) of the above corollary holds if and only if $Q$ is \emph{Zamolodchikov periodic}, that is, the $T$-system associated with $Q$ is periodic. Obviously, this is a special case of Zamolodchikov integrability.
\end{remark}

In fact, only cases (\ref{item:finite_finite}) and (\ref{cor:finite_or_affine_item:affinite}) of Corollary~\ref{cor:finite_or_affine} are possible when $Q$ is Zamolodchikov integrable:

\begin{theorem}\label{thm:implies}
If a bipartite recurrent quiver $Q$ is Zamolodchikov integrable then $G(Q)$ is either 
 \begin{enumerate}
  \item an admissible $ADE$ bigraph, or
  \item an \affinite $ADE$ bigraph.
 \end{enumerate} 
\end{theorem}

We postpone the proof of this theorem until Section~\ref{subsec:implies_tropical}.

\def\Edges{ \operatorname{Edges}}

{\Large\part{The classification of \affinite $ADE$ bigraphs}\label{part:classification}}
Each affine $ADE$ Dynkin diagram $\affL$ has the associated dominant eigenvector $\v_\affL:\Vert(\affL)\to\R$ corresponding to the eigenvalue $2$. In other words, for every $v\in\affL$ we have
\[2 \v_\affL(v)=\sum_{(v,w)\in \Edges(\affL)} \v_\affL(w).\]
We normalize $\v_\affL$ so that its entries are positive integers with the smallest entry equal to $1$. The values of $\v_\affL$ are given in Figure~\ref{figure:affine_ADE}.

\section{Self and double bindings}

In this section, we classify all the bipartite \affinite $ADE$ bigraphs $G=(\Gamma,\Delta)$ such that $\Gamma$ has either one or two connected components. If $\Gamma$ has just one connected component then $G$ is called a \emph{self binding}, and if $\Gamma$ has two connected components then $G$ is called a \emph{double binding}. We start with self bindings.

Throughout this section we assume that $h(\Delta)>2$, i.e. that $\Delta$ has at least one edge (because if $h(\Delta)=2$ then all connected components of $\Delta$ are of type $A_1$). 

\subsection{Self bindings}\label{subsec:self}

\begin{lemma}\label{lemma:self_bindings_eigenvalues}
 If $G=(\Gamma,\Delta)$ is a self binding then all the connected components of $\Delta$ are of type $A_2$.
\end{lemma}
\begin{proof}
 Let $\v:\Vert(G)\to\R$ be the common eigenvector for $A_\Gamma$ and $A_\Delta$ from Lemma~\ref{lemma:eigenvalues}. Thus $A_\Gamma\v=2\v$. Since $\Gamma$ has just one connected component, we may rescale $\v$ so that it is equal to $\v_\Gamma$. Now, let $\l_\Delta:=2\cos(\pi/h(\Delta))$ be the dominant eigenvalue for $A_\Delta$. We have that for every $v\in\Vert(G)$, 
 \[\sum_{(v,w)\in\Delta} \v(w)=\l_\Delta\v(v).\]
 Since there exists a vertex $v$ for which $\v(v)=1$, it follows that $\l_\Delta$ is an integer. This can only happen when $h(\Delta)=3$, that is, when all the connected components of $\Delta$ have Coxeter number $3$. The only finite $ADE$ Dynkin diagram with Coxeter number $3$ is $A_2$. 
\end{proof}

\newcommand{\selfb}[1]{\Scal_{#1}}

\begin{figure}
{$
\psmatrix[colsep=1cm,rowsep=0.5cm,mnode=circle]
            & \bl      \\
\wh         &           & \wh\\
  \bl       &           & \bl\\
            & \wh    
\psset{arrows=-,arrowscale=2}
\ra{2,1}{1,2}
\ra{2,1}{3,1}
\ra{4,2}{3,1}
\ra{2,3}{1,2}
\ra{2,3}{3,3}
\ra{4,2}{3,3}
\ba{1,2}{4,2}
\ba{2,1}{3,3}
\ba{3,1}{2,3}
\endpsmatrix $}
 \caption{\label{fig:self_binding} Self binding $\selfb{5}$.}
\end{figure}

\begin{proposition}\label{prop:self_bindings_classif}
\begin{itemize}
 \item For every $n\geq 1$, there is a self binding $\selfb{4n+1}=(\Gamma_n,\Delta_n)$ where $\Gamma_n$ is an affine $ADE$ Dynkin diagram of type $\affA_{4n+1}$, that is, a single cycle with $4n+2$ vertices, and two vertices of $\Gamma_n$ are connected by an edge of $\Delta_n$ iff they are the opposite vertices of that cycle (see Figure~\ref{fig:self_binding}); 
 \item There are no other self bindings.
\end{itemize}
\end{proposition}
\begin{proof}
 Let $G=(\Gamma,\Delta)$ be a self binding. By Lemma~\ref{lemma:self_bindings_eigenvalues}, all the components of $\Delta$ are just isolated single edges. Let us define an involution $\i:\Vert(G)\to\Vert(G)$ such that $v$ and $\i(v)$ are exactly the vertices connected by the edges of $\Delta$. This is a fixed point free involution, otherwise $\Delta$ would have a connected component of type $A_1$. Moreover, since $G$ is bipartite, $\i$ should reverse the colors of vertices. Finally, if $(u,v)\in\Gamma$ then one must also have $(\i(u),\i(v))\in\Gamma$ because otherwise the adjacency matrices $A_\Gamma$ and $A_\Delta$ would not commute. Thus $\i$ is a color-reversing involutive automorphism of $G$ without fixed points. The only affine $ADE$ Dynkin diagram admitting such an automorphism is $\affA_{4n+1}$ for $n\geq 1$, where the automorphism is just a rotation by $180^\circ$.
\end{proof}

\subsection{Double bindings: scaling factor}

The classification of double bindings is going to be much richer than that of self bindings. Throughout the rest of this section, we assume that $G=(\Gamma,\Delta)$ is a double binding, and that $\Vert(G)=X\sqcup Y$, where $X$ and $Y$ are the two connected components of $\Gamma$, and recall that they are affine $ADE$ Dynkin diagrams. A \emph{parallel binding} is a bigraph of type $\affL\otimes A_2$ and, following \cite{S}, is denoted $\affL\equiv\affL$.

\def\scf{ \operatorname{scf}}
\begin{definition}
 The \emph{scaling factor} of $G$ (denoted $\scf(G)$) is the number $\l_\Delta^2$ where $\l_\Delta=2\cos(\pi/h(\Delta))$ is the dominant eigenvalue for $A_\Delta$.
\end{definition}

\begin{proposition}\label{prop:scf_types}
 The scaling factor $\scf(G)$ is an integer equal to either $1,2,$ or $3$. Moreover,
 \begin{enumerate}
  \item if $\scf(G)=1$ then all connected components of $\Delta$ are of type $A_2$;
  \item if $\scf(G)=2$ then all connected components of $\Delta$ are of type $A_3$;
  \item if $\scf(G)=3$ then all connected components of $\Delta$ are either of type $A_5$ or of type $D_4$.
 \end{enumerate}
\end{proposition}
\begin{proof}
 We view maps $\tau:\Vert(G)\to\R$ as pairs $\col{\tau_X}{\tau_Y}$ where $\tau_X:\Vert(X)\to\R$ and $\tau_Y:\Vert(Y)\to\R$ are restrictions of $\tau$ to the corresponding subsets. Let $\tau=\col{\tau_X}{\tau_Y}$ be the common dominant eigenvector for $A_\Gamma$ and $A_\Delta$ from Lemma~\ref{lemma:eigenvalues}. We may rescale it so that $\tau_X=\alpha\v_X$ and $\tau_Y=\v_Y$ for some $\alpha\in\R$. Since the entries of the dominant eigenvector are positive, we may assume $\alpha>0$. Since $A_\Delta\tau=\l_\Delta\tau$, we have
 \begin{eqnarray}
 \label{eq:eig_defn_1}\sum_{(v,w)\in\Delta} \v_Y(w)&=&\l_\Delta \alpha \v_X(v),\quad \forall\,v\in X;\\
 \label{eq:eig_defn_2}\sum_{(v,w)\in\Delta}\alpha  \v_X(v)&=&\l_\Delta \v_Y(w),\quad \forall\,w\in Y. 
 \end{eqnarray}
If we substitute $v\in X$ such that $\v_X(v)=1$ in (\ref{eq:eig_defn_1}), we will get that $\l_\Delta\alpha\in\Z_{>0}$. Similarly, if we substitute $w\in X$ such that $\v_Y(w)=1$ in (\ref{eq:eig_defn_2}), we will get that $\l_\Delta/\alpha\in\Z_{>0}$. Therefore their product $\l_\Delta^2$ belongs to $\Z_{>0}$ as well. A straightforward case analysis shows that this can only happen when $h(\Delta)=3,4,$ or $6$, and the result follows.
\end{proof}
A simple consequence of the proof is the following observation:

\begin{corollary}\label{cor:points}
 Up to switching $X$ and $Y$, we have:
  \begin{eqnarray}
 \label{eq:points_1}\sum_{(v,w)\in\Delta} \v_Y(w)&=&\scf(G) \v_X(v),\quad \forall\,v\in X;\\
 \label{eq:points_2}\sum_{(v,w)\in\Delta}\v_X(v)&=&\v_Y(w),\quad \forall\,w\in Y. 
 \end{eqnarray}
\end{corollary}
\begin{proof}
 We know that $\l_\Delta^2\in\{1,2,3\}$ and thus $\l_\Delta\in\{1,\sqrt2,\sqrt3\}$. Thus the only $\alpha\in\R$ satisfying $\l_\Delta/\alpha\in\Z_{>0}$ and $\l_\Delta\alpha\in\Z_{>0}$ is either $\alpha=\l_\Delta$ or $\alpha=1/\l_\Delta$.
\end{proof}

By the same reasoning as in the proof of Proposition~\ref{prop:self_bindings_classif}, if $\scf(G)=1$ then $G$ is a parallel binding. It remains to classify double bindings with scaling factor $2$ and $3$. We say that a double binding is \emph{nontrivial} if it is not a parallel binding, i.e. if the scaling factor is $2$ or $3$.

\begin{definition}
 When $X$ is an affine $ADE$ Dynkin diagram of type $\affL$ and $Y$ is an affine $ADE$ Dynkin diagram of type $\affL'$ then we say that $G$ is a double binding \emph{of type $\affL\ast\affL'$}.
\end{definition}

Note that Corollary~\ref{cor:points} is not symmetric in $X$ and $Y$, so if $G$ is a double binding of type $\affL\ast\affL'$ then necessarily $X$ has type $\affL$, $Y$ has type $\affL'$ and (\ref{eq:points_1}) and (\ref{eq:points_2}) hold. In other words, we treat double bindings of types $\affL\ast\affL'$ and $\affL'\ast\affL$ differently.

A simple consequence of (\ref{eq:points_2}) is
\begin{corollary}\label{cor:max_value_v_Y}
 For any double binding $G$, the maximal value of $\v_X$ is less than or equal to the maximal value of $\v_Y$. 
\end{corollary}
\begin{proof}
 Let $\v_X(u)$ be the maximal value of $\v_X$, then clearly $\v_Y(w)\geq\v_X(u)$ for any $(u,w)\in\Delta$ by (\ref{eq:points_2}).
\end{proof}

Denote by $\v_Y^{-1}(1)$ the set of vertices $u$ of $Y$ with $\v_Y(u)=1$. 

\begin{proposition}\label{prop:not_the_same}
 There are no non-trivial double bindings of type $\affL\ast\affL$ (i.e. when $X$ and $Y$ have the same type).
\end{proposition}
\begin{proof}
 Let $M$ be the maximal value of $\v_X$ and $\v_Y$, and let $W=\v_X^{-1}(M),\ U=\v_Y^{-1}(M)$ be the sets of vertices where $\v_X$ (resp., $\v_Y$) takes the maximal value. It is clear from (\ref{eq:points_2}) that every vertex from $U$ is $\Delta$-connected to at most one vertex from $W$. By the same reason, every vertex from $\Vert(Y)\setminus U$ is not $\Delta$-connected to any vertex from $W$. Thus every vertex from $W$ is allowed to be $\Delta$-connected only to vertices from $U$, and by (\ref{eq:points_1}), each of them should be connected to at least two vertices in $U$. We get a contradiction since the sizes of $W$ and $U$ are supposed to be the same. 
\end{proof}

\subsection{Double bindings involving type $\affE$}
We say that $Y$ is \emph{one-two-bipartite} if for every $u,w\in\Vert(Y)$ with $\v_Y(u)=1$ and $\v_Y(w)=2$, we have $\e_u\neq \e_w$ (that is, all ones in $\v_Y$ are white and all twos in $\v_Y$ are black, or vice versa). Note that if $Y$ is of type $\affE_6$ or $\affE_8$ then $Y$ is one-two-bipartite, see Figure~\ref{figure:affine_ADE}.

\begin{lemma}\label{lemma:one_two_bip}
 Let $G$ be a double binding, and assume that $Y$ is one-two-bipartite. Then $\scf(G)$ divides $\#\v_Y^{-1}(1)$.
\end{lemma}
\begin{proof}
 Let $w\in \v_Y^{-1}(1)$. By \ref{eq:points_2}, there is exactly one vertex $v\in X$ with $(w,v)\in\Delta$, and moreover, $\v_X(v)=1$. By (\ref{eq:points_1}),
 \[\sum_{(v,u)\in\Delta} \v_Y(u)=\scf(G).\]
 Since $\scf(G)\leq 3$ and $\v_Y(w)=1$ is one of the terms in the left hand side, all the other terms in the left hand side are equal to either $1$ or $2$. But all vertices $u$ with $(v,u)\in\Delta$ must be of the same color, since the graph is bipartite. The set of $\Delta$-neighbors of $v$ consists of exactly $\scf(G)$ vertices $u$ with $\v_Y(u)=1$. By (\ref{eq:points_2}), $v$ is the only $\Delta$-neighbor of each such $u$. Therefore, the set $\v_Y^{-1}(1)$ is partitioned into classes, and each class has $\scf(G)$ members that have the same $\Delta$-neighbor.
\end{proof}

\begin{corollary}\label{cor:type_e_scf}
 If $G$ is a non-trivial double binding of type $\affL\ast\affE_6$ then $\scf(G)=3$.
\end{corollary}

%
%

\begin{proposition}\label{prop:no_L_E8_and_En_L}
\begin{enumerate}
 \item\label{item:no_of_type_affL_affE8} There are no non-trivial double bindings of type $\affL\ast \affE_8$;
 \item\label{item:no_of_type_affEn_affL} the only non-trivial double binding of type $\affE_n\ast \affL$ is the double binding $\affE_6\ast\affE_7$ depicted in Figure~\ref{figure:double_bindings_scf_2}. 
\end{enumerate}
\end{proposition}
\begin{proof}
To prove (\ref{item:no_of_type_affL_affE8}), just observe that if $Y$ is of type $\affE_8$ then $\#\v_Y^{-1}(1)=1$ and apply Lemma~\ref{lemma:one_two_bip}.

To prove (\ref{item:no_of_type_affEn_affL}), we can first eliminate all the cases except for $\affE_6\ast\affE_7$:
\begin{itemize}
\item by~(\ref{item:no_of_type_affL_affE8}), there are no bindings of type $\affE_n\ast\affE_8$;
 \item by Proposition~\ref{prop:not_the_same}, there are no bindings of types $\affE_6\ast\affE_6$ or $\affE_7\ast\affE_7$;
 \item by Corollary~\ref{cor:max_value_v_Y}, there are no bindings of types $\affE_7\ast\affE_6$, $E_n\ast A_m$, or $E_n\ast D_m$.
\end{itemize}

Now we need to prove that there is only one double binding of type $\affE_6\ast\affE_7$. Let $\{w_1,w_2,w_3\}$ be all the vertices of $X$ (which is of type $\affE_6$) with $\v_X(w_i)=2$ for $i=1,2,3$. Since $Y$ is of type $\affE_7$, it has $5$, say, white vertices and $3$ black vertices. Let $\{u_1,u_2,u_3\}$ be these three black vertices. Since $w_1,w_2,w_3$ are all of the same color, it is clear from (\ref{eq:points_2}) that they are white (because if the left hand side of (\ref{eq:points_2}) is even then the right hand side should be also even), and thus the other $4$ vertices of $X$ are black.  To sum up, the edges of $\Delta$ connect the vertices $u_1,u_2,u_3$ to the vertices $w_1,w_2,w_3$, and we have 
\[\v_X(w_1)=\v_X(w_2)=\v_X(w_3)=\v_Y(u_1)=\v_Y(u_3)=2,\quad \v_Y(u_2)=4.\]
A simple case analysis shows that $u_2$ is $\Delta$-connected to two vertices, say, to $w_1$ and $w_2$ while $u_1$ and $u_3$ are then both connected to $w_3$. Now, using the fact that the adjacency matrices $A_\Gamma$ and $A_\Delta$ commute, there is only one way to recover the rest of the double binding, and we get exactly $\affE_6\ast\affE_7$ from Figure~\ref{figure:double_bindings_scf_2}.
\end{proof}

We conclude the analysis of double bindings for which one of the components is of type $\affE_n$ with the following proposition:
\begin{proposition}
\begin{enumerate}
 \item\label{item:no_bindings_affA_affE} There are no non-trivial double bindings of type $\affA_m\ast \affE_n$;
 \item\label{item:one_binding_affD4_affE6} there is exactly one non-trivial double binding of type $\affD_m\ast\affE_6$, namely, the binding $\affD_4\ast\affE_6$ depicted in Figure~\ref{figure:double_bindings_scf_3};
 \item\label{item:one_binding_affD6_affE7} there is exactly one non-trivial double binding of type $\affD_m\ast\affE_7$, namely, the binding $\affD_6\ast\affE_7$ depicted in Figure~\ref{figure:double_bindings_scf_3}.
\end{enumerate} 
\end{proposition}
\begin{proof}
First, we show (\ref{item:no_bindings_affA_affE}). If $X$ is of type $\affA_m$ and $Y$ is of type $\affE_n$, then $\v_X(w)=1$ for all $w\in X$. If $Y$ has type $\affE_7$ then there is a vertex $u\in Y$ with $\v_Y(u)=4$ which is impossible since $u$ has at most three neighbors, so by (\ref{eq:points_2}), $\v_Y(u)\leq 3$. By Proposition~\ref{prop:no_L_E8_and_En_L}, $Y$ cannot be of type $\affE_8$, so assume now that $Y$ is of type $\affE_6$. Let $\v_Y^{-1}(1)=\{u_1,u_2,u_3\}$. Then there is a vertex $w_1\in X$ connected by $\Delta$ to all of them. Let $w_2$ be such that $(w_2,w_1)\in\Gamma$. Since $w_2$ is of different color, it can only be $\Delta$-connected to vertices $u\in Y$ with $\v_Y(u)=2$. But the sum
$\sum_{(u,w_2)\in\Delta} \v_Y(u)$ should be equal to $3$ which is impossible because it is even. Thus (\ref{item:no_bindings_affA_affE}) follows.

Next, we prove (\ref{item:one_binding_affD4_affE6}), so assume $X$ has type $\affD_m$ and $Y$ has type $\affE_6$. By Corollary~\ref{cor:type_e_scf}, the scaling factor in this case equals to $3$. Let $\v_Y^{-1}(1)=\{u_1,u_2,u_3\}$. Then all of them are connected to some vertex $w_1\in X$ with $\v_X(w_1)=1$. Therefore $w_1$ has a unique $\Gamma$-neighbor $w_2\in X$, and $\v_X(w_2)=2$. Since the adjacency matrices $A_\Gamma$ and $A_\Delta$ commute, $w_2$ should be connected to all three vertices $u_4,u_5,u_6$ of $Y$ satisfying $\v_Y(u_i)=2$ for $i=4,5,6$. Since $X$ has three more vertices $w_3,w_4,w_5$ with $\v_X(w_i)=1$ for $i=3,4,5$, each of them has to be connected to the remaining vertex $u_7$ of $Y$ with $\v_Y(u_7)=3$. It follows that there are no more vertices in $X$, so we are done with (\ref{item:one_binding_affD4_affE6}).

Finally, we show (\ref{item:one_binding_affD6_affE7}), so let $X$ have type $\affD_m$ and let $Y$ have type $\affE_7$. Assume first that the scaling factor is $2$, and let $u\in \Vert(Y)$ be a vertex with $\v_Y(u)=3$. Then by (\ref{eq:points_1}), if $(u,w)\in\Delta$ for some $w\in\Vert(X)$, then $\v_X(w)\geq 2$, but since $X$ is of type $\affD_m$, $\v_X(w)$ must be equal to $2$. Since $\v_Y(u)$ is odd, this contradicts (\ref{eq:points_2}). 

Thus the scaling factor has to be equal to $3$. Because $Y$ is of type $\affE_7$, $Y$ has $3$, say, black vertices $u_1,u_2,u_3$, and $5$ white vertices, and we have $\v_Y(u_1)=\v_Y(u_3)=2,\v_Y(u_2)=4$. It follows now that:
\begin{itemize}
 \item $X$ has exactly $2$ white vertices $w_1$ and $w_2$; 
 \item one of the components of $\Delta$ has type $A_5$ and connects the vertices $u_1 - w_1 - u_2 - w_2 - u_3$.
\end{itemize}
Again, using commuting adjacency matrices, one can reconstruct the rest of the double binding and see that it is in fact $\affD_6\ast\affE_7$ in Figure~\ref{figure:double_bindings_scf_3}.
\end{proof}

\subsection{Double bindings involving type $\affA$}
One can identify the vertices of the cycle $A_{2m-1}$ with $\Z_m:=\Z/m\Z$. We define double and triple coverings to be the following double bindings: in a  \emph{double covering} $\affA_{2n-1}\ast\affA_{4n-1}$, a vertex $j\in\Z_{4m}$ of $Y$ is connected by a blue edge to a vertex $i\in\Z_{2m}$ of $X$ iff $i\equiv j\pmod {2m}$. Similarly, in a \emph{triple covering} $\affA_{2n-1}\ast\affA_{6n-1}$, a vertex $j\in\Z_{6m}$ of $Y$ is connected by a blue edge to a vertex $i\in\Z_{2m}$ of $X$ iff $i\equiv j\pmod {2m}$. These are obviously \affinite $ADE$ bigraphs.

\begin{proposition}
 The only possible double bindings of type $\affA_{m-1}\ast\affA_{k-1}$ are:
 \begin{enumerate}
  \item parallel bindings $\affA_{2n-1}\equiv \affA_{2n-1}$;
  \item double coverings $\affA_{2n-1}\ast\affA_{4n-1}$;
  \item triple coverings $\affA_{2n-1}\ast\affA_{6n-1}$.
 \end{enumerate}
\end{proposition}
\begin{proof}
 By (\ref{eq:points_2}), each vertex of $Y$ has exactly one blue neighbor, and each vertex of $X$ has exactly $\scf(G)$ blue neighbors. Let $(w_i)_{i\in\Z_k}$ be the vertices of $Y$ listed in cyclic order, and let $(v_i)_{i\in\Z_m}$ be the vertices of $X$ in cyclic order. Let $f:\Z_k\to\Z_m$ be the map such that $v_{f(i)}$ is the unique blue neighbor of $w_i$ for all $i\in \Z_k$. Since the adjacency matrices have to commute, we get that $\{f(i+1),f(i-1)\}=\{f(i)+1,f(i)-1\}$ which immediately yields the result of the proposition.
\end{proof}

By Corollary~\ref{cor:max_value_v_Y}, there are no double bindings of type $\affD_m\ast\affA_n$ so the only case left in this section is $\affA_n\ast\affD_m$.

\begin{proposition}
The only possible double bindings of type $\affA_n\ast\affD_m$ are the double bindings of type $\affA_{2n-1}\ast\affD_{n+2}$ in Figure~\ref{figure:double_bindings_scf_2} and the exceptional double binding of type $\affA_3\ast\affD_5$ in Figure~\ref{figure:double_bindings_scf_3}.
\end{proposition}
\begin{proof}
 We have two options: either $\scf(G)=2$ or $\scf(G)=3$. If $\scf(G)=2$ then we know that each non-leaf vertex of $Y$ is connected to exactly two vertices of $X$, and is the only blue neighbor of each of them. On the other hand, there are two more vertices $v_1,v_2$ in $X$, and each of them has two blue neighbors which are leaves in $Y$. Now using commuting adjacency matrices condition one can easily recover that $G$ is the double binding of type $\affA_{2n-1}\ast\affD_{n+2}$ from Figure~\ref{figure:double_bindings_scf_2}.
 
 Now assume that $\scf(G)=3$. This means that each vertex of $X$ is connected to an odd number of leaves of $Y$. Since $Y$ has exactly four leaves, it follows that $X$ has either two or four vertices. If $X$ has two vertices then the sum of values of $\v_Y$ is six so $Y$ has type $\affD_4$ but then all the leaves of $Y$ have the same color so one of the vertices of $X$ is not going to be connected to any of them. We are left with the case when $X$ has four vertices and each of them is connected to a leaf of $Y$ and to a non-leaf of $Y$. Therefore $Y$ has type $\affD_5$ from which one can quickly see that $G$ is the unique double binding of type $\affA_3\ast\affD_5$ from Figure~\ref{figure:double_bindings_scf_3}.
\end{proof}

\subsection{Double bindings of type $\affD_{m+1}\ast\affD_{k+1}$}\label{subsec:double_DD}
\begin{proposition}
 The only possible double bindings of type $\affD_{m+1}\ast\affD_{k+1}$ are the double bindings of type $\affD_{n}\ast\affD_{2n-2}$ and the double bindings of type $\affD_{n+1}\ast\affD_{3n-1}$ constructed in the proof of this proposition and depicted for small $n$ in Figure~\ref{figure:double_bindings_scf_3}.
\end{proposition}
\begin{proof}
 Let $v_1^+,v_1^-,v_2,\dots,v_{m-1},v_{m}^+,v_{m}^-$ be the vertices of $X$, the component of type $\affD_{m+1}$, and $w_1^+,w_1^-,w_2,\dots,w_{k-1},w_k^+,w_k^-$ be the vertices of $Y$ which has type $\affD_{k+1}$. Here we assume that $v_1^+,v_1^-$ are the leaves attached to $v_2$ and so on. By~(\ref{eq:points_2}), each of $w_1^+,w_1^-,w_k^+,w_k^-$ is connected to exactly one leaf of $X$. Without loss of generality assume that $w_1^+$ is connected to $v_1^+$ by a blue edge. Since the adjacency matrices commute, $w_2$ has to be connected to $v_2$ by blue edges. We claim that $w_1^-$ cannot be connected to $v_1^+$. Indeed, otherwise there would be at least two blue-red paths and at most one red-blue path from $v_1^+$ to $w_2$, so the matrices would not commute. On the other hand, $w_1^-$ is connected to a leaf, and this leaf has to be a neighbor of $v_2$. So without loss of generality we may assume that $w_1^-$ is connected to $v_1^-$ (we only make a choice here when $Y$ has type $\affD_4$ in which case all the four leaves of $Y$ are connected to $w_2$). 
 
  By Proposition~\ref{prop:not_the_same}, we have $m\neq k$ and by (\ref{eq:points_1})-(\ref{eq:points_2}) we actually have $m<k$. We claim that for each $i=2,\dots,m-1$, $w_i$ is connected to $v_i$, and thus to nothing else by (\ref{eq:points_2}). We show it by induction on $i$, where the base $i=2$ has already been shown. Assume that $w_{i}$ is connected to $v_{i}$. Then there is a red-blue path from $v_{i+1}$ to $w_{i}$, and $v_{i+1}$ is not connected to $w_{i-1}$ so it has to be connected to $w_{i+1}$, and the claim follows for $i=2,\dots,m-1$. Now there is a red-blue path from $v_m^+$ to $w_{m-1}$ so by the same reasoning $v_m^+$ and $v_m^-$ are connected to $w_m$. Now there is a red-blue path from $w_{m+1}$ to $v_m^+$ and to $v_m^-$ so $w_{m+1}$ is connected to $v_{m-1}$. Again using induction we can show that for $i=1,2,\dots,m-2$, $w_{m+i}$ is connected to $v_{m-i}$. This includes the fact that $2m-2<k$ which is true since before we stop we need to add another blue edge to $v_1^+$ in order to satisfy (\ref{eq:points_1}). 
  
  If $\scf(G)=2$ then (\ref{eq:points_1}) is satisfied for all vertices of $X$ except for $v_1^+$ and $v_1^-$ so we complete the construction of the graph by joining $v_1^+$ to $w_k^+$ and $v_1^-$ to $w_k^-$, where necessarily $k=2m-1$. This can considered to be the definition of $\affD_n\ast\affD_{2n-2}$. 
  
  If $\scf(G)=3$ then (\ref{eq:points_1}) is not satisfied for $v_m^+$ yet so we note that $2m-1<k$ and thus have to connect both $v_1^+$ and $v_1^-$ to $w_{2m-1}$. But then there is a red-blue path from $w_{2m}$ to $v_1^+$ so $w_{2m}$ has to be connected to $v_2$. Now for $i=1,2,\dots,m-2$ it follows that $w_{2m+i-1}$ is connected to $v_{i+1}$. After that (\ref{eq:points_1}) fails only for $v_m^+$ and $v_m^-$ which we connect to $w_k^+$ and $w_k^-$ respectively. Here $k$ is necessarily equal to $3m-1$ yielding the double binding of type $\affD_{n+1}\ast\affD_{3n-1}$. 
\end{proof}

\subsection{The classification of self and double bindings}
We summarize the results of Sections~\ref{subsec:self}-\ref{subsec:double_DD} in the following theorem:
\begin{theorem}

\begin{itemize}

 \item The only possible self bindings are $\selfb{4n+1}$ for $n\geq 1$.
 \item all the double bindings with scaling factor $2$ are listed in Figure~\ref{figure:double_bindings_scf_2};
 \item all the double bindings with scaling factor $3$ are listed in Figure~\ref{figure:double_bindings_scf_3};
 \item the only other double bindings are parallel bindings $\affL\equiv\affL$.
\end{itemize}
 \qed
\end{theorem}

\begin{figure}
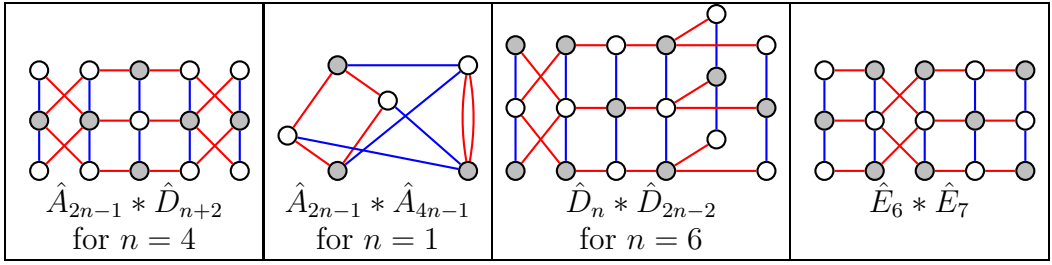

 \begin{tabular}{|c|c|c|c|} \hline
 {
$
\psmatrix[colsep=0.4cm,rowsep=0.4cm,mnode=circle]
\wh         & \wh       & \bl       & \wh       & \wh       \\
\bl         & \bl       & \wh       & \bl       & \bl		\\
\wh         & \wh       & \bl       & \wh       & \wh
\psset{arrows=-,arrowscale=2}
\foreach \y in {1,2,3,4,5}
{
\ba{1,\y}{2,\y}
\ba{2,\y}{3,\y}
}
\ra{1,1}{2,2}
\ra{3,1}{2,2}
\ra{2,2}{2,3}
\ra{2,4}{2,3}
\ra{2,4}{1,5}
\ra{2,4}{3,5}
\ra{2,1}{1,2}
\ra{2,1}{3,2}
\ra{1,2}{1,3}
\ra{1,4}{1,3}
\ra{3,2}{3,3}
\ra{3,4}{3,3}
\ra{1,4}{2,5}
\ra{3,4}{2,5}
\endpsmatrix $} 
&
{$
\psmatrix[colsep=0.4cm,rowsep=0.2cm,mnode=circle]
            & \bl       &           &           & \wh\\
            &           & \wh\\
\wh \\
            & \bl       &           &           & \bl
\psset{arrows=-,arrowscale=2}
\ra{3,1}{1,2}
\ra{3,1}{4,2}
\ra{2,3}{4,2}
\ra{2,3}{1,2}
\ba{1,2}{1,5}
\ba{4,2}{1,5}
\ba{3,1}{4,5}
\ba{2,3}{4,5}
\ncarc[arcangle=-10,linecolor=red]{-}{1,5}{4,5}
\ncarc[arcangle=10,linecolor=red]{-}{1,5}{4,5}
\endpsmatrix $} 
&
{$
\psmatrix[colsep=0.4cm,rowsep=0.15cm,mnode=circle]
            &           &           &           &    \wh                \\
  \bl       &   \bl     &  \wh      &  \bl      &           & \wh       \\
            &           &           &           &     \bl              \\
\wh         & \wh       &   \bl     &  \wh      &           &  \bl      \\
            &           &           &           &    \wh                \\
  \bl       &   \bl     &  \wh      &  \bl      &           & \wh       
\psset{arrows=-,arrowscale=2}
\foreach \y in {1,2,3,4,6}
{
\ba{2,\y}{4,\y}
\ba{6,\y}{4,\y}
}
\ba{1,5}{3,5}
\ba{5,5}{3,5}
\ra{2,1}{4,2}
\ra{6,1}{4,2}
\ra{4,3}{4,2}
\ra{4,3}{4,4}
\ra{4,6}{4,4}
\ra{3,5}{4,4}
\ra{4,1}{2,2}
\ra{4,1}{6,2}
\ra{6,3}{6,2}
\ra{6,3}{6,4}
\ra{6,6}{6,4}
\ra{5,5}{6,4}
\ra{2,3}{2,2}
\ra{2,3}{2,4}
\ra{2,6}{2,4}
\ra{1,5}{2,4}
\endpsmatrix $} 
&
{
$
\psmatrix[colsep=0.4cm,rowsep=0.4cm,mnode=circle]
\wh         & \bl       & \bl       & \wh       & \bl       \\
\bl         & \wh       & \wh       & \bl       & \wh		\\
\wh         & \bl       & \bl       & \wh       & \bl
\psset{arrows=-,arrowscale=2}
\foreach \y in {1,2,3,4,5}
{
\ba{1,\y}{2,\y}
\ba{2,\y}{3,\y}
}
\ra{1,1}{1,2}
\ra{2,3}{1,2}
\ra{2,3}{2,4}
\ra{2,5}{2,4}
\ra{2,3}{3,2}
\ra{3,1}{3,2}
\ra{2,1}{2,2}
\ra{1,3}{2,2}
\ra{1,3}{1,4}
\ra{1,5}{1,4}
\ra{3,3}{2,2}
\ra{3,3}{3,4}
\ra{3,5}{3,4}
\endpsmatrix $} \\
$\affA_{2n-1}\ast\affD_{n+2}$ &$\affA_{2n-1}\ast\affA_{4n-1}$ & $\affD_{n}\ast\affD_{2n-2}$ &$\affE_{6}\ast\affE_{7}$ \\
for $n=4$ & for $n=1$ & for $n=6$ & \\\hline
 \end{tabular}
\caption{\label{figure:double_bindings_scf_2}Three infinite and one exceptional family of double bindings with scaling factor $2$. All blue components have type $A_3$.}
\end{figure}

\begin{figure}
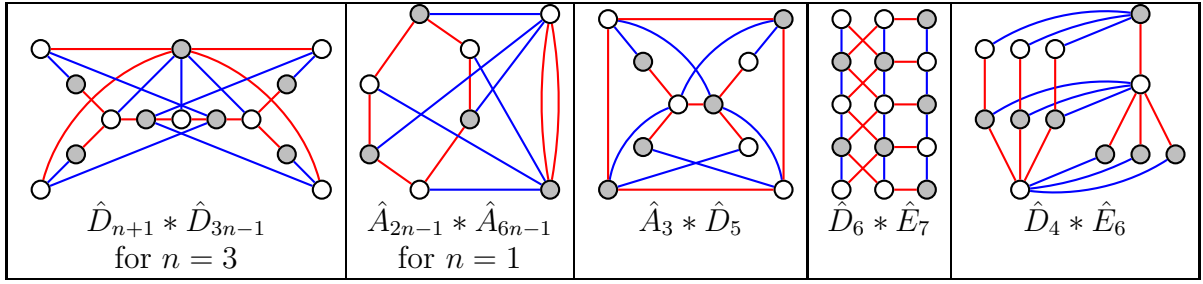

 \begin{tabular}{|c|c|c|c|c|} \hline
 {
$
\psmatrix[colsep=0.2cm,rowsep=0.2cm,mnode=circle]
   \wh      &           &           &           &    \bl    &           &           &           &   \wh     \\
            &   \bl     &           &           &           &           &           &   \bl                 \\
            &           &    \wh    &  \bl      &  \wh      &  \bl      &   \wh                             \\
            &   \bl     &           &           &           &           &           &   \bl                 \\
   \wh      &           &           &           &           &           &           &           &   \wh     
\psset{arrows=-,arrowscale=2}
\ncarc[arcangle=-30,linecolor=red]{-}{1,5}{5,1}
\ncarc[arcangle=30,linecolor=red]{-}{1,5}{5,9}
\ra{1,1}{1,5}
\ra{1,5}{1,9}
\ra{2,2}{3,3}
\ra{4,2}{3,3}
\ra{3,4}{3,3}
\ra{3,4}{3,5}
\ra{3,6}{3,5}
\ra{3,6}{3,7}
\ra{2,8}{3,7}
\ra{4,8}{3,7}
\ba{1,1}{2,2}
\ba{1,1}{3,6}
\ba{5,1}{3,6}
\ba{5,1}{4,2}
\ba{2,8}{1,9}
\ba{3,4}{1,9}
\ba{3,4}{5,9}
\ba{4,8}{5,9}
\ba{1,5}{3,3}
\ba{1,5}{3,5}
\ba{1,5}{3,7}
\endpsmatrix $} 
& 
{$
\psmatrix[colsep=0.4cm,rowsep=0.2cm,mnode=circle]
            & \bl       &           &           & \wh\\
            &           & \wh\\
\wh \\
            &           & \bl\\
\bl \\
            & \wh       &           &           & \bl
\psset{arrows=-,arrowscale=2}
\ra{3,1}{1,2}
\ra{3,1}{5,1}
\ra{2,3}{4,3}
\ra{2,3}{1,2}
\ra{5,1}{6,2}
\ra{4,3}{6,2}
\ba{1,2}{1,5}
\ba{4,3}{1,5}
\ba{5,1}{1,5}
\ba{3,1}{6,5}
\ba{2,3}{6,5}
\ba{6,2}{6,5}
\ncarc[arcangle=-10,linecolor=red]{-}{1,5}{6,5}
\ncarc[arcangle=10,linecolor=red]{-}{1,5}{6,5}
\endpsmatrix $} 
& 
{
$
\psmatrix[colsep=0.2cm,rowsep=0.3cm,mnode=circle]
   \wh      &           &           &           &           &   \bl     \\
            &   \bl     &           &           &       \wh                 \\
            &           &    \wh    &  \bl                               \\
            &   \bl     &           &           &        \wh                \\
   \bl      &           &           &           &           &      \wh     
\psset{arrows=-,arrowscale=2}
\ncarc[arcangle=30,linecolor=blue]{-}{5,1}{3,3}
\ncarc[arcangle=-30,linecolor=blue]{-}{5,6}{3,4}
\ncarc[arcangle=30,linecolor=blue]{-}{1,1}{3,4}
\ncarc[arcangle=-30,linecolor=blue]{-}{1,6}{3,3}
\ra{1,1}{1,6}
\ra{5,6}{1,6}
\ra{5,6}{5,1}
\ra{1,1}{5,1}
\ra{2,2}{3,3}
\ra{4,2}{3,3}
\ra{3,4}{3,3}
\ra{3,4}{2,5}
\ra{3,4}{4,5}
\ba{1,1}{2,2}
\ba{5,6}{4,2}
\ba{5,1}{4,5}
\ba{1,6}{2,5}
\endpsmatrix $} 
 & 
{
$
\psmatrix[colsep=0.3cm,rowsep=0.3cm,mnode=circle]
   \wh      & \wh       &   \bl     \\
   \bl      & \bl       &   \wh     \\
   \wh      & \wh       &   \bl     \\
   \bl      & \bl       &   \wh     \\
   \wh      & \wh       &   \bl     
\psset{arrows=-,arrowscale=2}
\foreach \x in {1,2,3,4}
{
\ra{\x,1}{\plusOne{\x},2}
\ra{\x,2}{\plusOne{\x},1}
\ba{\x,1}{\plusOne{\x},1}
\ba{\x,2}{\plusOne{\x},2}
\ba{\x,3}{\plusOne{\x},3}
}
\foreach \x in {1,2,3,4,5}
{
\ra{\x,3}{\x,2}
}
\endpsmatrix $} 
 & 
{
$
\psmatrix[colsep=0.2cm,rowsep=0.2cm,mnode=circle]
            &           &           &           &           &     \bl               \\
   \wh      &  \wh      &   \wh                                                     \\
            &           &           &           &           &     \wh               \\
   \bl      &  \bl      &   \bl                                                     \\
            &           &           &           &   \bl     &     \bl   &\bl        \\
            &  \wh                                                                  
\psset{arrows=-,arrowscale=2}
\ncarc[arcangle=20,linecolor=blue]{-}{2,1}{1,6}
\ncarc[arcangle=10,linecolor=blue]{-}{2,2}{1,6}
\ba{2,3}{1,6}
\ncarc[arcangle=20,linecolor=blue]{-}{4,1}{3,6}
\ncarc[arcangle=10,linecolor=blue]{-}{4,2}{3,6}
\ba{4,3}{3,6}
\ncarc[arcangle=20,linecolor=blue]{-}{5,7}{6,2}
\ncarc[arcangle=10,linecolor=blue]{-}{5,6}{6,2}
\ba{5,5}{6,2}
\ra{2,1}{4,1}
\ra{2,2}{4,2}
\ra{2,3}{4,3}
\ra{6,2}{4,1}
\ra{6,2}{4,2}
\ra{6,2}{4,3}
\ra{1,6}{3,6}
\ra{5,5}{3,6}
\ra{5,6}{3,6}
\ra{5,7}{3,6}
\endpsmatrix $}\\ 
$\affD_{n+1}\ast\affD_{3n-1}$ &$\affA_{2n-1}\ast\affA_{6n-1}$ & $\affA_{3}\ast\affD_{5}$ &$\affD_{6}\ast\affE_{7}$&$\affD_{4}\ast\affE_{6}$ \\
for $n=3$ & for $n=1$ &  & &  \\\hline
 \end{tabular}
\caption{\label{figure:double_bindings_scf_3}Two infinite and three exceptional families of double bindings with scaling factor $3$. All blue components have types $A_5$ or $D_4$.}
\end{figure}

\begin{remark}
 In~\cite[Section 9.1]{GP}, we introduced duality of symmetric bigraphs (not to be confused with Stembridge's dual bigraphs in \cite{S}). Here we briefly list some pairs of dual symmetric bigraphs for certain choices of the auxiliary  data\footnote{i.e. $\i,V_+,V_0,V_-,X$ in the notation of \cite{GP}} which we omit:
 \begin{itemize}
  \item $\affA_{2n-1}\ast\affD_{n+2}$ is dual to $\affA_{2n-1}\otimes A_3$;
  \item $\affA_{2n-1}\ast\affA_{4n-1}$ is dual to $\affA_{4n-1}\otimes A_3$;
  \item $\affD_{n}\ast\affD_{2n-2}$ is dual to $\affA_{2n-1}\ast\affA_{4n-1}$;
  \item $\affE_{6}\ast\affE_{7}$ is dual to itself;
  \item $\affD_{n+1}\ast\affD_{3n-1}$ is dual to $\affA_{2n-1}\ast\affA_{6n-1}$;
  \item $\affD_6\ast\affE_7$ is dual to $\affD_4\ast\affE_6$;
  \item $\affA_3\ast\affD_5$ is dual to the triple covering $\affA_1\ast\affA_5$.
 \end{itemize}
\end{remark}

\section{The classification}\label{sec:classification}
To classify \affinite $ADE$ bigraphs, we mostly follow the strategy of \cite{S}: we are going to show that the \emph{component graph} $\Ccal$ of $\Gamma$ defined below is a path with either at most one loop (in case there is a self binding) or at most one non-parallel double binding.

\begin{definition}
 Let $G=(\Gamma,\Delta)$ be a bigraph. Let $C_1,C_2,\dots,C_m$ be the connected components of $\Gamma$. Define the graph $\Ccal=\Ccal(G)$ with vertex set $[m]:=\{1,2,\dots,m\}$ such that $(i,j)$ is an edge of $\Ccal$ iff there is a blue edge $(u,v)\in\Delta$ with $u\in C_i$ and $v\in C_j$.
\end{definition}

\def\Gt{\widetilde{G}}
Let $G$ be an \affinite $ADE$ bigraph. We define its reduced version $\Gt$ to be the same as $G$ but with all the blue edges removed from each self binding in $G$. Clearly, $\Ccal(\Gt)$ is $\Ccal(G)$ with all the loops removed. It is also clear that $\Gt$ is going to be an \affinite $ADE$ bigraph as well.

Several properties of \affinite $ADE$ bigraphs have literally the same statements as their analogs for admissible $ADE$ bigraphs of \cite{S}, so we list them with the corresponding references to the parts of \cite{S} where they are proved:

\begin{lemma}\label{lemma:properties}
Let $G$ be an \affinite $ADE$ bigraph. Then:
\begin{enumerate}
 \item the component graph $\Ccal(\Gt)$ is acyclic (see \cite[proof of Lemma 2.5(b)]{S};
 \item in fact, $\Ccal(\Gt)$ is a path (see \cite[Proof of (ii) in Section 5]{S});
 \item $G$ contains at most one non-parallel double binding (see \cite[Proof of (ii) in Section 5]{S}).
\end{enumerate}
\qed
\end{lemma}

These properties allow us to describe every \affinite $ADE$ bigraph by a string of symbols $\affA_{n},\affD_n,\affE_n,\selfb{n}$ with symbols $\ast,\equiv$ inserted between them, for example, $\affL_1\equiv \affL_1\ast \affL_2$ has three red connected components (i.e. $m=3$) and $C_1$ and $C_2$ form a parallel binding while $C_2$ and $C_3$ form a double binding of type $\affL_1\ast\affL_2$. 

\begin{lemma}\label{lemma:self_bindings}
 Assume that $G$ is an \affinite $ADE$ bigraph containing a self binding. Then it contains exactly one self binding and all the double bindings in $G$ are parallel.
\end{lemma}
\begin{proof}
 Assume for the sake of contradiction that $G$ has at least two self bindings. We may remove everything else so that they occur at the ends of $\Ccal(\Gt)$ (which is a path on $[m]$). After some relabeling, the edges of $\Ccal(\Gt)$ become exactly $\{(i,i+1)\}_{i\in [m-1]}$. We are going to construct a blue cycle in $G$ as follows: let $v^1_1\in C_1$ be any vertex, then there is a blue path $v^1_1,v^1_2,\dots,v^1_m$ with $v^1_i\in C_i$. Since $C_m$ is a self binding, $v^1_m$ is connected by a blue edge to some other vertex $v^2_m\in C_m$, from which we can construct a blue path $v^2_m,v^2_{m-1},\dots,v^2_1$ with $v^2_i\in C_i$ again. But now $v_2^1$ is connected by a blue edge to some other vertex $v_3^1\in C_1$, so we may continue our path until it crosses itself yielding a blue cycle in $G$ which is a contradiction since all the finite $ADE$ Dynkin diagrams are acyclic.
 
 Assume now that there is a self binding and a non-parallel double binding in $G$. Again, we may assume that the self binding occurs in $C_1$ and the double binding occurs between $C_{m-1}$ and $C_m$ with $\Ccal(\Gt)$ being a path on $[m]$. Take the maximal blue path $P$ in $G$. Since all the vertices in $C_1,\dots,C_{m-1}$ have blue degree at least $2$, both endpoints of $P$ belong to $C_m$ and have blue degree $1$. But since the blue components of the double binding $C_{m-1}\ast C_m$ are either $A_3,A_5,$ or $D_4$ (see Proposition~\ref{prop:scf_types}), the vertices of $P$ adjacent to the endpoints have blue degrees at least $3$. Therefore they coincide because every finite $ADE$ Dynkin diagram contains at most one vertex of degree $3$. So $P$ has length at most $3$, and therefore $m=2$. It is clear that adding a self binding to any of the double bindings involving type $\affA$ yields either a cycle or a blue component with at least two vertices of degree $3$. 
\end{proof}

\def\selfbindingname{(\affA_{4n+1})^m}
\begin{proposition}\label{prop:self_binding}
 The only \affinite bigraphs involving self-bindings are
\[\selfbindingname:=\selfb{4n+1}\equiv\affA_{4n+1}\equiv\cdots\equiv\affA_{4n+1} \quad\text{($m$ factors, $m\geq 1,n\geq 1$)}.\]
\end{proposition}
\begin{proof}
 If $G$ is an \affinite bigraph with a self binding then we know that $\Ccal(\Gt)$ is a path by Lemma~\ref{lemma:properties}, so let $C_1,C_2,\dots,C_m$ be its connected components with the self binding happening in $C_l$ for some $l\in [m]$. If $l\neq 1,m$ then we immediately get two vertices of degree $3$ in every blue component, so we may assume that $l=1$. By Lemma~\ref{lemma:self_bindings}, all the double bindings are parallel and the result follows.
\end{proof}

\newcommand{\plus}[2]{%
\number\numexpr#1+#2\relax%
}
\newcommand{\po}[1]{%
\number\numexpr#1+1\relax%
}
\newcommand{\mo}[1]{%
\number\numexpr#1-1\relax%
}

\begin{figure}
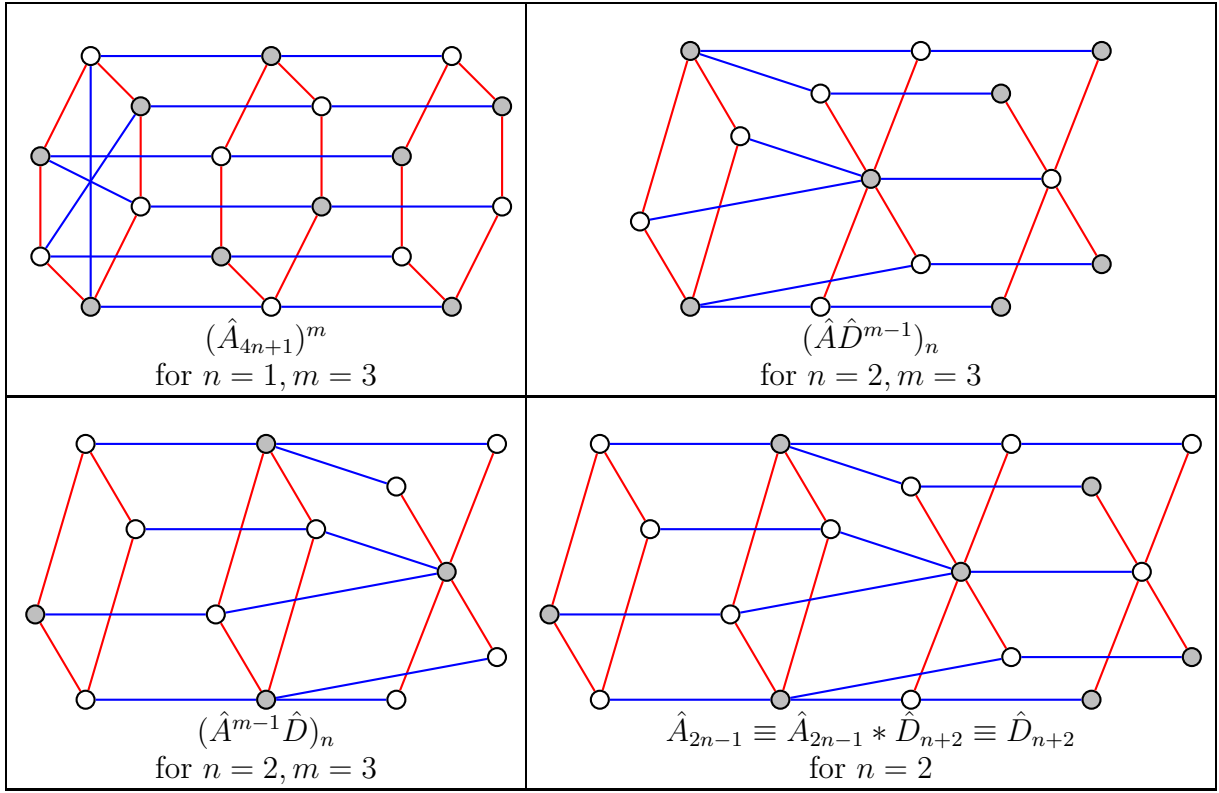

 \begin{tabular}{|c|c|} \hline
  & \\
 {
$
\psmatrix[colsep=0.4cm,rowsep=0.4cm,mnode=circle]
            &   \wh     &           &           &           & \bl       &           &           &           &\wh                       \\
            &           &   \bl     &           &           &           &\wh        &           &           &            &\bl          \\
\bl         &           &           &           &\wh        &           &           &           &\bl                                   \\
            &           &\wh        &           &           &           &\bl        &           &           &            &\wh          \\
\wh         &           &           &           &\bl        &           &           &           &\wh                                   \\
            &\bl        &           &           &           &\wh        &           &           &           &\bl                         
\psset{arrows=-,arrowscale=2}
\foreach \y in {2,6,10}
{
\ra{1,\y}{2,\po{\y}}
\ra{4,\po{\y}}{2,\po{\y}}
\ra{4,\po{\y}}{6,{\y}}
\ra{1,{\y}}{3,\mo{\y}}
\ra{5,\mo{\y}}{3,\mo{\y}}
\ra{5,\mo{\y}}{6,{\y}}
}
\ba{1,2}{1,6}
\ba{1,10}{1,6}
\ba{6,2}{6,6}
\ba{6,10}{6,6}
\ba{3,1}{3,5}
\ba{3,9}{3,5}
\ba{5,1}{5,5}
\ba{5,9}{5,5}
\ba{2,3}{2,7}
\ba{2,11}{2,7}
\ba{4,3}{4,7}
\ba{4,11}{4,7}
\ba{1,2}{6,2}
\ba{3,1}{4,3}
\ba{5,1}{2,3}
\endpsmatrix $} 
&
{$
\psmatrix[colsep=0.4cm,rowsep=0.3cm,mnode=circle]
            &  \bl      &           &           &           &           &\wh        &           &           &            &\bl        \\
            &           &           &           &\wh        &           &           &           &\bl                                 \\
            &           &\wh                                                                                                         \\
            &           &           &           &           &\bl        &           &           &           &\wh                     \\
\wh                                                                                                                                  \\
            &           &           &           &           &           &\wh        &           &           &            &\bl        \\
            &\bl        &           &           &\wh        &           &           &           &\bl                                   
\psset{arrows=-,arrowscale=2}
\ra{1,2}{3,3}
\ra{7,2}{3,3}
\ra{7,2}{5,1}
\ra{1,2}{5,1}
\ra{2,5}{4,6}
\ra{7,5}{4,6}
\ra{1,7}{4,6}
\ra{6,7}{4,6}
\ra{2,9}{4,10}
\ra{7,9}{4,10}
\ra{1,11}{4,10}
\ra{6,11}{4,10}
\ba{1,2}{1,7}
\ba{1,11}{1,7}
\ba{1,2}{2,5}
\ba{2,9}{2,5}
\ba{7,2}{6,7}
\ba{6,11}{6,7}
\ba{7,2}{7,5}
\ba{7,9}{7,5}
\ba{3,3}{4,6}
\ba{5,1}{4,6}
\ba{4,10}{4,6}
\endpsmatrix $} 
 \\
$(\affA_{4n+1})^m$ & $(\affA\affD^{m-1})_n$\\
for $n=1,m=3$      & for $n=2,m=3$\\\hline
 & \\
{$
\psmatrix[colsep=0.4cm,rowsep=0.3cm,mnode=circle]
            &  \wh      &           &           &           &    \bl    &           &           &           &            &\wh        \\
            &           &           &           &           &           &           &           & \wh                                    \\
            &           &\wh        &           &           &           & \wh                                                                   \\
            &           &           &           &           &           &           &           &           &\bl                     \\
\bl         &           &           &           & \wh                                                                                        \\
            &           &           &           &           &           &           &           &           &            &\wh        \\
            &\wh        &           &           &           &\bl        &           &           &\wh                                   
\psset{arrows=-,arrowscale=2}
\ra{1,2}{3,3}
\ra{7,2}{3,3}
\ra{7,2}{5,1}
\ra{1,2}{5,1}
\ra{1,6}{3,7}
\ra{7,6}{3,7}
\ra{7,6}{5,5}
\ra{1,6}{5,5}
\ra{2,9}{4,10}
\ra{7,9}{4,10}
\ra{1,11}{4,10}
\ra{6,11}{4,10}
\ba{1,6}{1,11}
\ba{1,6}{2,9}
\ba{3,7}{4,10}
\ba{5,5}{4,10}
\ba{7,6}{6,11}
\ba{7,6}{7,9}
\ba{1,2}{1,6}
\ba{3,3}{3,7}
\ba{5,1}{5,5}
\ba{7,2}{7,6}
\endpsmatrix $} 
&
{$
\psmatrix[colsep=0.4cm,rowsep=0.3cm,mnode=circle]
            &  \wh      &           &           &           &    \bl    &           &           &           &            &\wh &   &   &   &        \\
            &           &           &           &           &           &           &           & \wh       &            &    &   & \bl             \\
            &           &\wh        &           &           &           & \wh                                                                   \\
            &           &           &           &           &           &           &           &           &\bl         &    &   &   &\wh                  \\
\bl         &           &           &           & \wh                                                                                        \\
            &           &           &           &           &           &           &           &           &            &\wh  &  &  &  & \bl       \\
            &\wh        &           &           &           &\bl        &           &           &\wh        &            &     &  & \bl           
\psset{arrows=-,arrowscale=2}
\ra{1,2}{3,3}
\ra{7,2}{3,3}
\ra{7,2}{5,1}
\ra{1,2}{5,1}
\ra{1,6}{3,7}
\ra{7,6}{3,7}
\ra{7,6}{5,5}
\ra{1,6}{5,5}
\ra{2,9}{4,10}
\ra{7,9}{4,10}
\ra{1,11}{4,10}
\ra{6,11}{4,10}
\ra{2,13}{4,14}
\ra{7,13}{4,14}
\ra{1,15}{4,14}
\ra{6,15}{4,14}
\ba{1,6}{1,11}
\ba{1,6}{2,9}
\ba{3,7}{4,10}
\ba{5,5}{4,10}
\ba{7,6}{6,11}
\ba{7,6}{7,9}
\ba{1,2}{1,6}
\ba{3,3}{3,7}
\ba{5,1}{5,5}
\ba{7,2}{7,6}
\ba{1,11}{1,15}
\ba{2,9}{2,13}
\ba{4,10}{4,14}
\ba{6,11}{6,15}
\ba{7,9}{7,13}
\endpsmatrix $} 
\\
$(\affA^{m-1}\affD)_n$ & $\affA_{2n-1}\equiv\affA_{2n-1}\ast\affD_{n+2}\equiv\affD_{n+2}$\\
for $n=2,m=3$      & for $n=2$\\\hline
 \end{tabular}
\caption{\label{figure:classification}Items (\ref{it:self}), (\ref{it:ADmn}), (\ref{it:AmDn}), and (\ref{it:AADD}) of our classification.}
\end{figure}

\begin{theorem} \label{thm:class}
 Let $G$ be an \affinite $ADE$ bigraph. Then $G$ is isomorphic to exactly one of the following bigraphs:
 \begin{enumerate}
  \item\label{first_infinite} $\affL\otimes \Lambda'$ where $\affL$ and $\Lambda'$ are an affine and a finite $ADE$ Dynkin diagram respectively;
  \item\label{it:self} $\selfbindingname$ ($m\geq 1,n\geq 1$);
  \item\label{it:ADmn} $(\affA\affD^{m-1})_n:=\affA_{2n-1}\ast\affD_{n+2}\equiv\cdots\equiv \affD_{n+2}$ ($m$ factors, $m\geq 2, n\geq 2$);
  \item\label{it:AmDn} $(\affA^{m-1}\affD)_n:=\affA_{2n-1}\equiv\cdots\equiv\affA_{2n-1}\ast\affD_{n+2}$ ($m$ factors, $m\geq 3, n\geq 2$);
  \item $(\affA\affA^{m-1})_n:=\affA_{2n-1}\ast\affA_{4n-1}\equiv\cdots\equiv \affA_{4n-1}$ ($m$ factors, $m\geq 2, n\geq 1$);
  \item $(\affA^{m-1}\affA)_n:=\affA_{2n-1}\equiv\cdots\equiv\affA_{2n-1}\ast\affA_{4n-1}$ ($m$ factors, $m\geq 3, n\geq 1$);
  \item $(\affD\affD^{m-1})_n:=\affD_{n}\ast\affD_{2n-2}\equiv\cdots\equiv \affD_{2n-2}$ ($m$ factors, $m\geq 2, n\geq 4$);
  \item $(\affD^{m-1}\affD)_n:=\affD_{n}\equiv\cdots\equiv\affD_{n}\ast\affD_{2n-2}$ ($m$ factors, $m\geq 3, n\geq 4$);
  \item $\affE\affE^{m-1}:=\affE_{6}\ast\affE_{7}\equiv\cdots\equiv \affE_{7}$ ($m$ factors, $m\geq 2$);
  \item $\affE^{m-1}\affE:=\affE_{6}\equiv\cdots\equiv\affE_{6}\ast\affE_{7}$ ($m$ factors, $m\geq 3$);
  \item\label{it:AADD} $\affA_{2n-1}\equiv\affA_{2n-1}\ast\affD_{n+2}\equiv\affD_{n+2}$ ($n\geq 2$);
  \item $\affA_{2n-1}\equiv\affA_{2n-1}\ast\affA_{4n-1}\equiv\affA_{4n-1}$ ($n\geq 1$);
  \item $\affD_n\equiv \affD_{n}\ast\affD_{2n-2}\equiv \affD_{2n-2}$ ($n\geq 4$);
  \item $\affA_{2n-1}\ast\affA_{6n-1}$;
  \item\label{last_infinite} $\affD_{n+1}\ast\affD_{3n-1}$;
  \item\label{item_E6E6E7E7} $\affE_6\equiv\affE_{6}\ast\affE_{7}\equiv \affE_{7}$;
  \item\label{item_scf3} $\affA_3\ast\affD_5$;
  \item $\affD_6\ast\affE_7$ ;
  \item $\affD_4\ast\affE_6$.
 \end{enumerate}
\end{theorem}
Note that the infinite families are (\ref{first_infinite})-(\ref{last_infinite}), so there are $15$ infinite families and $4$ exceptional bigraphs. Please see Figure~\ref{figure:classification} for examples.
\begin{proof}
 By Proposition~\ref{prop:self_binding}, we may assume that $G$ has no self bindings. If all the double bindings in $G$ are parallel then $G$ is a tensor product. Otherwise consider the unique double binding $C_l\ast C_{l+1}$ of $G$. If it has scaling factor $3$ then all of its components are of type either $A_5$ or $D_4$ by Proposition~\ref{prop:scf_types}, so it is clear that adding an edge to all vertices of the same color in $A_5$ or in $D_4$ does not produce a finite $ADE$ Dynkin diagram (in fact, it always produces an affine $ADE$ Dynkin diagram). Therefore if the scaling factor is $3$ then $m=2$ and $G$ is just the double binding itself. If the scaling factor is $2$ then all the blue components are of type $A_3$ so obviously either $l=1$ or $l=m-1$ and $m$ is arbitrary, or $l=2$ and $m=4$, and the theorem follows.
\end{proof}


{\Large\part{$T$-systems of type $A \otimes \affA$.}}

\section{Variables via domino tilings}\label{sec:variables}

\subsection{Speyer's formula}

In \cite{Sp} Speyer gives the following formula for variables of the octahedron recurrence. One can think of it as the $T$-system associated with the tensor product of type $A_{\infty} \otimes A_{\infty}$, where $A_\infty$ is the infinite path graph and the tensor product is in the sense of Definition~\ref{dfn:tensor_product}, which works the same way for infinite graphs.

Let $\Zz_v(t)$ be the Aztec diamond of radius $t$ centered at vertex $v$. By abuse of notation, denote $\Zz_v(t)$ also the set of vertices of $\Zz_v(t)$ that are not its outer corners. Let $\D$ be a domino tiling of $\Zz_v(t)$. Each domino has a {\it {cut edge}} that separates its two halves. 
Let $G_{\D}$ be the graph obtained by taking $\Zz_v(t)$ as the set of vertices, and the set of all cut edges in $\D$ as the set of edges. For a vertex $u \in \Zz_v(t)$ let $d_{\D}(u)$ be the degree of vertex $u$ in graph $G_{\D}$. It is easy to see that each $d_{\D}(u)$ 
can take values $0,1,2$ only. 

\begin{theorem} \cite{Sp}
 The formula for the variable $T_v(t)$ in an $A_{\infty} \otimes A_{\infty}$ $T$-system is as follows:
 $$T_v(2t+1) = T_v(2t+2) = \sum_{\D} \prod_{u \in \Zz_v(2t+1)} u^{1 - d_{\D}(u)},$$ 
 $$T_v(-2t) = T_v(-2t-1) = \sum_{\D} \prod_{u \in \Zz_v(2t)} u^{1 - d_{\D}(u)},$$ 
  where $t \geq 0$ and the sum is taken over all domino tilings $\D$ of $\Zz_v(2t+1)$ and of $\Zz_v(2t)$, respectively. 
\end{theorem}

\begin{example}
 In Figure \ref{fig:sadd1} we see an example of an Aztec diamond $\Zz_v(2)$, its domino tiling $\D$, and the associated graph $G_{\D}$. The Laurent monomial this tiling contributes is $\frac{a l n}{c m}$, which is easily seen to be one of the monomials in 
 $$T_v(-2) = \frac{\frac{ev+bl}{f} \frac{vk+dn}{h} + \frac{av+bd}{c} \frac{ln+vo}{m}}{v} = \frac{evk}{fh}+\frac{edn}{fh}+\frac{blk}{fh}+\frac{bldn}{fhv}+\frac{aln}{cm}+\frac{avo}{cm}+\frac{bdo}{cm}+\frac{bdln}{cmv}.$$
\end{example}

   \begin{figure}
    \begin{center}
\vspace{-.1in}
\scalebox{0.6}{
\input{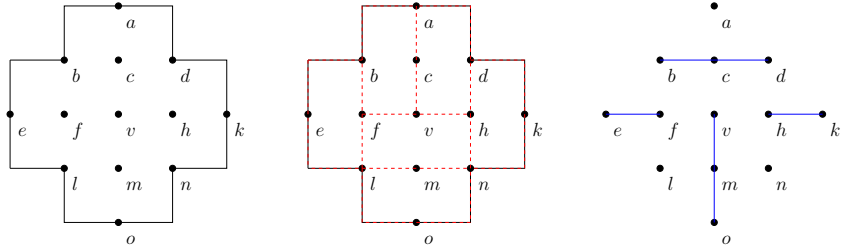} 
}
\vspace{-.1in}
    \end{center} 
    \caption{An Aztec diamond $\Zz_v(2)$, its domino tiling $\D$, and the associated graph $G_{\D}$.}
    \label{fig:sadd1}
\end{figure}

\begin{remark}
 Alternative approaches to giving explicit formulas for the octahedron recurrence can be found in the works of Di Francesco and Kedem \cite{DK1,DK2,DK3,DK4} and Henriques \cite{Hen}. We shall use Speyer's language as the most convenient for our purposes. 
\end{remark}

\subsection{Formula with cylindric boundary conditions}

Consider now the case of $T$-system of type $A_m \otimes \affA_{2n-1}$. The quiver is naturally embedded on a cylinder. Consider the lifting of the quiver to the universal cover of the cylinder, where the vertex variables are periodic. We claim that the following variation of Speyer's theorem 
holds. 

Let $\Zz_v(t)$ now be {\it {the intersection}} of the Aztec diamond of radius $t$ centered at vertex $v$ with the universal cover of the cylinder, where we include two layers of frozen variables with values $1$ on both boundaries. 
An example for $n=2$ and $m=3$ is shown in Figure \ref{fig:sadd2}. For each domino tiling $\D$ of $\Zz_v(t)$ define $G_{\D}$ and $d_{\D}(u)$ as before, but now using the periodicity of variables on the universal cover. 

   \begin{figure}
    \begin{center}
\vspace{-.1in}
\scalebox{0.6}{
\input{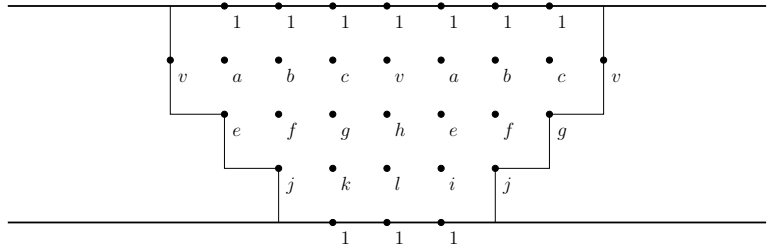} 
}
\vspace{-.1in}
    \end{center} 
    \caption{An example of region $\Zz_v(4)$ on the universal cover of a cylinder with $n=2$ and $m=3$.}
    \label{fig:sadd2}
\end{figure}

\begin{theorem} 
 The formula for the variable $T_v(t)$ in an $A_m \otimes \affA_{2n-1}$ $T$-system is as follows:
 $$T_v(2t+1) = T_v(2t+2) = \sum_{\D} \prod_{u \in \Zz_v(2t+1)} u^{1 - d_{\D}(u)},$$ 
 $$T_v(-2t) = T_v(-2t-1) = \sum_{\D} \prod_{u \in \Zz_v(2t)} u^{1 - d_{\D}(u)},$$ 
  where $t \geq 0$ and the sum is taken over all domino tilings $\D$ of $\Zz_v(2t+1)$ and of $\Zz_v(2t)$, respectively. 
\end{theorem}

\begin{proof}
 We are going to apply Speyer's theorem to the $A_{\infty} \otimes A_{\infty}$ case with variables as shown in Figure \ref{fig:sadd3}. 
    \begin{figure}
    \begin{center}
\vspace{-.1in}
\scalebox{0.6}{
\input{sadd3.pstex_t} 
}
\vspace{-.1in}
    \end{center} 
    \caption{An assignment of variables outside of the universal cover.}
    \label{fig:sadd3}
\end{figure}

There are three logical steps to the proof. First, we claim that as we run the $T$-system dynamics, the Laurent monomials with the minimal power of $\e$ remain the same in the vertices which carry variables $1,\e,\e^2,\ldots$ at the beginning, while at the same time the minimal degree 
of $\e$ in Laurent monomials in the rest of the vertices (i.e. the ones in the middle of the universal cover) is $0$. 
Indeed, let us argue this by induction. Applying a mutation 
at a vertex with minimal Laurent monomial $1$ (i.e. with value $1+O(\e)$), we see that the new value is 
\[\frac{(1+O(\e))(1+O(\e))+(\e+O(\e^2))(O(1)+O(\e))}{1+O(\e)},\]
where $O(\e^k)$ denotes terms with $\e$-degree at least $k$.
It is clear then that specializing at $\e = 0$ we get $1$, which must be then the Laurent monomial with the smallest degree of $\e$ in the result. A similar argument applies in other locations carrying a power of $\e$ at the beginning. 

Next, we claim that plugging in $\e=0$ into the formulas for the $T$-system of type $A_{\infty} \otimes A_{\infty}$ constructed as above returns exactly the formulas for $T$-system of type $A_m \otimes \affA_{2n-1}$.
Again, we can argue this by induction. At the very beginning the claim is obvious. The step is also easy to see from the first claim above. This is because by induction assumption the  exchange relations for $A_{\infty} \otimes A_{\infty}$ $T$-system specialize to exchange relations for $A_m \otimes \affA_{2n-1}$ $T$-system,
and all the powers of $\e$ involved are non-negative. 

Finally, we want to argue that Speyer's formula applied to the above $A_{\infty} \otimes A_{\infty}$ case and specialized at $\e = 0$ indeed returns the formula stated in the theorem. 
For that, we claim that in order for a domino tiling $\D$ to contribute a term with degree of $\e$ equal $0$ (i.e. a term which will not die after specializing) the chunks of Aztec diamonds $\Zz_v(t)$ that are outside of the universal cover need to be tiled with horizontal tiles only. 
Such $\D$-s are then in bijection with the tilings of the part of $\Zz_v(t)$ that is inside the universal cover strip, as desired. 

    \begin{figure}
    \begin{center}
\vspace{-.1in}
\scalebox{0.8}{
\input{sadd4.pstex_t} 
}
\vspace{-.1in}
    \end{center} 
    \caption{An assignment of variables outside of the universal cover.}
    \label{fig:sadd4}
\end{figure}

Let us look at a chunk of $\Zz_v(t)$ that falls outside of the universal cover. Give each potential domino square weight $\e^r$ equal to the larger weight a vertex adjacent to this square has. Considering both ways a domino can be positioned, it is clear that the weight picked up 
by the corresponding edge in $G_{\D}$ is equal to the weight of its squares minus one (see Figure\ref{fig:sadd4}):
$$r + (r-1) = {\color{red}{r} + {r}} -1 \text{ and } r + r = {\color{red}{(r+1)} + {r}} -1.$$
From this it is easy to see that the dominos lying in this chunk can pick up maximal weight of at most the weight of all squares minus potential number of dominos, which is 
$$2 \cdot R + 4 \cdot (R-1) + \dotsc + 2R \cdot 1 - R(R+1)/2,$$ where $\e^R$ is the maximal power of $\e$ in the chunk. On the other hand, the total weight to burn in the chunk is 
$$1 \cdot R + 3 \cdot (R-1) + \dotsc + (2R-1) \cdot 1,$$ which is easily seen to be the same. Thus, in order for the $\e$ to not enter the resulting overall weight picked up by $G_{\D}$ inside the chunk, we need the equality to hold, which happens only if every square in the chunk is covered 
by a domino that lies in this chunk. This happens only when the chunk is tiled by the horizontal dominoes. 
\end{proof}

\section{Boundary affine slices and Goncharov-Kenyon Hamiltonians}

Let us refer to copies of $\affA_{2n-1}$ in $A_m \otimes \affA_{2n-1}$ as {\it {affine slices}}. We will distinguish {\it {boundary affine slices}} which correspond to the two boundary vertices of the Dynkin diagram $A_m$, and {\it {internal affine slices}} which correspond
to the internal vertices of the Dynkin diagram $A_m$. In this section, we identify the recurrence coefficients of boundary affine slices as Goncharov-Kenyon Hamiltonians introduced in \cite{GK}.
We shall see in Section \ref{sec:pleth} that the recurrence coefficients of the internal affine slices can be expressed through the Goncharov-Kenyon Hamiltonians using plethysm of symmetric functions. While we leave the question of an explicit formula for internal affine slices coefficients 
open, we will be able to deduce some of their properties in Sections~\ref{sec:pleth} and~\ref{sec:laur}.

\subsection{Thurston height}

Recall from \cite{Th} the following definition of {\it {Thurston height}} function associated to a domino tiling. Consider a cylinder $\Cc_{m,2n}$ which we can think of as $(m+1) \times 2n$ rectangle with sides of length $m+1$ glued. We can identify the $m \times 2n$ non-boundary nodes with vertices of the quiver $A_m \otimes \affA_{2n-1}$. Fix a chessboard coloring of the cylinder, and fix a node $O$ at its bottom boundary. Let $\D$ be a domino tiling of $\Cc_{m,2n}$. Define the function 
$$h: \text{ nodes of } \Cc_{m,2n} \longrightarrow \mathbb Z$$ as follows:
\begin{itemize}
 \item $h(O)=0$;
 \item if $a \rightarrow b$ is a directed edge of a domino in $\D$ and the cell to the right of it is black, then $h(b) - h(a) = 1$;
  \item if $a \rightarrow b$ is a directed edge of a domino in $\D$ and the cell to the right of it is white, then $h(b) - h(a) = -1$.
\end{itemize}
If there is no cell to the right, we can still decide between the two options by looking at the cell to the left and assuming the cell to the right has an opposite color. 
 
\begin{theorem}
 Thurston height $h$ is a well-defined function on the nodes of $\Cc_{m,2n}$.
\end{theorem}

An example of a domino tiling of $\Cc_{m,2n}$ and the associated Thurston height function can be seen in Figure \ref{fig:sadd5}.

    \begin{figure}
    \begin{center}
\vspace{-.1in}
\scalebox{0.8}{
\input{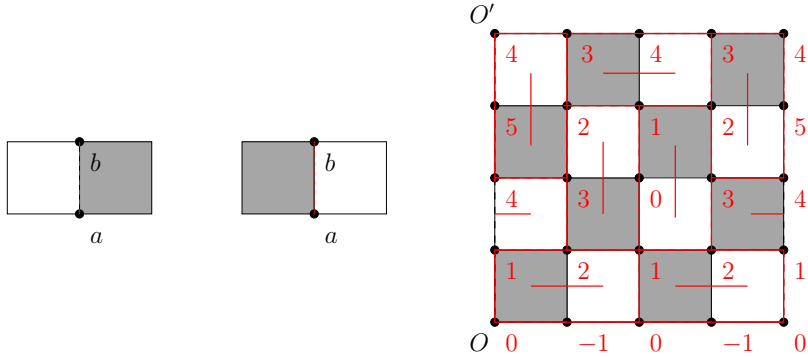} 
}
\vspace{-.1in}
    \end{center} 
    \caption{Two cases to consider when constructing $h$ and an example on $\Cc_{3,4}$.}
    \label{fig:sadd5}
\end{figure}

\begin{proof}
 It is known \cite{Th} that Thurston height function is well-defined for regions in the plane without holes. Thus, it is well-defined on the infinite periodic tiling obtained by lifting $\D$ to the universal cover of $\Cc_{m,2n}$. It remains to argue that this height function 
 is also periodic, and thus can be folded back onto the cylinder. Assume it is not periodic, then it must steadily grow or steadily decline as we circle around the cylinder. However, then it would reach arbitrary high or arbitrary low values, which is impossible since any node is within distance 
 $m$ from the lower boundary, which is filled with $0$-s and $-1$-s. The contradiction implies the desired property. 
\end{proof}

We can now define {\it {the height of a tiling}} $\D$ as  $$h(\D) = h(O') - h(O),$$ where $O'$ is the node on the top boundary component opposite of $O$. 

\begin{prop}
 The function $h(\D)$ takes values $4k$, $k= -(m+1)/2, \ldots, (m+1)/2$ if $m$ is odd, and takes values $4k+1$, $k = -m/2-1, \ldots, m/2$ if $m$ is even. There is a unique tiling having minimal height and a unique tiling having maximal height. 
\end{prop}

\begin{proof}
 As we walk from $O$ to $O'$ straight up, at each step the height changes either by $\pm 1$ or $\pm 3$, depending on whether the step cuts a domino and what the colors on the sides are. The claims of the proposition then easily follow. 
\end{proof}

Let us refer to the tiling with the minimal height as {\it {the see}}, and denote it $\S$. One can give an alternative definition of the height of a tiling $h(\D)$ as follows.  For any tiling $\D$, put $\S$ and $\D$ on the same picture. What we get 
is a {\it {double dimer model}}, where all dominos will split into closed cycles. An example of such superposition for the tiling in Figure \ref{fig:sadd5} is given in Figure \ref{fig:sadd6}. The dominos of the sea $\S$ are shown in blue.   

    \begin{figure}
    \begin{center}
\vspace{-.1in}
\scalebox{0.8}{
\input{sadd6.pstex_t} 
}
\vspace{-.1in}
    \end{center} 
    \caption{The superposition of $\S$ and $\D$.}
    \label{fig:sadd6}
\end{figure}
The cycles created in the process may include contractible cycles and non-contractible cycles. Let us refer to the latter as {\it {hula hoops}}. Note that the contractible cycles may be just double edges, if $\S$ and $\D$ share dominos. 
The example in Figure \ref{fig:sadd6} has zero contractible cycles and three hula hoops. Denote $\hh(\D)$ the number of hula hoops created by superposing $\D$ and $\S$.

\begin{prop}
 We have $$h(\D) = h(\S) + 4 \hh(\D).$$
\end{prop}

\begin{proof}
 One can always walk from $O$ to $O'$ so that the only steps that cross dominos, rather than follow their boundaries, are the ones crossing the hula hoops. It is easy to see that each such crossing is responsible for a difference of $4$ between the accumulated parts of 
 $h(\S)$ and $h(\D)$. Furthermore, it is easy to see that since $\S$ is the tiling with minimal height, each such crossing must make $h(\D)$ larger by $4$ than $h(\S)$, as opposed to smaller. Otherwise we could change $\S$ by using the dominos of $\D$ from the hula hoop, and decrease 
 its height even further, which is impossible. The proposition claim follows. 
\end{proof}

\subsection{The recurrence}

Define {\it {Goncharov-Kenyon Hamiltonians}} to be the sums 
$$H_r = \sum_{\hh(\D) = r} \prod_{u \in \Cc_{m,2n}} u^{1 - d_{\D}(u)},$$
where $d_{\D}(u)$ is as before the degree of $u$ in the associated graph $G_{\D}$ on the cylinder, and the sum is taken over all tilings $\D$ of height $4r + h(\S)$.
Here on the boundary we always have $u=1$, which makes $H_r$-s into functions of variables at the vertices of the quiver $A_m \otimes \affA_{2n-1}$.

\begin{example}\label{example:ham}
 Take $m=3$ and $n=1$. We have six variables $a,b,c,d,e,f$ at the vertices of the quiver $A_3 \otimes \affA_{1}$. Figure \ref{fig:sadd7} shows the domino tilings contributing to $H_1$ and the monomials they contribute.  
     \begin{figure}
    \begin{center}
\vspace{-.1in}
\scalebox{0.6}{
\input{sadd7.pstex_t} 
}
\vspace{-.1in}
    \end{center} 
    \caption{The tilings $\D$ with height $h(\D)=-4$.}
    \label{fig:sadd7}
\end{figure}
As a result, we find 
$$H_1 = \frac{ab}{de}+\frac{a}{be}+\frac{b}{ad}+\frac{c}{f}+\frac{d}{a}+\frac{ef}{bc}+\frac{e}{cf}.$$
Similarly we find 
$$H_2 = \frac{abc}{def}+\frac{ab}{dcf}+\frac{bc}{adf}+\frac{ac}{be}+\frac{be}{acdf}+\frac{ef}{ad}+\frac{af}{cd}+\frac{e}{b}+\frac{cd}{af}+\frac{de}{acf}+\frac{b}{e}+\frac{df}{be}+\frac{def}{abc}, \text{ and }$$
$$H_3 = \frac{bc}{ef}+\frac{c}{be}+\frac{b}{cf}+\frac{a}{d}+\frac{f}{c}+\frac{de}{ab}+\frac{e}{ad}.$$
 The only tiling contributing to $H_0$ is the sea $\S$, and it is easy to see that $$H_0 = H_4 =1.$$
\end{example}

Let $v'$ be the vertex diametrically opposite to $v$ on the same affine slice. Let 
$$v(j) = 
\begin{cases}
v & \text{if $j$ is even;}\\
v' & \text{if $j$ is odd.}
\end{cases}
$$
We are ready to state the main theorem of the section. 

\begin{theorem} \label{thm:rec}
 For any vertex $v$ on the top boundary affine slice of the quiver $A_m \otimes \affA_{2n-1}$ the $T$-system satisfies for any $t$ the following recursion
 $$T_{v(0)}(t+(m+1)n) - H_1 T_{v(1)}(t+mn) + \dotsc \pm H_{m} T_{v(m)}(t+n) \mp T_{v(m+1)}(t) = 0.$$
 Similarly, for any vertex $v$ on the bottom boundary affine slice and any $t$ we have 
  $$T_{v(0)}(t+(m+1)n) - H_m T_{v(1)}(t+mn) + \dotsc \pm H_{1} T_{v(m)}(t+n) \mp T_{v(m+1)}(t) = 0.$$
\end{theorem}


\section{Proof of the recurrence}

In this section we prove Theorem \ref{thm:rec}.
We consider the case of $v$ lying in the top affine slice of the quiver $A_m \otimes \affA_{2n-1}$. The case of the bottom affine slice is similar. 

As we have seen, the Laurent monomials entering both the $H_i$-s and the $T_v$-s have an interpretation in terms of weights of domino tilings. We are going to construct an involution which associates 
each Laurent monomial in the expansion by linearity of $H_i T_{v(i)}(t+(m+1-i)n)$ to an equal Laurent monomial in either $H_{i-1} T_{v(i-1)}(t+(m+2-i)n)$ or $H_{i+1} T_{v(i+1)}(t+(m-i)n)$. 
This implies that all of the terms cancel out as desired, since thus 
created pairs of Laurent monomials are equal but have opposite signs.  

Let $\D_{\Cc}$ be a domino tiling of the cylinder, contributing a term into $H_i$. Let $H_{\Cc}$ be the topmost among $i$ hula hoops created by superposing $\D_{\Cc}$ with the sea $\S$.  

Let $\Zz_{v(i)}(t+(m+1-i)n)$ be the fragment of an Aztec diamond lying inside the universal cover, as defined above. Let $\D_{\Zz}$ be a domino tiling of $\Zz_{v(i)}(t+(m+1-i)n)$, contributing a term into 
$T_{v(i)}(t+(m+1-i)n)$. Superpose $\D_{\Zz}$ with the universal cover of the sea $\S$, which is a tiling of the universal cover of the cylinder. Consider the part of the result that intersects $\Zz_{v(i)}(t+(m+1-i)n)$.

\begin{lem}
 The resulting double dimer contains a single chain of dominos, called {\it {the hose}}, connecting the top left to the top right cells of $\Zz_{v(v)}(t+(m+1-i)n)$. The rest is filled with pairs of dominos that are shared by 
 $\D_{\Zz}$ and the universal cover of $\S$.
\end{lem}

\begin{example}
 In Figure \ref{fig:sadd10} an example is presented of a superposition of $\D_{\Zz}$, shown in red, with the universal cover of $\S$, shown in blue. 
      \begin{figure}
    \begin{center}
\vspace{-.1in}
\scalebox{0.6}{
\input{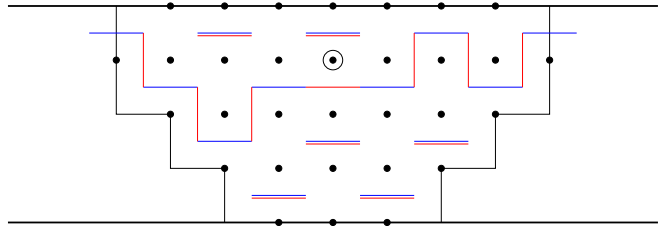} 
}
\vspace{-.1in}
    \end{center} 
    \caption{An example of superposition of $\D_{\Zz}$ (red) and the universal cover of $\S$ (blue).}
    \label{fig:sadd10}
\end{figure}
Here $m=3$, $n=2$, the vertex $v(2)=v$ is circled and $\Zz_v(4)$ is shown. One can clearly see the hose, while the rest of the dominos form $2$-cycles.
\end{example}

\begin{proof}
 It is easy to see that the resulting double dimer in $\Zz_{v(i)}(t+(m+1-i)n)$ must consist of exactly one path and several cycles. This is because there are only two places where it crosses the boundary of $\Zz_{v(i)}(t+(m+1-i)n)$,
 thus those two places must be the ends of the path, i.e. the hose. To see why all cycles must have length $2$  observe that the sea always flows in the same direction once you start crossing between its dominos, and thus you can never really 
 turn around to form a long cycle. 
\end{proof}

Now we are ready to define the involution. Assume we are given a pair $(\D_{\Cc}, \D_{\Zz})$ with corresponding Laurent monomials contributing to the product $H_i T_{v(i)}(t+(m+1-i)n)$. Take the hose associated with $\D_{\Zz}$ and start following 
its edges {\it {on the cylinder}} $\Cc_{m,2n}$. One of the two events is going to occur:
\begin{itemize}
 \item either the hose wrapping around $\Cc_{m,2n}$ will intersect itself first, without intersecting the hula hoops of $\D_{\Cc}$; or 
 \item the hose will intersect the top hula hoop $H$ of $\D_{\Cc}$ before intersecting itself. 
\end{itemize}

In the first case, take the first such self-intersection, and {\it {extract}} from it the corresponding hula hoop. By this we mean cut out from the hose the dominos of the part between endpoints of self-intersection, and add the corresponding red dominos 
to $\D_{\Cc}$ instead of the blue ones it is currently using. In the second case, take the first such intersection with $H$ and {\it {insert}} $H$ to extend the hose, by pasting it at this first point of intersection. We then remove the red edges 
of $H$ from $\D_{\Cc}$, substituting the blue sea edges instead. In either case we get a new pair $(\D_{\Cc}', \D_{\Zz}')$.

\begin{lem}
 The resulting pair is a well-defined pair of domino tilings that contributes either to $H_{i-1} T_{v(i-1)}(t+(m+2-i)n)$ or to $H_{i+1} T_{v(i+1)}(t+(m-i)n)$, depending on which of the two events occurred. 
\end{lem}

\begin{example}
 An example of a pair  $(\D_{\Cc}, \D_{\Zz})$ contributing to $H_1 T_{v(1)}(4)$ for which the first event occurs is shown in Figure \ref{fig:sadd11}. 
The new pair $(\D_{\Cc}', \D_{\Zz})'$
       \begin{figure}
    \begin{center}
\vspace{-.1in}
\scalebox{0.6}{
\input{sadd11.pstex_t} 
}
\vspace{-.1in}
    \end{center} 
    \caption{The effect of the first event on a pair $(\D_{\Cc}, \D_{\Zz})$.}
    \label{fig:sadd11}
\end{figure}
in this case contributes to $H_{i+1} T_{v(i+1)}(t+(m-i)n) = H_2 T_{v(2)}(2)$. 

 An example of a pair  $(\D_{\Cc}, \D_{\Zz})$ contributing to $H_1 T_{v(1)}(4)$ for which the second event occurs is shown in Figure \ref{fig:sadd12}. 
The new pair $(\D_{\Cc}', \D_{\Zz})'$
       \begin{figure}
    \begin{center}
\vspace{-.1in}
\scalebox{0.6}{
\input{sadd12.pstex_t} 
}
\vspace{-.1in}
    \end{center} 
    \caption{The effect of the second event on a pair $(\D_{\Cc}, \D_{\Zz})$.}
    \label{fig:sadd12}
\end{figure}
in this case contributes to $H_{i-1} T_{v(i-1)}(t+(m+2-i)n) = H_0 T_{v(0)}(6)$. 

In both cases, the fragments that get either extracted or inserted are circled by a green dashed line. 
\end{example}

\begin{proof}
 The only somewhat non-trivial part of the claim is why after a hula hoop is extracted from a hose, what remains is still a proper hose. The reason is that all blue dominos in the hose flow East, which means that the red dominos must flow
 North, East or South, but not West. This means that the red and the blue dominos that we need to connect after the extraction are compatible. 
\end{proof}

The final claim we need to conclude the theorem is the following.

\begin{lem}
 This map is a weight-preserving involution on Laurent monomials.
\end{lem}

\begin{proof}
 If the first event occurred in $(\D_{\Cc}, \D_{\Zz})$ and $(\D_{\Cc}', \D_{\Zz}')$ was created, then the second event occurs in $(\D_{\Cc}', \D_{\Zz}')$ at exactly the same place, and $(\D_{\Cc}, \D_{\Zz})$ is created. Same holds vice versa.
 Thus, the map is an involution. The fact that it is weight preserving is easy to see from the way we assign weights to domino tilings. 
\end{proof}

\section{Affine slices and plethysm} \label{sec:pleth}

In this section, we explain how to express the recurrence coefficients of the  affine slices in $A_m \otimes \affA_{2n-1}$ through the Goncharov-Kenyon Hamiltonians $H_i$. Note that Theorem \ref{thm:rec} did not quite answer that yet for the boundary slices, since it involved 
variables at {\it {two}} vertices $v$ and $v'$, rather than a single $v$. We rely here on results of \cite{P}, as well as on the language of tensors introduced there.

Recall that in \cite{P} the $T$-system variables are interpreted as certain polynomial $SL_{m+1}$-invariants of a collection of $2n$ vectors in $\mathbb C^{m+1}$ and one matrix $A \in SL_{m+1}$.
The key theorem is the following strengthening of \cite[Theorem 1.11]{P}.

\begin{thm}
 The variables on the $r$-th slice of  $A_m \otimes \affA_{2n-1}$, $r=1, \ldots, m$ satisfy the same recurrence as the exterior powers $\wedge^r(\hat A^q)$, $q \in \mathbb Z$, where $\hat A = A^2$.
\end{thm}

In particular, according to the Cayley-Hamilton theorem, the recurrence for $r=1$ is given by the characteristic polynomial of $\hat A$.  

\begin{proof}
 As $q$ grows, we keep repeating the Dehn twists, which inserts $\hat A \otimes \dotsc \otimes \hat A$ into the tensor. Thus, we obtain the tensor $\hat A^q \otimes \dotsc \otimes \hat A^q$ in the middle. Furthermore, this tensor is attached to the anti-symmetrizing Levi-Cevita 
 tensor, which results in the anti-symmetrization of $\otimes^r(\hat A^q)$, which is $\wedge^r(\hat A^q)$.
\end{proof}

\begin{cor} \label{cor:pleth}
 The recurrence coefficients of the affine slices $r=1, \ldots, m$ are expressed in terms of the Goncharov-Kenyon Hamiltonians $H_i$ as plethysms of elementary symmetric functions and the power sum symmetric function $e_j[e_r[p_2]]$ are expressed through the original elementary symmetric functions $e_i=H_i$.
\end{cor}

\begin{proof}

Theorem \ref{thm:rec} tells us the recurrence satisfied by the sequence of $T_v$-s and $T_{v'}$-s. To obtain the recurrence satisfied by $T_v$-s only, we need to take every second term of the sequence. In terms of the recurrence, this means we just need to square the roots of the recurrence polynomial.
This means that if $H_i = e_i(\lambda)$, then the coefficients on boundary levels are just the plethysms $e_i[p_2]$, which of course can be expressed as polynomials in the $H_i$-s. 

 Since in the construction of the ring of invariants in \cite{P} the dimension count forces the vectors and the matrix $A$ to be generic, the $r=1$ affine slice cannot satisfy any linear recurrence of length shorter than $2n(m+1)$. This means that any two such linear recurrences must coincide, 
 and thus the plethysms $e_i[p_2]$ of Goncharov-Kenyon Hamiltonians $H_i$ are the coefficients of the characteristic polynomial of $\hat A$.
 
Now, if $\lambda_i$, $i=1, \ldots, m+1$ are the eigenvalues of $\hat A$, then products $$\lambda_{i_1} \dotsc \lambda_{i_r}, \;\;1 \leq i_1 < \dotsc < i_r \leq m+1$$ are the eigenvalues of $\wedge^r(\hat A)$. Then the coefficients of the corresponding characteristic polynomial of 
$\wedge^r(\hat A^q)$ are exactly the plethysms $e_j[e_r[p_2(\lambda)]].$
\end{proof}

\begin{cor}\label{cor:conserved}
 The Goncharov-Kenyon Hamiltonians $H_i$ are conserved quantities of the $T$-system. 
\end{cor}

\begin{proof}
 As in the previous proof, the minimal recurrence satisfied by the boundary affine slice is unique, and thus its coefficients are the same no matter which moment we pick as $t=0$.
\end{proof}

\begin{cor} \label{cor:pleth2}
 The recurrence for the $r$-th affine slice has the form 
 $$T_v \left(t+2n{{m+1} \choose r}\right) - \dotsc \pm T_v(t)=0$$ 
 with exactly ${{m+1} \choose r}+1$ terms on the left. 
\end{cor}

\begin{proof}
 This is clear since the size of $\wedge^r(\hat A)$ is ${m+1} \choose r$.
\end{proof}

\begin{cor} \label{cor:pal}
 The recurrence coefficients of $r$-th and $(m+1-r)$-th affine slices are the same up to the reversal of the order.  If $m$ is odd, then the coefficients of the slice $r=(m+1)/2$ are palindromic.
\end{cor}

\begin{proof}
 Since $\hat A \in SL_{m+1}$, we know that the constant term of the characteristic polynomial is $1$. Alternatively, we have already seen that $H_{m+1}=1$. Either way, we see that $\prod_{i=1}^{m+1} \lambda_i =1$. This means that the eigenvalues of $\wedge^r(\hat A^q)$ and of 
 $\wedge^{m+1-r}(\hat A^q)$ are inverses of each other, and the claim follows. 
\end{proof}

\begin{example}
Consider the case $m=3$. In this case we have $3$ affine slices, two boundary and one internal. The recurrence relations satisfied by the $T$-system are as follows.
\begin{itemize}
 \item If $v$ lies on the $r=1$ affine slice, 
$$T_v(t+8n) - (H_1^2-2H_2) T_v(t+6n) + (H_2^2-2H_1H_3+2) T_v(t+4n) - (H_3^2-2H_2) T_v(t+2n) + T_v(t) = 0.$$
 
  \item If $v$ lies on the $r=2$ affine slice, 
$$T_v(t+12n) - (H_2^2-2H_1H_3+2) T_v(t+10n) + ((H_1^2-2H_2) (H_3^2-2H_2) - 1) T_v(t+8n) -$$ $$((H_1^2-2H_2)^2+(H_3^2-2H_2)^2-2(H_2^2-2H_1H_3+2)) T_v(t+6n)$$ $$+ ((H_1^2-2H_2) (H_3^2-2H_2) - 1) T_v(t+4n) - (H_2^2-2H_1H_3+2) T_v(t+2n) + T_v(t)= 0.$$
 
  \item If $v$ lies on the $r=3$ affine slice, 
$$T_v(t+8n) - (H_3^2-2H_2) T_v(t+6n) + (H_2^2-2H_1H_3+2) T_v(t+4n) - (H_1^2-2H_2) T_v(t+2n) + T_v(t) = 0.$$
\end{itemize}

Here for example $(H_1^2-2H_2)^2+(H_3^2-2H_2)^2-2(H_2^2-2H_1H_3+2)$ is determined by the plethysm 
$$e_3[e_2(\lambda_1, \lambda_2, \lambda_3, \lambda_4)] = e_3(\lambda_1 \lambda_2, \lambda_1 \lambda_3,  \lambda_1 \lambda_4, \lambda_2 \lambda_3, \lambda_2 \lambda_4, \lambda_3 \lambda_4) =$$ 
$$=(\lambda_1^3\lambda_2\lambda_3\lambda_4+\dotsc) + (\lambda_1^2+\lambda_2^2+\lambda_3^2+\dotsc) + 2 (\lambda_1^2\lambda_2^2\lambda_3\lambda_4+\dotsc) = e_4e_1^2+e_3^2-2e_2e_4 = e_1^2+e_3^2-2e_2,$$
 followed by the plethysm $e_1[p_2] = e_1^2-2e_2$, $e_2[p_2] = e_2^2-2e_1e_3+2$, $e_3[p_2]=e_3^2-2e_2$. Throughout we use $e_4=1$.
\end{example}

\section{Laurent property and positivity} \label{sec:laur}

Recall that the {\it {upper cluster algebra}} $\mathfrak U_{A}$ associated with a cluster algebra $A$ is the algebra of all elements of the fraction field of $A$ that can be expressed as Laurent polynomials in {\it {any}} cluster of $A$. Due to Laurent property 
of cluster algebras \cite{FZ} we know that $\mathfrak U_{A} \subseteq A$. The equality holds in some cases, while in other cases $\mathfrak U_{A}$ is strictly larger. We refer the reader to \cite{FZ3} for a rigorous definition and properties of upper cluster algebras. 

\begin{thm} \label{thm:upper}
 Goncharov-Kenyon Hamiltonians $H_i$ are elements of the upper cluster algebra associated with the quiver $A_m \otimes \affA_{2n-1}$. 
\end{thm}

Of course, the $H_i$-s are Laurent expressions in terms of the initial cluster of this $T$-system by definition. Since we know they are conserved quantities, the same holds for any cluster in the $T$-system. However, the claim of the theorem is much stronger, since the $T$-system represents only one 
way to mutate the quiver, while the Laurentness is true for {\it {any}} such way. 

\begin{cor} \label{cor:upper}
 The coefficients of recurrence polynomials of all vertices of $A_m \otimes \affA_{2n-1}$ lie in the upper cluster algebra.
\end{cor}

\begin{proof}
 Since those coefficients are polynomials in the $H_i$-s by Corollary \ref{cor:pleth}, the statement follows. 
\end{proof}

Sherman and Zelevinsky \cite{ShZ} have defined a {\it {positive cone}} inside the upper cluster algebra $\mathfrak U_A$ to be the subset of all elements of $\mathfrak U_A$ that are expressible as {\it {positive}} Laurent expression in any cluster of $A$.

\begin{conjecture} \label{conj:cone}
 Goncharov-Kenyon Hamiltonians $H_i$ are elements of the positive cone of the corresponding upper cluster algebras. 
\end{conjecture}

Again, by definition the $H_i$-s are positive in terms of the clusters along time evolution of the $T$-system, but the claim of the conjecture is much stronger. 


\subsection{Proof of Theorem \ref{thm:upper}}

We are going to use the standard technique invented in \cite[Theorem 1.5]{FZ3}. Specifically, to know that certain $H_i$ lies in the upper cluster algebra, it suffices to check 
the Laurent
condition with respect to some seed together with all
the seeds obtained from it by a single mutation.
The fact that the $H_i$-s are positive in the initial seed of the $T$-system is true by definition. Thus, it remains to check positivity in all seeds obtained by mutating just a single variable in the initial seed.

Let $v$ be the variable that is mutated, and assume the surrounding variables are as in Figure \ref{fig:sadd16}.  
       \begin{figure}
    \begin{center}
\vspace{-.1in}
\scalebox{0.8}{
\input{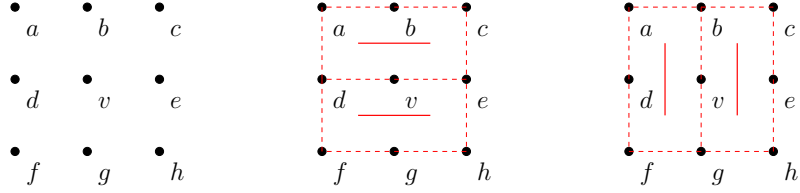} 
}
\vspace{-.1in}
    \end{center} 
    \caption{The two possible ways to get $v$ in the denominator of the corresponding monomial.}
    \label{fig:sadd16}
\end{figure}
Note that some of the variables may be equal to $1$ if $v$ is close to the boundary. 

When we mutate at $v$, we make a substitution $$v \longleftarrow \frac{bg+de}{v'}.$$
Let us consider the effect of this substitution on the Laurent monomials entering $H_i$, which as we know correspond to domino tilings $\D$:
$$H_i = \sum_{\hh(\D) = i} \prod_{u \in \Cc_{m,2n}} u^{1 - d_{\D}(u)}.$$
For each tiling $\D$ where $v$ does not appear at all in the monomial, i.e. where $d_{\D}(u)=1$, or where $v$ appears in the numerator, i.e. $d_{\D}(u)=0$, the Laurentness is not violated by the substitution $v \longleftarrow \frac{bg+de}{v'}.$
Thus, it remains to consider the terms where $v$ appears in the denominator, i.e. $d_{\D}(u)=2$. 

The key observation is that such tilings $\D$ come in pairs. This is because locally around vertex $v$ they need to look in one of the 
two ways shown in Figure \ref{fig:sadd16}. Furthermore, the local move swapping between those two ways to tile the surrounding $2 \times 2$ square does not change the height of the tiling. Thus, all tilings $\D$ contributing to the terms of $H_i$ with $v$ in the denominator 
indeed come in pairs, differing by the application of this local $2 \times 2$ square swap. Let $\D$ and $\D'$ be such a pair. Then  
$$\prod_{u \in \Cc_{m,2n}} u^{1 - d_{\D}(u)} + \prod_{u \in \Cc_{m,2n}} u^{1 - d_{\D'}(u)}= \left(\frac{bg}{v} + \frac{de}{v} \right) b^{-d_{\D}(b)} g^{-d_{\D}(g)} \prod_{u \in \Cc_{m,2n}, u \not = b,v,g} u^{1 - d_{\D}(u)}.$$
We see that after the substitution this becomes 
$$\left[\left(\frac{bg}{v} + \frac{de}{v} \right) b^{-d_{\D}(b)} g^{-d_{\D}(g)} \prod_{u \in \Cc_{m,2n}, u \not = b,v,g} u^{1 - d_{\D}(u)}\right]_{v \longleftarrow \frac{bg+de}{v'}} =$$ $$= v' b^{-d_{\D}(b)} g^{-d_{\D}(g)} \prod_{u \in \Cc_{m,2n}, u \not = b,v,g} u^{1 - d_{\D}(u)},$$
which is a Laurent expression. The statement follows. \qed

\begin{remark}
 Both Theorem~\ref{thm:upper} and Corollary~\ref{cor:conserved} can be deduced directly from Urban Renewal Theorem, see, for example, \cite[Section~5.2]{Sp}.
\end{remark}

\section{Conjectures}\label{sec:geometric_conjectures}

Let $v$ be any vertex of any of the quivers in the classification of Theorem \ref{thm:class}. We formulate here conjectures generalizing the results of this paper from $A_m \otimes \affA_{2n-1}$, which is a special case of the first family, 
to {\it {all}} families from our classification. 

The following conjecture generalizes Corollaries \ref{cor:pleth} and \ref{cor:pleth2}. In light of Theorem \ref{thm:class} this conjecture is a stronger version of \cite[Conjecture 1.9]{P}.

\begin{conjecture}\label{conj:linearizable}
 For every vertex $v$, there exist numbers $i$ and $N$ and rational functions $J_0 = 1, J_1, \dots, J_{N}, J_{N+1}=1$ in $\Q(\x)$ such that
 \begin{itemize}
  \item The $J_k$-s are the conserved quantities of the $T$-system;
  \item For any $t$ we have $$J_0 T_v(t+i(N+1)) - J_1 T_v(t+iN) + \dotsc \pm J_{N+1} T_v(t) = 0.$$
 \end{itemize}
\end{conjecture}

Note that among the linear recurrences satisfied by the sequences there is a {\it {minimal one}}. This is because if two recurrences are satisfied, then so is one given by the greatest common divisor of their characteristic polynomials. Let us from now on assume that the choices of $i,N$ and $J_k$-s 
are made so that the resulting recurrence is minimal. 

The following conjecture generalizes Theorem \ref{thm:upper} and Corollary \ref{cor:upper}.

\begin{conjecture}
The $J_k$-s belong to the upper cluster algebra of the cluster algebra associated with $Q$. In particular, they are Laurent polynomials in variables at any moment $t$. 
\end{conjecture}

The following conjecture generalizes Conjecture \ref{conj:cone}.

\begin{conjecture} 
The $J_k$-s are {\emph {positive}} Laurent expressions in terms of any cluster of the cluster algebra. In other words, they are elements of the {\emph {positive cone}} inside the upper cluster algebra, as defined by Sherman and Zelevinsky \cite{ShZ}. 
\end{conjecture}

\def\H{\mathcal{J}}
Our next conjecture is open even in Type $\affA_{2n-1}\otimes A_m$. Let $Q$ be any \affinite quiver and let the $J_k$-s be as above. Consider an infinite Toeplitz matrix $\H = \H(m,n)$ where the entries are defined as follows:
$$
\H_{i,j} = 
\begin{cases}
J_{j-i} & \text{if $0 \leq j-i \leq m+1$;}\\
0 & \text{otherwise.}
\end{cases}
$$

\begin{conjecture} \label{conj:tp}
 All minors of $\H$ are either identically $0$ or positive Laurent polynomials in $\x$.
\end{conjecture}

In other words, we conjecture that the $J_i$-s form a {\it {totally positive sequence}}, or {\it {P\'olya frequency sequence}}, see \cite{Br} for the background. We also state the following weaker version of Conjecture~\ref{conj:tp}:

\begin{conjecture} \label{conj:real}
 The roots of the recurrence polynomial
 $$J(z) = z^{N+1}-J_1 z^N + \dotsc \pm J_N z \mp 1$$
 are positive real numbers. 
\end{conjecture}

%
%

Each of the types $A_m$, $D_m$ and $E_6$ has a canonical involution on the Dynkin diagram, sending the diagram to itself. Denote this involution $\eta$. Assume our $T$-system is of the tensor product type, and more specifically of the form $\Lambda' \otimes \hat \Lambda$, where $\Lambda'$ 
is a finite type Dynkin diagram of type $A_m$, $D_{2m+1}$ or $E_6$, and $\hat \Lambda$ is an arbitrary extended Dynkin diagram.
Let $v'$ be the vertex of $\Lambda' \otimes \hat \Lambda$ having the same $\hat \Lambda$ coordinate, but whose $\Lambda'$ coordinate is obtained from that of $v$ via involution $\eta$. The following conjecture generalizes Corollary \ref{cor:pal}.

\begin{conjecture} \label{conj:symm1}
The recurrence polynomials of $v$ and $v'$ have the same coefficients but in the opposite order. In particular, if $v=v'$, then the recurrence polynomial $J(z)$ is palindromic. 
\end{conjecture}

Assume now we are in {\it {any other case}}, i.e. either our $T$-system belongs to a different family of the classification, or it is a tensor product but $\Lambda'$ is not of type $A_m$, $D_{2m+1}$ or $E_6$. 
The following conjecture again generalizes Corollary \ref{cor:pal}.

\begin{conjecture}  \label{conj:symm2}
The recurrence polynomial $J(z)$ of $v$ is palindromic. 
\end{conjecture}

{\Large\part{Tropical $T$-systems.}\label{part:tropical}}

\section{The behavior of tropical $T$-systems}
\subsection{Tropical $T$-systems: definition}\label{sec:behavior}
Each bipartite recurrent quiver $Q$ has the corresponding $T$-system which we will call \emph{the geometric $T$-system} associated with $Q$ in order to distinguish it from another system which we introduce in this section. We refer the reader to Example~\ref{example:evolution} for an illustration of most of the statements that we prove in Sections~\ref{sec:behavior}-\ref{sec:conservation}.

\begin{definition}
 Let $Q$ be a bipartite recurrent quiver, and let $\l:\VertQ\to\Z$ be any map. Then the \emph{tropical $T$-system} associated with $Q$ is a family of integers $\Ttr_v(t)\in\Z$ for every $v\in\VertQ,t\in\Z$ with $t+\e_v$ even satisfying the following relations:
 \begin{eqnarray*}
  \Ttr_v(t+1)+\Ttr_v(t-1)&=&\max\left(\sum_{u\to v}\Ttr_u(t),\sum_{v\to w}\Ttr_w(t)\right);\\
  \Ttr_v(\e_v)&=&\l(v).
 \end{eqnarray*}
\end{definition}
It is apparent from the definition that $\Ttr_v(t)$ is the \emph{tropicalization} of $T_v(t)$. One can define a tropical $T$-system with values in $\Q$ or $\R$, but for our purposes it is sufficient to consider only the integer-valued version (see also Remark~\ref{rmk:real_values}). The defining recurrence relation can be translated into the language of bigraphs as follows: if $G(\Gamma,\Delta)$ is a bipartite bigraph then the relation becomes
\[\Ttr_v(t+1)+\Ttr_v(t-1)=\max\left(\sum_{(u,v)\in\Gamma}\Ttr_u(t),\sum_{(v,w)\in\Delta}\Ttr_w(t)\right).\]

\def\maxdeg{{ \operatorname{deg}_{\max}}}
\def\mindeg{{ \operatorname{deg}_{\min}}}

%
%

If $P(\x)\in\Z[\x^{\pm1}]$ is a multivariate Laurent polynomial in variables $(x_v)_{v\in\VertQ}$ then define $P\mid_{\x=q^\l}\in\Z[q^{\pm1}]$ to be the univariate Laurent polynomial in $q$ obtained from $P$ by substituting $x_v=q^{\l(v)}$ for all $v\in\VertQ$. Further, define $\maxdeg(q,P\mid_{\x=q^\l})$ to be the maximal degree of $q$ in $P\mid_{\x=q^\l}$. The following claim gives a connection between the geometric and tropical $T$-systems:

\begin{proposition}[{see \cite[Lemma 6.3]{GP}}]\label{prop:maxdeg}
 For every $v\in\VertQ,t\in\Z$ with $t+\e_v$ even and any $\l:\VertQ\to\Z$, we have 
 \[\Ttr_v(t)=\maxdeg\left(q,T_v(t)\mid_{\x=q^\l}\right).\]
\end{proposition}

\subsection{Linear algebraic properties of the affine Coxeter transformation}
Let $\affL$ be a bipartite affine $ADE$ Dynkin diagram, and let $w$ and $b$ be the numbers of white and black vertices in $\affL$ respectively. One can view a map $\map:\VertL\to\Z$ as a vector $\col{\map_W}{\map_B}\in\Z^{w+b}$. Then the adjacency matrix $A_\affL$ of $\affL$ has the form
\[A_\affL=\begin{pmatrix}
           0 & \bwmat\\
           \bwmat^t & 0
          \end{pmatrix}\]
where $\bwmat$ is a $w\times b$ matrix and $ ^t$ denotes matrix transpose. Define the \emph{mutation matrices} 
\[\mutmat_W:=\begin{pmatrix}
       -I_w & \bwmat\\
       0 & I_b
      \end{pmatrix};\quad
\mutmat_B=:=\begin{pmatrix}
       I_w & 0\\
       \bwmat^t & -I_b
      \end{pmatrix}.\]
Here $I_k$ is the identity $k\times k$ matrix. Finally, the \emph{Coxeter transformation} for $\affL$ is defined as a product $\coxmat=\mutmat_B\mutmat_W$. By Lemma~\ref{lemma:eigenvalues}, the matrix $A_\affL$ has a dominant eigenvector $\eig=\col{\eig_W}{\eig_B}$ corresponding to the eigenvalue $2$. This means  \[\bwmat\eig_B=2\eig_W;\quad \bwmat^t\eig_W=2\eig_B.\]
Just as in Part~\ref{part:classification}, all the coordinates of $\eig$ are assumed to be positive integers with greatest common divisor equal to $1$. For a vector $\map=\col{\map_W}{\map_B}$, we define three linear functionals as follows:
\[\speed_W(\map):=\<\eig_B,\map_B\>-\<\eig_W,\map_W\>;\quad\speed_B(\map):=-\speed_W(\map);\]
\[\summ(\map):=\<\eig_W,\map_W\>+\<\eig_B,\map_B\>.\]
Here $\<\cdot,\cdot\>$ denotes the standard inner product in $\R^w$ and in $\R^b$. 

\begin{proposition}\label{prop:linear_algebra}
For any vector $\map=\col{\map_W}{\map_B}$, the following holds:
\begin{eqnarray*}
 \speed_B(\mutmat_W(\map))&=&\speed_W(\map);\\
 \speed_W(\mutmat_B(\map))&=&\speed_B(\map);\\
 \summ(\mutmat_W(\map))&=&\summ(\map)+2\speed_W(\map);\\
 \summ(\mutmat_B(\map))&=&\summ(\map)+2\speed_B(\map).
\end{eqnarray*} 
\end{proposition}
\begin{proof}
 We will only prove the equalities for $\mutmat_W$, and the argument is a pretty straightforward calculation:
 \begin{eqnarray*}
  \speed_B(\mutmat_W(\map))&=&\<\eig_W,A\map_B-\map_W\>-\<\eig_B,\map_B\>=\<A^t\eig_W,\map_B\>-\<\eig_W,\map_W\>-\<\eig_B,\map_B\>\\
  &=&\<\eig_B,\map_B\>-\<\eig_W,\map_W\>=\speed_W(\map);\\
  \summ(\mutmat_W(\map))&=&\<\eig_W,A\map_B-\map_W\>+\<\eig_B,\map_B\>=\<A^t\eig_W,\map_B\>-\<\eig_W,\map_W\>+\<\eig_B,\map_B\>\\
  &=&\<A^t\eig_W,\map_B\>-\<\eig_W,\map_W\>+\<\eig_B,\map_B\>= \summ(\map)+2\speed_W(\map).\\
 \end{eqnarray*}
\end{proof}

Proposition~\ref{prop:linear_algebra} says that $\speed$ is preserved while $\summ$ grows linearly as we mutate. It turns out that up to a shift by $\eig$, the mutation action is periodic:
\begin{proposition}\label{prop:almost_periodic}
 For any affine $ADE$ Dynkin diagram $\affL$, there exists an integer $h_a(\affL)$ called \emph{the affine Coxeter number} and an integer $\coeff(\affL)$ such that for any vector $\map=\col{\map_W}{\map_B}$, we have 
 \begin{equation}\label{eq:linear}
 \coxmat^{h_a(\affL)}\map=\map+\coeff(\affL)\speed_W(\map)\eig.  
 \end{equation}
 Moreover,
 \[\coeff(\affL)=\frac{4h_a(\affL)}{\<\eig,\eig\>},\]
 and the values of $h_a(\affL)$ and $\coeff(\affL)$ are given in Table~\ref{table:affine_cox_number}.
\end{proposition}
\begin{proof}
 Stekolshchik~\cite[Remark~4.3]{Stek} gives complete information on the Jordan normal form of $\coxmat$: all eigenvalues of $\coxmat$ are roots of unity and the greatest common divisor of their periods is $h_a(\affL)$. Moreover, all of them have multiplicity one except for one of them ($\lambda=1$) which has multiplicity $2$. In our notation, the eigenvector attached to eigenvalue $1$ is precisely $\eig$ and the adjoint vector is $\v':=\frac14\col{\eig_W}{-\eig_B}$ (see \cite[Proposition~3.10]{Stek}). We have $\coxmat \v'=\eig+\v'$, and they are orthogonal to each other and to all other eigenvectors. The result follows.
\end{proof}

\begin{table}
\centering
\begin{tabular}{|c|c|c|c|c|c|c|}\hline
 $\affL$       & $\affA_{2n-1}$  & $\affD_{2n}$  &$\affD_{2n+1}$  & $\affE_6$  & $\affE_7$ & $\affE_8$\\\hline
$h_a(\affL)$   & $n$             & $2n-2$        & $2(2n-1)$      & $6$        & $12$      & $30$\\\hline
$\coeff(\affL)$& $2$             & $1$           & $2$            & $1$        & $1$       & $1$\\\hline
\end{tabular}
\caption{\label{table:affine_cox_number}(see \cite[Table~4.1]{Stek}) Affine Coxeter numbers of affine $ADE$ Dynkin diagrams }
\end{table}

\begin{example}
 Let $\affL=\affD_4$. Then $h_a(\affL)=2$ and the dominant eigenvector is given by 
 \[\eig=\begin{smallmatrix}
         1& &1\\
          &2&  \\
         1& &1
        \end{smallmatrix}.\]
 Thus
 \[\coeff(\affL)=\frac{4\cdot 2}{1^2+1^2+1^2+1^2+2^2}=1.\]
 
 Consider the following vector of initial values:
 \[\map=\begin{smallmatrix}
         1& &1\\
          &2&  \\
         1& &2
        \end{smallmatrix}.\]
Let us assume that the vertex in the middle is white. Then $\speed_W(\map)=1+1+1+2-2\cdot 2=1$. Since $h_a(\affL)=2$, we need to calculate \[\coxmat^2(\map)=\mutmat_B\mutmat_W\mutmat_B\mutmat_W\map.\]
The sequence of vectors that we will get is:
\[\map=\begin{smallmatrix}
         1& &1\\
          &2&  \\
         1& &2
        \end{smallmatrix}\xrightarrow{\mutmat_W}
        \begin{smallmatrix}
         1& &1\\
          &3&  \\
         1& &2
        \end{smallmatrix}\xrightarrow{\mutmat_B}
        \begin{smallmatrix}
         2& &2\\
          &3&  \\
         2& &1
        \end{smallmatrix}\xrightarrow{\mutmat_W}
        \begin{smallmatrix}
         2& &2\\
          &4&  \\
         2& &1
        \end{smallmatrix}\xrightarrow{\mutmat_B}
        \begin{smallmatrix}
         2& &2\\
          &4&  \\
         2& &3
        \end{smallmatrix}=\map+\eig.\]
We indeed see that 
\[\coxmat^{h_a(\affL)}(\map)=\coxmat^2(\map)=\map+\eig=\map+\coeff(\affL)\speed(\map)\eig.\]
\end{example}

We would like to apply these observations to the tropical $T$-system $\Ttr_v(t)$ defined above. Let $G=(\Gamma,\Delta)$ be a bigraph, and let $\affL$ be a connected component of $\Gamma$ isomorphic to an affine $ADE$ Dynkin diagram. We define $\Ttr_\affL(t)$ to be the vector in $\R^w$ for $t$ even and in $\R^b$ for $t$ odd which sends $v\in\VertL$ to $\Ttr_v(t)$ when $t+\e_v$ is even. In particular, the vectors $\col{\Ttr_\affL(2t)}{\Ttr_\affL(2t+1)}$ and $\col{\Ttr_\affL(2t+2)}{\Ttr_\affL(2t+1)}$ belong to $\R^{w+b}$. Moreover, they satisfy the following inequality:
\[\col{\Ttr_\affL(2t+2)}{\Ttr_\affL(2t+1)}\geq \mutmat_W\col{\Ttr_\affL(2t)}{\Ttr_\affL(2t+1)}.\]
Here $\geq$ means that each coordinate of the vector on the left hand side is at least the corresponding coordinate of the vector on the right hand side. This inequality holds trivially by the definition of the tropical $T$-system. Moreover, it is an equality if and only if for every white vertex $v$ of $\affL$, we have
\begin{equation}\label{eq:inequality_tropical}
 \sum_{(u,v)\in\Gamma}\Ttr_u(2t+1)\geq \sum_{(v,w)\in\Delta}\Ttr_w(2t+1).
\end{equation}

Now, using the positivity of the coordinates of $\eig$ and Proposition~\ref{prop:linear_algebra}, we get that 
\begin{eqnarray*}
\speed_B\col{\Ttr_\affL(2t+2)}{\Ttr_\affL(2t+1)}&\geq& \speed_W\col{\Ttr_\affL(2t)}{\Ttr_\affL(2t+1)}; \\
\summ\col{\Ttr_\affL(2t+2)}{\Ttr_\affL(2t+1)}&\geq& \summ\col{\Ttr_\affL(2t)}{\Ttr_\affL(2t+1)}+2\speed_W\col{\Ttr_\affL(2t)}{\Ttr_\affL(2t+1)}.
\end{eqnarray*}
And again, each inequality becomes an equality if and only if (\ref{eq:inequality_tropical}) holds for every white vertex $v$ of $\affL$.

Define the following functions of $t$:
\begin{eqnarray*}
\speed_\affL(t)&:=&\begin{cases}
                    \speed_W\col{\Ttr_\affL(t)}{\Ttr_\affL(t+1)}, & \text{if $t$ is even};\\
                    \speed_B\col{\Ttr_\affL(t+1)}{\Ttr_\affL(t)}, & \text{if $t$ is odd};
                   \end{cases}\\
\summ_\affL(t)&:=&\begin{cases}
                    \summ\col{\Ttr_\affL(t)}{\Ttr_\affL(t+1)}, & \text{if $t$ is even};\\
                    \summ\col{\Ttr_\affL(t+1)}{\Ttr_\affL(t)}, & \text{if $t$ is odd}.
                   \end{cases} 
\end{eqnarray*}

We have thus shown the following:
\begin{proposition}\label{prop:speed}
 Let $G=(\Gamma,\Delta)$ be a bipartite bigraph, and let $\affL$ be a connected component of $\Gamma$ isomorphic to an affine $ADE$ Dynkin diagram. Then for every $t\in\Z$ we have 
 \[\speed_\affL(t+1)\geq\speed_\affL(t);\quad \summ_\affL(t+1)\geq \summ_\affL(t)+2\speed_\affL(t).\] 
 Moreover, either (\ref{eq:inequality_tropical}) holds or both inequalities are strict. \qed
 \end{proposition}

\section{Solitonic behavior: soliton resolution}\label{sec:resolution}
\label{subsec:implies_tropical}
It turns out that for Zamolodchikov integrable quivers, the tropical $T$-system behaves linearly for all but finitely many moments of time. Namely, let $Q$ be a bipartite recurrent quiver and assume $Q$ is Zamolodchikov integrable but not Zamolodchikov periodic. Let $G(Q)=(\Gamma,\Delta)$ be the corresponding bipartite bigraph. Then Corollary~\ref{cor:finite_or_affine} together with Remark~\ref{rmk:zperiod} imply that all connected components of $\Gamma$ are affine $ADE$ Dynkin diagrams. 
 
 The following proposition will be later illustrated by Example~\ref{example:evolution}.
\begin{proposition}\label{prop:tropical_linearization}
Assume that $Q$ is Zamolodchikov integrable and all connected components of $\Gamma$ are affine $ADE$ Dynkin diagrams as above. Then for every map $\l:\VertQ\to\Z$ there exists an integer $t_0$ such that for every $|t|>t_0$ and for every $v\in\VertQ$ with $t+\e_v$ even we have 
 \[\sum_{(u,v)\in\Gamma}\Ttr_u(t+1)\geq \sum_{(v,w)\in\Delta}\Ttr_w(t+1).\]
 In other words, for any initial data $\l$, the inequality~(\ref{eq:inequality_tropical}) is violated only finitely many times.
\end{proposition}
\begin{proof}
 If the inequality~(\ref{eq:inequality_tropical}) is violated infinitely many times, then there exists a connected component $\affL$ of $\Gamma$ such that $\speed_\affL(t)\to+\infty$ as $t\to +\infty$, because each time (\ref{eq:inequality_tropical}) is violated, $\speed_\affL(t)$ increases by at least $1$ (see Proposition~\ref{prop:speed}). In this case, again, by Proposition~\ref{prop:speed}, $\summ_\affL(t)$ grows superlinearly. By Proposition~\ref{prop:maxdeg}, $\summ_\affL(t)$ is just a linear combination of $\maxdeg(q,T_v(t)\mid_{\x=q^\l})$ for $v\in\affL$, and thus there is a vertex $v\in\affL$ for which $\maxdeg(q,T_v(t)\mid_{\x=q^\l})$ grows superlinearly. But the values of $T_v(t)$ satisfy a linear recurrence, and thus $\maxdeg(q,T_v(t)\mid_{\x=q^\l})$ cannot grow faster than linearly. 
\end{proof}
\begin{remark}\label{rmk:real_values}
 This proof works exactly the same way if the values of $\Ttr$ are assumed to lie in $\Q$ instead of $\Z$. We do not know whether the result of Proposition~\ref{prop:tropical_linearization} holds when the values of $\Ttr$ belong to $\R$.
\end{remark}

Now we are finally able to deduce Theorem~\ref{thm:implies}:

\begin{proof}[Proof of Theorem~\ref{thm:implies}]
By Corollary~\ref{cor:finite_or_affine}, Remark~\ref{rmk:zperiod} and Proposition~\ref{prop:affinite_subadditive}, we need to show that if all components of $\Gamma$ and of $\Delta$ are affine $ADE$ Dynkin diagrams then $Q$ cannot be recurrent. By Proposition~\ref{prop:tropical_linearization}, the inequality
 \[\sum_{(u,v)\in\Gamma}\Ttr_u(t+1)\geq \sum_{(v,w)\in\Delta}\Ttr_w(t+1)\]
 is violated finitely many times. By symmetry between $\Gamma$ and $\Delta$, the reverse inequality 
 \[\sum_{(u,v)\in\Gamma}\Ttr_u(t+1)\leq \sum_{(v,w)\in\Delta}\Ttr_w(t+1)\]
 is also violated finitely many times. Therefore after finitely many steps we will have 
 \begin{equation}\label{eq:equality}
  \sum_{(u,v)\in\Gamma}\Ttr_u(t+1)= \sum_{(v,w)\in\Delta}\Ttr_w(t+1)
   \end{equation}
 for all $v\in\VertQ$. To see that this is impossible, consider the following integers $\Ytr_v(t)$ defined for $t+\e_v$ even:
 \[\Ytr_v(t)=\sum_{(u,v)\in\Gamma}\Ttr_u(t+1) - \sum_{(v,w)\in\Delta}\Ttr_w(t+1).\]
 It is well-known that the numbers $\Ytr_v(t)$ give (up to a sign) a solution to the \emph{tropical $Y$-system} associated with $Q$, see, for example,~\cite{IIKKN1}. Since the mutations for the tropical $Y$-system are involutions as well, they are invertible, so we get a contradiction with (\ref{eq:equality}) because it states that for all initial data $\l$, the tropical $Y$-system $\Ytr_v(t)$ eventually becomes zero.
\end{proof}

Combining Proposition~\ref{prop:tropical_linearization} with Proposition~\ref{prop:almost_periodic}, we get the following corollary which we call ``soliton resolution'':
\begin{corollary}\label{cor:speeds}
 Let $Q$ be a Zamolodchikov integrable quiver and let $\affL$ be a component of $\Gamma$ isomorphic to an affine $ADE$ Dynkin diagram. Then for every map $\l:\VertQ\to\Z$, there exist integers $\speed^+_\affL(\l)$ and $\speed^-_\affL(\l)$ such that 
 \begin{itemize}
  \item for all $t\gg0$ and all $v\in\affL$ we have
  \begin{equation}\label{eq:speed1}
   \Ttr_v(t+2h_a(\affL))=\Ttr_v(t)+\coeff(\affL)\speed^+_\affL(\l)\v(v);
  \end{equation}

  \item for all $t\ll0$ and all $v\in\affL$ we have
  \begin{equation}\label{eq:speed2}
\Ttr_v(t-2h_a(\affL))=\Ttr_v(t)+\coeff(\affL)\speed^-_\affL(\l)\v(v);   
  \end{equation}
 \end{itemize}
 In other words, the values of $\Ttr_v$ grow linearly for $|t|\gg 0$.
\end{corollary}

For instance, the integers $\speed^+_\affL(\l)$ and $\speed^-_\affL(\l)$ are calculated in example~\ref{example:evolution}.

Let us explain the soliton terminology. Assume $Q$ is an \affinite $ADE$ bigraph and consider the associated tropical $T$-system $\Ttr$. Its restriction to each affine slice $\affL$ behaves independently of other slices when $|t|\gg 0$. We treat it as a particle (a \emph{1-soliton}). Then what happens is that when $t$ grows from $-\infty$, the particles move independently with constant speeds given by~(\ref{eq:speed2}). Then for small values of $t$ they start interacting with each other and eventually they again start moving independently with constant speeds given by~\ref{eq:speed1}). 
Such a phenomenon is commonly called \emph{soliton resolution}, see \cite{T}.

\begin{corollary}
 Let $Q$ be a Zamolodchikov integrable quiver and let $\affL$ be a component of $\Gamma$ isomorphic to an affine $ADE$ Dynkin diagram. Then for every map $\l:\VertQ\to\Z$, the following are equivalent:
 \begin{enumerate}
  \item\label{item:speed_plus_0} $\speed^+_\affL(\l)=0$;
  \item\label{item:speed_minus_0} $\speed^-_\affL(\l)=0$;
  \item\label{item:is_periodic} $\Ttr_v$ is a periodic sequence for every $v\in\VertQ$.
 \end{enumerate}
\end{corollary}
\begin{proof}
 It is obvious that (\ref{item:is_periodic}) implies (\ref{item:speed_plus_0}) and (\ref{item:speed_minus_0}). The fact that each of them implies (\ref{item:is_periodic}) follows from Corollary~\ref{cor:speeds}: if $\speed^+_\affL(\l)=0$ then $\Ttr_v$ is periodic for $t\gg0$, but then it is periodic for all $t$.
\end{proof}





\section{Solitonic behavior: speed conservation}\label{sec:conservation}

In this section we show that the speeds with which affine slices move get preserved after the scattering process is over, in the sense of Corollary~\ref{cor:speeds}. This can be viewed as a tropical version of Corollary \ref{cor:pal}.

Specifically, let $\affL_1,\dots,\affL_m$ be the $m$ affine slices of $A_m \otimes \affA_{2n-1}$, and for $r=1,2,\dots,m$ denote
\[\speed_r^+:=\speed_{\affL_r}^+(\l);\quad \speed_r^-:=\speed_{\affL_r}^-,\]
where $\l:\VertQ\to\Z$ is fixed throughout this section. 

\begin{thm} \label{thm:speed}
 For any $1 \leq r \leq m$ we have $$\speed_r^{+} = \speed_{m+1-r}^{-}.$$
\end{thm}

\def\sp{{\mathbf{S}}}

\begin{example}\label{example:evolution}
 Let us give an example of the kind of phenomenon in Theorem~\ref{thm:speed}. Let us say that $m=3$ and $n=2$, so our quiver $Q$ is $A_3\otimes \affA_1$ depicted in Figure\ref{fig:A3A1}.
 \begin{figure}
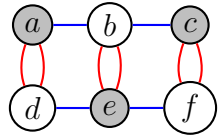


$
\psmatrix[colsep=0.4cm,rowsep=0.4cm,mnode=circle]
\bl{a}      & \wh{b}    & \bl{c}\\
\wh{d}      & \bl{e}    & \wh{f}
\psset{arrows=-,arrowscale=2}
\foreach \y in {1,2,3}
{
\ncarc[arcangle=20,linecolor=red]{-}{1,\y}{2,\y}
\ncarc[arcangle=-20,linecolor=red]{-}{1,\y}{2,\y}
}
\foreach \y in {1,2}
{
\ba{\y,1}{\y,2}
\ba{\y,2}{\y,3}
}
\endpsmatrix $
\caption{\label{fig:A3A1} The bigraph $A_3\otimes \affA_1$}
 \end{figure}
 
 \newcommand{\rr}[1]{\bf{\color{blue}#1}}
 
\newcommand{\sm}[6]{\scalebox{1}{$\begin{smallmatrix}
    #1&#2&#3\\#4&#5&#6
   \end{smallmatrix}$}}
\newcommand{\spp}[3]{\scalebox{0.9}{$\begin{smallmatrix}
    #1&#2&#3\end{smallmatrix}$}}
 We will compactly draw this quiver as \sm{a}{b}{c}{d}{e}{f}. Let $\affL_1,\affL_2,\affL_3$ be the three red connected components, and assume we start our mutation sequence with black vertices. Then we have
 \[\speed_{\affL_1}(t)=d-a;\quad\speed_{\affL_2}(t)=b-e;\quad\speed_{\affL_3}(t)=f-c.\]
 We denote $\sp(t)=(\speed_{\affL_1}(t),\speed_{\affL_2}(t),\speed_{\affL_3}(t))$. Now, for $t\ll0$, $\speed_{\affL_r}(t)=-\speed^-_r$ and for $t\gg0$, $\speed_{\affL_r}(t)=\speed^+_r$ for $r=1,2,3$. The mutations and speeds for initial values \sm{6}{6}{7}{3}{10}{5} are given in Table~\ref{table:evolution}. It is clear from the table that 
 \[\speed^-_1=\speed_3^+=3,\quad\speed^-_2=\speed_2^+=4,\quad\speed^-_3=\speed^+_1=2.\]
 This agrees with the statement of Theorem~\ref{thm:speed}. Next, it is also apparent from the table that the entries of $\sp(t)$ weakly increase, and each of them changes if and only if for at least one vertex in the corresponding connected component, the sum of blue neighbors is strictly larger than the sum of red neighbors. This is precisely the statement of Proposition~\ref{prop:speed}. Finally, observe that for every vertex $v\in\affL_r$, we have 
 \[\Ttr_v(4)=\Ttr_v(2)+2\speed^+_r,\]
 which is an application of Corollary~\ref{cor:speeds}. 
 \begin{table} 
  \begin{tabular}{|c|c|c|c|c|c|c|c|c|c|c|}\hline
  $t$&
  $-5$&
  $-4$&
  $-3$&
  $-2$&
  $-1$&
  $0$&
  $1$&
  $2$&
  $3$&
  $4$\\\hline
 \sm{a}{b}{c}{d}{e}{f}&
 \sm{6}{6}{7}{3}{10}{5}&
 \sm{0}{6}{3}{3}{2}{5}&
 \sm{0}{-2}{3}{\rr{-1}}{2}{1}&
 \sm{-2}{{-2}}{-1}{{-1}}{\rr{-2}}{1}&
 \sm{{-2}}{\rr{-1}}{{-1}}{\rr{-1}}{-2}{-3}&
 \sm{\rr1}{{-1}}{\rr0}{-1}{0}{-3}&
 \sm{1}{\rr2}{0}{3}{0}{3}&
 \sm{5}{2}{6}{3}{\rr6}{3}&
 \sm{5}{10}{6}{7}{6}{9}&
 \sm{9}{10}{12}{7}{14}{9}
 \\\hline
 $\sp(t)$&
 \spp{-3}{-4}{-2}&
 \spp{-3}{-4}{-2}&
 \spp{-1}{-4}{-2}&
 \spp{-1}{0}{-2}&
 \spp{1}{1}{-2}&
 \spp{2}{1}{3}&
 \spp{2}{2}{3}&
 \spp{2}{4}{3}&
 \spp{2}{4}{3}&
 \spp{2}{4}{3}\\\hline
   \end{tabular}
   \caption{\label{table:evolution}The evolution of the tropical $T$-system of type $A_3\otimes\affA_1$. The blue boldface numbers are the ones for which the sum of blue neighbors was strictly larger than the sum of red neighbors. }
 \end{table}
\end{example}

Let $$H_r^{\oplus} = \max_{\hh(\D) = r} \sum_{u \in \Cc_{m,2n}} (1 - d_{\D}(u))\l(u),$$
be the tropicalizations of Goncharov-Kenyon Hamiltonians $H_r$. 
Here $d_{\D}(u)$ is as before the degree of $u$ in the associated graph $G_{\D}$ on the cylinder, and the sum is taken over all tilings $\D$ of height $4r + \hh(\S)$.
Here on the boundary we always have $u=0$.

\begin{lem} \label{lem:tropH}
 The $H_r^{\oplus}$ are conserved quantities of the tropical $T$-system of type $A_m \otimes \affA_{2n-1}$.
\end{lem}

\begin{proof}
Follows from Proposition~\ref{prop:maxdeg} combined with the fact that $H_r$'s themselves are conserved quantities of the corresponding geometric $T$-system, see Corollary~\ref{cor:conserved}.
\end{proof}

Our strategy consists of proving the following proposition.

\begin{prop} \label{prop:tropH}
 Both $\speed_r^{+}$ and $\speed_{m+1-r}^{-}$ are equal to $H_r^{\oplus}$ for respectively $t\gg0$ and $t\ll0$.
\end{prop}

\begin{example}
 Let us continue Example~\ref{example:evolution}. Recall the formula for $H_1$ calculated in Example~\ref{example:ham}:
\[H_1 = \frac{ab}{de}+\frac{a}{be}+\frac{b}{ad}+\frac{c}{f}+\frac{d}{a}+\frac{ef}{bc}+\frac{e}{cf}.\]
 Thus,
 \[H_1^\oplus=\max(a+b-d-e,a-b-e,b-a-d,c-f,d-a,e+f-b-c,e-c-f).\]
 For instance, at $t=-4$ we have
 \[H_1^\oplus=\max(2,-11,-2,-3,2,-1,-9)=2.\]
 Or we can take $t=0$ instead and get
 \[H_1^\oplus=\max(0,1,2,2,1,-3,2)=2.\]
 We encourage the reader to check that for other moments of time, $H_1^\oplus$ is always equal to $2$, which is a statement of Lemma~\ref{lem:tropH}. In agreement with Proposition~\ref{prop:tropH}, we have
 \[H_1^\oplus=\speed_1^+=\speed_3^-.\]
\end{example}

Theorem \ref{thm:speed} follows trivially from Proposition~\ref{prop:tropH} and Lemma \ref{lem:tropH}, since as a conserved quantity $H_r^{\oplus}$ is the same at any point in time, including $t\gg0$ and $t\ll0$. Let us prove Proposition \ref{prop:tropH}. We are going to prove the $\speed_r^{+} = H_r^{\oplus}$ part,
the other part is essentially verbatim. Let us formulate several key lemmas. 

Consider the time $t\gg0$ large enough for all speeds to have stabilized. Adopt the convention $\speed_0^+ = \speed_{m+1}^+ = 0$. 

\begin{lem} \label{lem:wsub}
 The speeds form a weakly subadditive function, i.e. for any $1 < r < m$ we have 
 $$2 \speed_r^+ \geq \speed_{r-1}^+ + \speed_{r+1}^+.$$
\end{lem}

Informally, our strategy is to show that the maximum in the definition of $H_r^{\oplus}$ is achieved on the term equal to $\speed_r^{+}$. Note that we do not claim that this is the only term where the maximum is achieved, just that it is one of such terms. The following lemma is a major step.
We postpone its proof, and first show how to use it to imply Proposition \ref{prop:tropH}.

\begin{lem} \label{lem:max}
 The maximum in the expression $$\max_{\hh(\D) = r} \sum_{u \in \Cc_{m,2n}} (1 - d_{\D}(u))\l(u)$$ is achieved at one of the tilings $\D$ consisting entirely of horizontal dominos. 
\end{lem}

Recall that  we can compute $h(\D)$ by walking up from vertex $O$ to vertex $O'$ on the cylinder, collecting a contribution of $\pm 1$ or $\pm 3$ on each step. Let $\epsilon_i = +1$ if the $i$-th step along this path contributes a positive value and let $\epsilon_i = -1$ if it 
contributes a negative value. It is easy to see that $$\hh(\D) = \frac{m+1+\sum_{i=1}^{m+1} \epsilon_i}{2}.$$ If all dominos of $\D$ are horizontal, each layer of the cylinder $\Cc_{m,2n}$ has exactly two ways to be tiled, one contributing $\epsilon_i = +1$ and the other 
contributing $\epsilon_i = -1$.

\begin{lem}
 If all dominos in $\D$ are horizontal, we have $$\sum_{u \in \Cc_{m,2n}} (1 - d_{\D}(u))\l(u) = \frac{1}{2} \sum_{i=1}^{m+1} \epsilon_i \cdot \left(\speed_i^+ - \speed_{i-1}^+\right).$$
\end{lem}

\begin{proof}
 Depending on which of the two ways to tile the $i$-th layer of $\Cc_{m,2n}$ is used, this part of the tiling contributes into degrees $d_{\D}(u)$ for exactly half of vertices on each of the affine slices $i$ and $i-1$. Specifically, it either contributes to degrees of white $u$-s 
 on the $i$-th affine slice and degrees of black $u$-s for the $(i-1)$-st affine slice, or the other way around. The statement of the lemma is the numerical expression of this observation.  
\end{proof}

Now we are ready to prove Proposition \ref{prop:tropH}.

\begin{proof}
 By Lemma \ref{lem:max} we know that the maximum in the definition of $H_r^{\oplus}$ is achieved at one of the tilings $\D$ with all dominos horizontal. There are $2^{m+1}$ such tilings, corresponding to $2^{m+1}$ choices one can make for each $\epsilon_i$: either $\epsilon_i=1$ or 
 $\epsilon_i=-1$. Furthermore, the maximum is taken over $\D$-s with $\hh(\D) = r$, which means exactly $r$ among $\epsilon$-s are $+1$. 
 
 According to Lemma \ref{lem:wsub} we know that 
 $$\speed_1^+ = \speed_1^+ - \speed_{0}^+ \geq \speed_2^+ - \speed_{1}^+ \geq \ldots$$ $$\geq \speed_{m+1}^+ - \speed_{m}^+ = - \speed_{m}^+.$$
 
 Thus the maximum is obviously achieved when {\it {the first}} $r$ among $\epsilon$-s are equal to $+1$, and the rest of them are equal to $-1$. The terms cancel out resulting in 
 $$\frac{1}{2} \sum_{i=1}^{m+1} \epsilon_i \cdot \left(\speed_i^+ - \speed_{i-1}^+\right) = \speed_r^+.$$
 Thus, the maximal term in the expression for $H_r^{\oplus}$ is equal to $\speed_r^+$, as desired. 
 
\end{proof}

\subsection{Proof of Lemma \ref{lem:max}}

Consider a domino tiling $\D$ of the cylinder $\Cc_{m,2n}$ which has vertical dominos. Our strategy will be to construct a different tiling $\D'$, which has strictly less vertical dominos than $\D$, and such that  
$$\sum_{u \in \Cc_{m,2n}} (1 - d_{\D}(u))\l(u)  \;\; \leq \sum_{u \in \Cc_{m,2n}} (1 - d_{\D'}(u))\l(u).$$

Recall that there is a distinguished tiling $\S$ which we call the sea, such that superposition of $\S$ with any other tiling does not contain contractible closed cycles, except for possibly double dominos. Consider the double dimer $\D \cup \S$
obtained by taking superposition of our $\D$ and $\S$. Several hula hoops are formed. In fact, if $\D$ is contributing to $H_r^{\oplus}$, then $\hh(\D)=r$ and exactly $r$ hula hoops are formed. 

       \begin{figure}
    \begin{center}
\vspace{-.1in}
\scalebox{0.8}{
\input{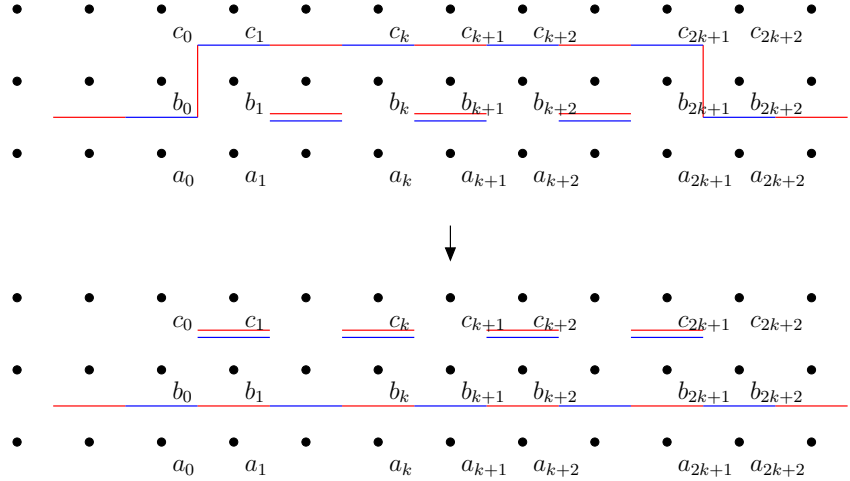} 
}
\vspace{-.1in}
    \end{center} 
    \caption{The superpositions $\S \cup \D$ (top) and $\S \cup \D'$ (bottom) at the local part where hula hoop $H$ is being straightened.}
    \label{fig:sadd17}
\end{figure}

Consider the lowest of the hula hoops $H$ which has vertical dominos. It exists because we assume $\D$ has some vertical dominos, and the part of $\D \cup \S$ not covered by hula hoops consists of horizontal double dominos. 
Choose one of the highest points in $H$ and consider the horizontal part of $H$ that contains this point together with two vertical dominos on its ends, see Figure \ref{fig:sadd17}. Note that $H$ may have several such parts, we just pick one of them. We create $\D'$ by straightening $H$ in this local spot, 
as shown at the bottom of Figure \ref{fig:sadd17}. It is clear that $\D'$ has strictly less vertical dominos than $\D$ does.

\begin{prop} \label{prop:flat}
 We have $\sum_{u \in \Cc_{m,2n}} (1 - d_{\D}(u))\l(u)  \leq \sum_{u \in \Cc_{m,2n}} (1 - d_{\D'}(u))\l(u).$.
\end{prop}

Denote the vertices surrounding this part of $H$ at time $t$ by $a_i$-s, $b_i$-s, and $c_i$-s, $0 \leq i \leq 2k+2$ as shown in Figure \ref{fig:sadd17}. Let us put $\mu(v):=\Ttr_v(t)$ for all $v\in\VertQ$.  We assume $t\gg0$ is sufficiently large for the claim of Corollary~\ref{cor:speeds} to hold. 
Then the time evolution of $a$-s depends only on the values of $\mu$ at $a$-s, etc. 
More formally, the following lemma holds,  describing the values of the tropical $T$-system at time $t+k$ for all $0\leq k\leq n-1$:

\begin{lem} \label{lem:time}
 Define $\e_k$ to be $0$ if $k$ is even and $1$ if $k$ is odd. Then we have 
 \begin{eqnarray*}
  \Ttr_{a_{k+1}}(t+k) &=& \sum_{i=1}^{k+1} \mu(a_{2i-1}) - \sum_{i=1}^{k} \mu(a_{2i}),\\
  \Ttr_{b_k}(t+k) &=& \sum_{i=0}^{k} \mu(b_{2i}) - \sum_{i=1}^{k} \mu(b_{2i-1}),\\
  \Ttr_{b_{k+2}}(t+k) &=& \sum_{i=1}^{k+1} \mu(b_{2i}) - \sum_{i=1}^{k} \mu(b_{2i+1}),\\
  \Ttr_{c_{k+1}}(t+k) &=& \sum_{i=1}^{k+1} \mu(c_{2i-1}) - \sum_{i=1}^{k} \mu(c_{2i}).
 \end{eqnarray*}  
\end{lem}

\begin{proof}
A straightforward application of recurrences $$\Ttr_{a_i}(t+j+1) = \Ttr_{a_{i-1}}(t+j) + \Ttr_{a_{i+1}}(t+j) - \Ttr_{a_i}(t+j-1), \text{  etc.}$$
which hold due to Proposition \ref{prop:tropical_linearization} and our choice of large enough initial time $t$.
\end{proof}

Now we are ready to prove Proposition \ref{prop:flat}. 

\begin{proof}
 Each edge of $G(\D)$ subtracts from the corresponding term of $H_r^{\oplus}$ two variables on its ends. Thus, we will compare those contributions for $G(\D)$ and $G(\D)$. We want to show that one is bigger than the other, which translates into 
 $$\mu(b_0) + \mu(b_1) + \sum_{i=1}^k (\mu(a_{2i}) + 2 \mu(b_{2i}) + \mu(c_{2i})) + \mu(b_{2k+1}) + \mu(b_{2k+2}) \geq \sum_{i=1}^k (\mu(a_{2i-1}) + 2 \mu(b_{2i-1}) + \mu(c_{2i-1})).$$
 By Lemma \ref{lem:time} this is easily seen to be equivalent to 
 $$\Ttr_{b_k}(t+k) + \Ttr_{b_{k+2}}(t+k) \geq \Ttr_{a_{k+1}}(t+k) + \Ttr_{c_{k+1}}(t+k),$$
 which holds by Proposition \ref{prop:tropical_linearization} and our choice of large enough $t$.
\end{proof}

\section{Conjectures}\label{sec:tropical_conjectures}

We conjecture that both soliton resolution and speed conservation properties hold for all families of our classification in Theorem \ref{thm:class}. 

For soliton resolution, we make the following conjecture, generalizing Proposition \ref{prop:tropical_linearization}. It can also be viewed as a tropical analog of Conjecture \ref{conj:real}. 

\begin{conjecture} \label{conj:sp2}
 For any quiver $Q$ in our  \affinite classification and any initial conditions either over $\mathbb Z$, or more generally over $\mathbb R$, there exists $t_0$ such that for $|t| > t_0$ the edges of finite component graph $\Delta$ do not affect the dynamics, i.e. for any vertex $v \in Q$ we have 
 \[\sum_{(u,v)\in\Gamma}\Ttr_u(t)\geq \sum_{(v,w)\in\Delta}\Ttr_w(t).\]
\end{conjecture}

In other words, for large enough time in both directions the affine slices of $Q$ evolve as separate particles. 

For speed conservation, we need to consider two cases, just as we did in Conjectures \ref{conj:symm1} and \ref{conj:symm2}.

Each of the types $A_m$, $D_m$ and $E_6$ has a canonical involution on the Dynkin diagram, sending the diagram to itself. As before, denote this involution $\eta$. Assume our tropical $T$-system is of the tensor product type, and more specifically of the form $\Lambda' \otimes \hat \Lambda$, where $\Lambda'$ 
is a finite type Dynkin diagram of type $A_m$, $D_{2m+1}$ or $E_6$, and $\hat \Lambda$ is an arbitrary extended Dynkin diagram.
Let $\hat \Lambda_v$ and $\hat \Lambda_{\eta(v)}$ be two affine slices of $\Lambda' \otimes \hat \Lambda$ such that their $\Lambda'$ coordinates are related by $\eta$. Let $\speed_{v}^{\pm}$ and $\speed_{\eta(v)}^{\pm}$ be the corresponding speeds for $t\gg0$ and $t\ll0$. Here we assume that 
Conjecture \ref{conj:sp2} holds and thus the speeds are well-defined. 
The following conjecture generalizes Theorem \ref{thm:speed}. 
It can also be viewed as a tropical analog of Conjecture \ref{conj:symm1}.

\begin{conjecture} \label{conj:symmm1}
We have $$\speed_v^{+} = \speed_{\eta(v)}^{-}.$$
\end{conjecture}

Assume now we are in {\it {any other case}}, i.e. either our tropical $T$-system belongs to a different family of the classification, or it is a tensor product but $\Lambda'$ is not of types $A_m$, $D_{2m+1}$ or $E_6$. 
Let $\hat \Lambda$ be any affine slice of the quiver, and let $\speed_{\hat \Lambda}^{\pm}$ be the corresponding speeds as  $t\gg0$ and $t\ll0$.
The following conjecture again generalizes Corollary \ref{cor:pal}.

\begin{conjecture}  \label{conj:symmm2}
We have $$\speed_{\hat \Lambda}^{+} = \speed_{\hat \Lambda}^{-}.$$
\end{conjecture}


\bibliographystyle{plain}
\bibliography{subadd}

\begin{thebibliography}{10}

\bibitem{ReuA}
Ibrahim Assem, Christophe Reutenauer, and David Smith.
\newblock Friezes.
\newblock {\em Advances in Mathematics}, 225(6):3134 -- 3165, 2010.

\bibitem{FZ3}
Arkady Berenstein, Sergey Fomin, and Andrei Zelevinsky.
\newblock Cluster algebras. {III}. {U}pper bounds and double {B}ruhat cells.
\newblock {\em Duke Math. J.}, 126(1):1--52, 2005.

\bibitem{Br}
Francesco Brenti.
\newblock Unimodal, log-concave and {P}\'olya frequency sequences in
  combinatorics.
\newblock {\em Mem. Amer. Math. Soc.}, 81(413):viii+106, 1989.

\bibitem{DK2}
Philippe Di~Francesco.
\newblock The solution of the {$A_r$} {$T$}-system for arbitrary boundary.
\newblock {\em Electron. J. Combin.}, 17(1):Research Paper 89, 43, 2010.

\bibitem{DK4}
Philippe Di~Francesco.
\newblock {$T$}-systems, networks and dimers.
\newblock {\em Comm. Math. Phys.}, 331(3):1237--1270, 2014.

\bibitem{DK1}
Philippe Di~Francesco and Rinat Kedem.
\newblock Positivity of the {$T$}-system cluster algebra.
\newblock {\em Electron. J. Combin.}, 16(1):Research Paper 140, 39, 2009.

\bibitem{DK3}
Philippe Di~Francesco and Rinat Kedem.
\newblock {$T$}-systems with boundaries from network solutions.
\newblock {\em Electron. J. Combin.}, 20(1):Paper 3, 62, 2013.

\bibitem{FZ}
Sergey Fomin and Andrei Zelevinsky.
\newblock Cluster algebras. {I}. {F}oundations.
\newblock {\em J. Amer. Math. Soc.}, 15(2):497--529 (electronic), 2002.

\bibitem{FZy}
Sergey Fomin and Andrei Zelevinsky.
\newblock {$Y$}-systems and generalized associahedra.
\newblock {\em Ann. of Math. (2)}, 158(3):977--1018, 2003.

\bibitem{FR}
Edward Frenkel and Nicolai Reshetikhin.
\newblock The {$q$}-characters of representations of quantum affine algebras
  and deformations of {$W$}-algebras.
\newblock In {\em Recent developments in quantum affine algebras and related
  topics ({R}aleigh, {NC}, 1998)}, volume 248 of {\em Contemp. Math.}, pages
  163--205. Amer. Math. Soc., Providence, RI, 1999.

\bibitem{FS}
Edward Frenkel and Andr{\'a}s Szenes.
\newblock Thermodynamic {B}ethe ansatz and dilogarithm identities. {I}.
\newblock {\em Math. Res. Lett.}, 2(6):677--693, 1995.

\bibitem{GP}
Pavel Galashin and Pavlo Pylyavskyy.
\newblock The classification of {Z}amolodchikov periodic quivers.
\newblock {\em arXiv:1603.03942}, 2016.

\bibitem{GT}
F.~Gliozzi and R.~Tateo.
\newblock Thermodynamic {B}ethe ansatz and three-fold triangulations.
\newblock {\em Internat. J. Modern Phys. A}, 11(22):4051--4064, 1996.

\bibitem{GK}
Alexander~B. Goncharov and Richard Kenyon.
\newblock Dimers and cluster integrable systems.
\newblock {\em Ann. Sci. \'Ec. Norm. Sup\'er. (4)}, 46(5):747--813, 2013.

\bibitem{Hen}
Andr{\'e} Henriques.
\newblock A periodicity theorem for the octahedron recurrence.
\newblock {\em J. Algebraic Combin.}, 26(1):1--26, 2007.

\bibitem{IIKKN1}
Rei Inoue, Osamu Iyama, Bernhard Keller, Atsuo Kuniba, and Tomoki Nakanishi.
\newblock Periodicities of {T}-systems and {Y}-systems, dilogarithm identities,
  and cluster algebras {I}: type {$B_r$}.
\newblock {\em Publ. Res. Inst. Math. Sci.}, 49(1):1--42, 2013.

\bibitem{IIKKN2}
Rei Inoue, Osamu Iyama, Bernhard Keller, Atsuo Kuniba, and Tomoki Nakanishi.
\newblock Periodicities of {T}-systems and {Y}-systems, dilogarithm identities,
  and cluster algebras {II}: types {$C_r$}, {$F_4$}, and {$G_2$}.
\newblock {\em Publ. Res. Inst. Math. Sci.}, 49(1):43--85, 2013.

\bibitem{KL}
David Kazhdan and George Lusztig.
\newblock Representations of {C}oxeter groups and {H}ecke algebras.
\newblock {\em Invent. Math.}, 53(2):165--184, 1979.

\bibitem{K}
Bernhard Keller.
\newblock The periodicity conjecture for pairs of {D}ynkin diagrams.
\newblock {\em Ann. of Math. (2)}, 177(1):111--170, 2013.

\bibitem{KS}
Bernhard Keller and Sarah Scherotzke.
\newblock Linear recurrence relations for cluster variables of affine quivers.
\newblock {\em Adv. Math.}, 228(3):1842--1862, 2011.

\bibitem{KR}
A.~N. Kirillov and N.~Yu. Reshetikhin.
\newblock Exact solution of the {$XXZ$} {H}eisenberg model of spin {$S$}.
\newblock {\em Zap. Nauchn. Sem. Leningrad. Otdel. Mat. Inst. Steklov. (LOMI)},
  145(Voprosy Kvant. Teor. Polya i Statist. Fiz. 5):109--133, 191, 195, 1985.

\bibitem{Kni}
Harold Knight.
\newblock Spectra of tensor products of finite-dimensional representations of
  {Y}angians.
\newblock {\em J. Algebra}, 174(1):187--196, 1995.

\bibitem{KN}
A.~Kuniba and T.~Nakanishi.
\newblock Spectra in conformal field theories from the {R}ogers dilogarithm.
\newblock {\em Modern Phys. Lett. A}, 7(37):3487--3494, 1992.

\bibitem{KNS}
Atsuo Kuniba, Tomoki Nakanishi, and Junji Suzuki.
\newblock Functional relations in solvable lattice models. {I}. {F}unctional
  relations and representation theory.
\newblock {\em Internat. J. Modern Phys. A}, 9(30):5215--5266, 1994.

\bibitem{KNSi}
Atsuo Kuniba, Tomoki Nakanishi, and Junji Suzuki.
\newblock {$T$}-systems and {$Y$}-systems in integrable systems.
\newblock {\em J. Phys. A}, 44(10):103001, 146, 2011.

\bibitem{N}
Hiraku Nakajima.
\newblock {$t$}-analogs of {$q$}-characters of {K}irillov-{R}eshetikhin modules
  of quantum affine algebras.
\newblock {\em Represent. Theory}, 7:259--274 (electronic), 2003.

\bibitem{Na}
Tomoki Nakanishi.
\newblock Periodicities in cluster algebras and dilogarithm identities.
\newblock In {\em Representations of algebras and related topics}, EMS Ser.
  Congr. Rep., pages 407--443. Eur. Math. Soc., Z\"urich, 2011.

\bibitem{OW}
E.~Ogievetsky and P.~Wiegmann.
\newblock Factorized {$S$}-matrix and the {B}ethe ansatz for simple {L}ie
  groups.
\newblock {\em Phys. Lett. B}, 168(4):360--366, 1986.

\bibitem{P}
Pavlo Pylyavskyy.
\newblock Zamolodchikov integrability via rings of invariants.
\newblock {\em arXiv:1506.05378}, 2015.

\bibitem{RVT}
F.~Ravanini, A.~Valleriani, and R.~Tateo.
\newblock Dynkin {TBA}s.
\newblock {\em Internat. J. Modern Phys. A}, 8(10):1707--1727, 1993.

\bibitem{R}
N.~Yu. Reshetikhin.
\newblock The spectrum of the transfer matrices connected with {K}ac-{M}oody
  algebras.
\newblock {\em Lett. Math. Phys.}, 14(3):235--246, 1987.

\bibitem{ShZ}
Paul Sherman and Andrei Zelevinsky.
\newblock Positivity and canonical bases in rank 2 cluster algebras of finite
  and affine types.
\newblock {\em Mosc. Math. J.}, 4(4):947--974, 982, 2004.

\bibitem{Sp}
David~E. Speyer.
\newblock Perfect matchings and the octahedron recurrence.
\newblock {\em J. Algebraic Combin.}, 25(3):309--348, 2007.

\bibitem{Stek}
R.~Stekolshchik.
\newblock {\em Notes on {C}oxeter transformations and the {M}c{K}ay
  correspondence}.
\newblock Springer Monographs in Mathematics. Springer-Verlag, Berlin, 2008.

\bibitem{S}
John~R. Stembridge.
\newblock Admissible {$W$}-graphs and commuting {C}artan matrices.
\newblock {\em Adv. in Appl. Math.}, 44(3):203--224, 2010.

\bibitem{Sz}
Andr{\'a}s Szenes.
\newblock Periodicity of {Y}-systems and flat connections.
\newblock {\em Lett. Math. Phys.}, 89(3):217--230, 2009.

\bibitem{T}
Terence Tao.
\newblock Why are solitons stable?
\newblock {\em Bull. Amer. Math. Soc. (N.S.)}, 46(1):1--33, 2009.

\bibitem{Th}
William~P. Thurston.
\newblock Conway's tiling groups.
\newblock {\em Amer. Math. Monthly}, 97(8):757--773, 1990.

\bibitem{V}
{\`E}.~B. Vinberg.
\newblock Discrete linear groups that are generated by reflections.
\newblock {\em Izv. Akad. Nauk SSSR Ser. Mat.}, 35:1072--1112, 1971.

\bibitem{Vo}
Alexandre~Yu. Volkov.
\newblock On the periodicity conjecture for {$Y$}-systems.
\newblock {\em Comm. Math. Phys.}, 276(2):509--517, 2007.

\bibitem{Z}
Al.~B. Zamolodchikov.
\newblock On the thermodynamic {B}ethe ansatz equations for reflectionless
  {$ADE$} scattering theories.
\newblock {\em Phys. Lett. B}, 253(3-4):391--394, 1991.

\end{thebibliography}

\end{document}